\documentclass[a4paper,10pt]{article}

\RequirePackage{enumerate}
\RequirePackage{verbatim}
\RequirePackage{hyperref}

\RequirePackage{amsmath,amsfonts,amssymb}
\RequirePackage[mathcal,mathbf]{euler}

\usepackage{xcolor}
\usepackage{latexsym,amscd,amsthm}
\usepackage[all]{xy}

\usepackage{fourier-orns}
\usepackage{amstext,amsopn}
\usepackage{graphicx}
\usepackage{epsfig}
\usepackage{psfrag}
\usepackage[all]{xy}
\usepackage{amscd}
\usepackage{amssymb}
\usepackage{fullpage}
\usepackage{url}

\usepackage{supertabular}

\usepackage{tikz}
\usetikzlibrary{arrows,decorations.markings}
\def\dar[#1]{\ar@<2pt>[#1]\ar@<-2pt>[#1]}
\entrymodifiers={!!<0pt,0.7ex>+}

\usepackage{epigraph}
\usepackage{stmaryrd}

\newtheorem{Thm}{Theorem}[section]
\newtheorem{Prop}[Thm]{Proposition}
\newtheorem{Lem}[Thm]{Lemma}
\newtheorem{Cor}[Thm]{Corollary}      
\newtheorem{Conj}[Thm]{Conjecture}

\newtheorem*{Theo}{Theorem}

\theoremstyle{remark}
\newtheorem{Rem}[Thm]{Remark}

\newtheorem*{Rema}{Remark}

\theoremstyle{definition}

\newtheorem{Def}[Thm]{Definition}
\newtheorem{Exa}[Thm]{Example}

\newcommand\kk{{\bold k}}

\newcommand\cdga{cdga}
\newcommand\CDGA{\bold{cdga}}
\renewcommand\mod{mod}
\newcommand\MOD{\bold{mod}}
\newcommand\CAT{\bold{cat}_\infty}

\newcommand\QCat{qcat}

\newcommand\lie{Lie}
\newcommand\LIE{\bold{Lie}}

\newcommand\SPA{\bold{spa}}
\newcommand\calg{Calg^{aug}}
\newcommand\calgnu{Calg^{nu}}
\newcommand\CALG{\bold{Calg}^{\bold{aug}}}
\newcommand\CALGNU{\bold{Calg}^{\bold{nu}}}
\newcommand\CALGSM{\bold{Calg}^{\bold{sm}}}
\newcommand\sset{{sSet}}
\newcommand\SSET{\bold{sSet}}
\newcommand\egr{{\epsilon-gr}}

\usepackage{marginnote}
\reversemarginpar

\newcommand{\adjointHA}{
\,\begin{matrix}\longrightarrow \\[-0.3cm] \longleftarrow\end{matrix}\,}

\author{Damien Calaque and Julien Grivaux}
\title{Formal moduli problems and formal derived stacks}
\date{}
\begin{document}

\maketitle
\thanks{\textit{To the memory of Jean-Louis Koszul}}
\epigraph{A mathematician, like a painter or a poet, is a maker of patterns. If his patterns are more permanent than theirs, it is because they are made with \textit{ideas}.}{G. H. Hardy -- \textit{A Mathematician's Apology}}

\begin{abstract}
This paper presents a survey on formal moduli problems. It starts with an introduction to pointed formal moduli problems and a sketch of proof of a Theorem (independently proven by Lurie and Pridham) which gives a precise mathematical formulation for Drinfeld's \textit{derived deformation theory} philosophy, which gives a correspondence between formal moduli problems and differential graded Lie algebras. The second part deals with Lurie's general theory of deformation contexts, which we present in a slightly different way than the original paper, emphasizing the (more symmetric) notion of Koszul duality contexts and morphisms thereof. In the third part, we explain how to apply this machinery to the case of non-split formal moduli problems under a given derived affine scheme; this situation has been dealt with recently by Joost Nuiten, and requires to replace differential graded Lie algebras with differential graded Lie algebroids. In the last part, we globalize this to the more general setting of formal thickenings of derived stacks, and suggest an alternative approach to results of Gaitsgory and Rozenblyum. 
\end{abstract}

\setcounter{tocdepth}{2}
\tableofcontents

\section*{Introduction}

The aim of these lecture notes is to present a recent work, due independently to Lurie and Pridham, concerning an equivalence of infinity-categories between formal moduli problems and differential graded Lie algebras. The link between deformation theory and dg-Lie algebras has a long history, which is as old as deformation theory itself. One of the first occurrence of a relation between these appears in Kodaira--Spencer's theory of deformations of complex compact manifolds (\textit{see} \cite{kodaira_complex_2005}): if $X$ is such a manifold, then infinitesimal deformations of $X$ over a local artinian $\mathbb{C}$-algebra $A$ with maximal ideal $m_A$ correspond to Maurer--Cartan elements of the dg-Lie algebra $\mathfrak{g}_X \otimes_{\mathbb{C}} m_A$ modulo gauge equivalence, where $\mathfrak{g}_X$ is the dg-Lie algebra $\Gamma(X, \mathscr{A}^{0, \bullet}_X \otimes \mathrm{T}_X)$, the differential being the $\overline{\partial}$ operator. This very concrete example illustrates the following general principle: to any sufficiently nice dg-Lie algebra it is possible to attach a deformation functor (\textit{see e.g.} \cite[\S 3.2]{Kontsevich}), which is defined on local artinian algebras as follows: 
\[
{Def}_{\mathfrak{g}} (A) = \mathrm{MC}(\mathfrak{g} \otimes m_A) / \, \textit{gauge equivalence}
\]
This correspondence was carried out by many people, including Quillen, Deligne and Drinfeld. In a letter to Schechtman \cite{DrDDT}, Drinfeld introduced in 1988 the \textit{Derived Deformation Theory (DDT)} philosophy: 
\begin{center}
\textit{Every (dg/derived) deformation problem is controlled by a dg-Lie (or $L_\infty$-) algebra.}
\end{center}
Since then, there has been a lot of work confirming this philosophy. We can refer the interested reader to the expository paper \cite{abramovich_differential_2009} for more details. Let us give now a few examples, which are related to derived algebraic geometry.

\par \medskip

In 1997 Kapranov \cite{Ka97-1} studied deformations of local systems: given a affine algebraic group $G$ and a $G$-local system $E$ on a manifold $S$, deformations of $E$ are encoded by a formal dg-scheme $\mathrm{RDef}(E)$ whose tangent complex $\mathbb{T}_{[E]}\mathrm{RDef}(E)$ at the closed point $[E]$ is $\mathrm{R}\Gamma(S, \mathrm{ad}(E)) [1]$. The next result is that $\mathbb{T}_{[E]}\mathrm{RDef}(E) [-1]$ is naturally an $L_\infty$-algebra, the corresponding Lie algebra structure on the cohomology groups $\mathrm{H}^i(S, \mathrm{ad}(E))$ being induced by the natural Lie structure on $\mathrm{ad}(E)$. 
\par \medskip
This last observation is not at all a hazard, and reflects a far more general phenomenon: in \textit{loc. cit.}, Kapranov proved\footnote{This idea is also present in Hinich's \cite{dgcu}.} that for any smooth dg-scheme $X$, the shifted tangent complex $\mathbb{T}_xX[-1]$ at a point $x :* \to X$ carries an $L_\infty$-algebra structure that determines the formal geometry of $X$ around $x$. A modern rephrasing of Kapranov's result can be presented as follows: for any derived affine scheme $X=Spec(A)$ with a $\kk$-point $x=Spec(A \to k)$, we have that 

\[
\mathbb{T}_xX[-1]  \simeq \mathbb{T}_{A} \otimes_{A} \kk [-1] \simeq  \mathbb{T}_{\kk / A} \simeq\mathrm{RDer}_A (\kk,\kk)
\]
and thus $\mathbb{T}_xX[-1]$ is a derived Lie algebra. 
\par \medskip
Let us explain another folklore calculation\footnote{We don't know a precise reference for that, but it seems to be known among experts. } confirming this result in another concrete example. Let $X=[G_0/G_1]$ be a $1$-stack presented by a smooth groupoid $G_1 \begin{matrix}  \overset{s}{\longrightarrow} \\[-0.2cm] \underset{t}{\longrightarrow} \end{matrix} \,G_0$, where  $G_0$ and $G_1$ are smooth affine algebraic varieties over a field $\kk$. If $x$ is a $\kk$-point of $G_0$, then the tangent complex $\mathbb{T}_{[x]}X$ of $X$ at the $\kk$-point $[x]$ is a 2-step complex $V \to W$ where
$V=T^s_{\mathrm{id}_x} G_1$ sits in degree $-1$, $W=T_xG_0$ sits in degree $0$, and the arrow is given by the differential of the target map $t$ at $\mathrm{id}_x$. 
\par \medskip
According to Kapranov's result, there should exist a natural $L_\infty$-structure on $\mathbb{T}_{[x]}X[-1]$. 
Let us sketch the construction of such a structure. 
Observe that $L:=(\mathrm{T}^s G_1)_{| G_0}$ has the structure of a $(\kk, A)$ Lie algebroid, where $G_0=Spec(A)$. In particular we have a $\kk$-linear Lie bracket $\Lambda^2_{\kk} L \to L$ and an anchor map $L\to TG_0=\mathrm{Der}_\kk(A)$. 
Very roughly, Taylor components at $x$ of the anchor give us maps $S^n(W)\otimes V \to W$, and Taylor components of the bracket at $x$ give us maps $S^n(W)\otimes\Lambda^2 (V) \to V$. 
These are the only possible structure maps for an $L_\infty$-algebra concentrated in degrees $0$ and $1$. 
Equations for the $L_\infty$-structure are guaranteed from the axioms of a Lie algebroid. 
\par \medskip
As a particular  but illuminating case, we can consider the stack $BG$ for some affine algebraic group $G$. Here $G_0=\{*\}$ and $G_1=G$. Then $\mathbb{T}_{*} BG[-1] \simeq \mathfrak{g}$ where $\mathfrak{g}$ is the Lie algebra of $G$, and the $L_{\infty}$-structure is simply the Lie structure of $\mathfrak{g}$.
\par \medskip
Let us explain how all this fits in the DDT philosophy. Any derived affine stack $Spec(A)$ endowed with a $\kk$-point $x$ defines a representable deformation functor $B \rightarrow \mathrm{Hom}(B, A)$, the $\mathrm{Hom}$ functor being taken in the category of $\kk$-augmented commutative differential graded algebras (cdgas). Hence it should be associated to a dg-Lie (or indifferently a $L_{\infty}$-) algebra. This dg-Lie algebra turns out to be  \textit{exactly} $\mathbb{T}_x \, Spec(A)[-1]$, endowed with Kapranov derived Lie structure. In this way, we get a clear picture of the DDT philosophy for representable deformation functors. However, the work of Lurie and Pridham goes way beyond that; what they prove is the following statement: 

\begin{Theo}[\cite{DAG-X,Pri}]
Over a base field of characteristic zero, there is an equivalence of $\infty$-categories between formal moduli problems and dg-Lie algebras. 
\end{Theo}
\noindent A very nice exposition of the above Theorem, with several examples and perspectives, is in To\"en's \cite{To}. 
\begin{Rema}
Lurie's work \cite{DAG-X} extends to $E_n$-deformation problems.
Pridham's work \cite{Pri} has some extension to the positive characteristic setting. 
In a forthcoming paper, Brantner and Mathew \cite{BrMa} actually generalize the above result over any field, proving that there is an equivalence of $\infty$-categories between formal moduli problems and so-called \textit{partition Lie algebras}. 
\end{Rema}
In these lectures, we will work in characteristic $0$, and our goal is to provide a version of Lurie--Pridham Theorem in families. In other words, we are aiming at first at a statement ``over a cdga $A$'', and build an extension from the affine case $X=Spec(A)$ to an arbitrary derived Artin stack $X$. One shall be very careful as there are different meanings to ``over a base''. We will provide two variants. The second one, which is the one we are interested in, will actually rather be named ``under a base''. 
\par \medskip
\fbox{\textit{Split families of formal moduli problems}}
\par \medskip
In the same year 1997, Kapranov \cite{Ka97-2} considered the family of \textit{all} formal neighborhood of points in a smooth algebraic variety $X$, which is nothing but the formal neighborhood $\hat{X\times X}$ of the diagonal in $X\times X$. He showed that the sheaf $T_X[-1]\simeq\mathbb{T}_{X/X\times X}$ is a Lie algebra object in $\mathrm{D}^{\mathrm{b}}(X)$, whose Chevalley--Eilenberg cdga gives back the structure sheaf of $\hat{X\times X}$. It is important to observe here that we have a kind of formal stack ($\hat{X\times X}$) that lives both over $X$ and under $X$. This example is a prototype for split families of formal moduli problems: a generalization of Lurie--Pridham Theorem has recently been proven by Benjamin Hennion \cite{He} in 2013, in the following form: 
\begin{Theo}[\cite{He}]
Let $A$ be a noetherian cdga concentrated in nonpositive degrees, and let $X$ be a derived Artin stack of finite presentation.
\begin{enumerate}
\item[--] $A$-pointed $A$-linear formal moduli problems are equivalent (as an $\infty$-category) to $A$-dg-Lie algebras. 
\item[--] $X$-pointed $X$-families of formal moduli problems are equivalent to Lie algebra objects in $QCoh(X)$.
\end{enumerate} 
\end{Theo}

\fbox{\textit{Formal moduli problems under a base}}
\par \medskip
In 2013, Caldararu, Tu and the first author \cite{CCT} looked at the formal neighborhood $\hat{Y}$ of a smooth closed subvariety $X$ into a smooth algebraic variety $Y$. They prove that the relative tangent complex $\mathbb{T}_{X/Y}$ is a dg-Lie algebroid whose Chevalley--Eilenberg cdga gives back the structure sheaf of $\hat{Y}$. Similar results had been previously proven by Bhargav Bhatt in a slightly different formulation (\textit{see} \cite{Bhatt}). We also refer to the work of Shilin Yu \cite{Yu}, who obtained parallel results in the complex analytic context. Note that this time, the derived scheme $\hat{Y}$ doesn't live anymore over $X$ (although it still lives \textit{under} $X$). Keeping this example in mind, we expect the following generalization of Hennion's result: 
\begin{enumerate}
\item[--] $A$-pointed $\kk$-linear formal moduli problems are equivalent to dg-Lie algebroids over $A$, which has recently been proven by Joost Nuiten in \cite{Nuit2}. This appears as Theorem \ref{thm-tooooop} of the present survey. 
\item[--] $X$-pointed formal moduli problems are equivalent to dg-Lie algebroids over $X$. A weaker version of this appears as Theorem \ref{thm-final} in the present paper. 
\end{enumerate}
\begin{Rema}
These results are somehow contained in the recent book \cite{GR1,GR2} of Gaitsgory--Rozenblyum, though in a slightly different formulation\footnote{Actually, Gaitsgory and Rozenblyum \textit{define} Lie algebroids as formal groupoids. Our results, as well as Nuiten's one, thus somehow indirectly give a formal exponentiation result for derived Lie algebroids. }. 
\end{Rema}
We finally observe that it would be very interesting to understand the results from \cite{Grigri}, describing the derived geometry of locally split first order thickenings $X\hookrightarrow S$ of a smooth scheme $X$, in terms of Lie algebroids on $X$. 

\subsection*{Description of the paper}

\begin{enumerate}
\item[\S 1] We present the basic notions necessary to understand Lurie--Pridham result, relate them to more classical constructions in deformation theory and provide some examples. We finally give a glimpse of Lurie's approach for the proof.
\item[\S 2] A general framework for abstract formal moduli problems is given in details. In \S 2.1, we recall Lurie's deformation contexts. In \S 2.2, we deal with dual deformation contexts. In \S 2.3, we introduce the useful notion of Koszul duality context, which is a nice interplay between a deformation context and a dual deformation context. In \S 2.4 we talk about morphisms between these, which is a rather delicate notion. In \S 2.5 we restate some results of Lurie using the notion of Koszul duality context, making his approach a bit more systematic. In \S 2.6 we discuss tangent complexes.  
\item[\S 3] We extend former results from dg-Lie algebras to dg-Lie algebroids. We prove Hennion's result \cite{He} in \S 3.1. Most of the material in \S 3.2, \S 3.3 and \S 3.4 happens to be already contained in the recent preprints \cite{Nuit1,Nuit2} of Joost Nuiten. \S 3.5 presents some kind of base change functor that will be useful for functoriality in the next Section. 
\item[\S 4] We explain how to globalize the results of \S 3. In \S 4.1 we introduce formal (pre)stacks and formal thickenings, mainly following \cite{GR1,GR2} and \cite{CPTVV}, and compare these with formal moduli problems. In \S 4.2 we state a consequence, for Lie algebroids, of the previous \S, and we sketch an alternative proof that is based on Koszul duality contexts and morphisms thereof. In \S 4.3 we show that formal thickenings of $X$ fully faithfully embed in Lie algebroids on $X$, and we conjecture that this actually is an equivalence. 
\end{enumerate}

\subsection*{Acknowledgements} 

The first author thanks Mathieu Anel, Benjamin Hennion, Pavel Safronov, and Bertrand To\"en for numerous discussions on this topic. 
This survey paper grew out of lecture notes, taken by the second author, of a 3 hours mini-course given by the first author at the session DAGIT of the \'Etats de la Recherche, that took place in Toulouse in June 2017. We both would like to thank the organizers for a wonderful meeting, as well as the participants to the lectures for their enthusiasm. 
\par \medskip
Damien Calaque acknowledges the financial support of the Institut Universitaire de France, and of the ANR grant ``SAT'' ANR-14-CE25-0008. Julien Grivaux acknowledges the financial support of the ANR grant ``MicroLocal'' ANR-15-CE40-0007, and ANR Grant ``HodgeFun'' ANR-16-CE40-0011. 

\subsection*{Credits}

\begin{itemize}
\item[--] \S1 actually doesn't contain more than the beginning of \cite{DAG-X}, and \S 2 is somehow a nice re-packaging of the general construction presented in \textit{loc.cit}. However, we believe the formalism of \S 2 is way more user-friendly for the reader wanting to apply \cite{DAG-X} in concrete situations.
\item[--] Even though we thought they were new at the time we gave these lectures, Joost Nuiten independently obtained results that contain and encompass the material that is covered in \S3 (\textit{see} the two very nice papers \cite{Nuit1,Nuit2}, that appeared while we were writing this survey). These results must be attributed to him. 
\item[--] Global results presented in \S 4 rely on the general theory of formal derived (pre)stacks from \cite{GR2} (see also \cite{CPTVV}). Some of these results seem to be new. 
\end{itemize}

\subsection*{Notation}
Below are the notation and conventions we use in this paper. 

\paragraph{Categories, model categories and $\infty$-categories}

\begin{enumerate}
\item[--] We make use all along of the language of $\infty$-categories. We will only use $(\infty,1)$-categories, that is categories where all $k$-morphisms for $k \geq 2$ are invertible. By $\infty$-categories, we will always mean $(\infty, 1)$-categories.
\item[--] Abusing notation, we will denote by the same letter an ordinary (i.e.~discrete) category and its 
associated $\infty$-category.
\item[--] If $M$ is a given model category, we write $\mathcal{W}_{M}$ for its subcategory of weak equivalences, 
and $\bold M:=M[\mathcal{W}_{M}^{-1}]$ for the associated $\infty$-category (unless otherwise specified, 
localization is always understood as the $\infty$-categorical localization, that is Dwyer-Kan simplicial localization). 
\item[--] Going from $M$ to $\bold M$ is harmless regarding (co)limits: homotopy (co)limits in $M$ correspond to $\infty$-categorical (co)limits in $\bold M$.
\item[--] Conversely, any presentable $\infty$-category can be strictified to a model category (\textit{i.e.} is of the form $\bold M$ for some model category $M$). In the whole paper, we will only deal with \textit{presentable} $\infty$-categories.
\item[--] Let $\CAT$ be the $\infty$-category of (small) $\infty$-categories. It can be obtained as the 
$\infty$-category associated with the model category $\QCat$ of (small) quasi-categories. 
\item[--] We denote by $\sset$ the model category of simplicial sets (endowed with Quillen's model structure). The associated $\infty$-category is denoted by $\SSET$, it is equivalent to $\infty Grpd$. We denote by $\SPA$ its stabilization (\textit{see} \S \ref{oups}), it is the category of spectra.
\item[--] When writing an adjunction horizontally (resp. vertically), we always write the left adjoint \textit{above} (resp. on the \textit{left}) and the right adjoint \textit{below} (resp. on the \textit{right}). This means that when we write an adjunction as $F:\bold C\,\begin{matrix}\longrightarrow \\[-0.3cm] \longleftarrow\end{matrix}\,\bold D:G$, $F$ is the left adjoint.
\end{enumerate}

\paragraph{Complexes}

\begin{enumerate}
\item[--] The letter $\kk$ will refer to a fixed field of characteristic zero. 
\item[--] The category $\mod_\kk$ is the model category of unbounded (cochain) complexes of $\kk$-modules from \cite{H} (it is known as the \textit{projective model structure}), and $\MOD_\kk$ is its associated $\infty$-category. For this model structure, weak equivalences are quasi-isomorphisms, and fibrations are componentwise surjective morphisms.
\item[--] The category $\mod_\kk^{\leq 0}$ is the model category of complexes of $\kk$-modules sitting in nonpositive degrees, and $\MOD_\kk^{\leq 0}$ is its associated $\infty$-category. For this model structure, fibrations are componentwise surjective morphisms in degree $\leq{-1}$.
\item[--] In the sequel, we will consider categories of complexes with an additional algebraic structure (like commutative differential graded algebras, or differential graded Lie algebras). They carry model structures for which fibrations and weak equivalences are exactly the same as the ones for complexes. It is such that the ``free-forget'' adjunction with $\mod_\kk$ (or, $\mod_\kk^{\leq0}$) is a Quillen adjunction. We refer the reader to the paper \cite{H} for more details. 
\end{enumerate}

\paragraph{Differential graded algebras}

\begin{enumerate}
\item[--] The category $\cdga_\kk$ denotes the model category of (unbounded) unital commutative differential graded $\kk$-algebras (that is commutative monoids in $\mod_\kk$), and $\CDGA_\kk$ denotes its associated $\infty$-category.
\item[--] A (unbounded) unital commutative differential graded $\kk$-algebra will be called a ``cdga''. 
\item[--] The category $\calg_\kk$ is the slice category $\cdga_{\kk/\kk}$, \textit{i.e.} the category of augmented $\kk$-algebras. It is equivalent to the category $\calgnu_\kk$ of non-unital commutative differential graded algebras. This equivalence is actually a Quillen equivalence. 
\item[--] The category $\cdga^{\leq 0}_\kk$ denotes the model category of cdgas over $\kk$ sitting in nonpositive degree.
The associated $\infty$-category is denoted by  $\CDGA^{\leq 0}_\kk$.
\item[--] For any $A\in\cdga_\kk$ we write $\mod_A$ for the model category of left $A$-modules, and $\MOD_A$ 
for its associated $\infty$-category. 
\item[--] Note that $\MOD_A$ is a \textit{stable} $\infty$-category. 
This can be deduced from the fact that $\mod_A$ is a triangulated dg category. The looping and delooping functors 
$\Omega_*$ and $\Sigma_*$ are then simply given by degree shifting $(*)[-1]$ and $(*)[1]$, respectively. 
\item[--] We define $\cdga_A$ is then the model category of $A$-algebras (commutative monoids in $\mod_A$), and $\CDGA_A$ is the 
corresponding $\infty$-category. 
\item[--] All these definitions have relative counterparts: if $A$ is a cdga and if $B$ is an object of $\mod_{A}$, then $\cdga_{A/B}$ is the slice category $(\cdga_A)_{/B}$ of relative $A$-algebras over $B$.
\end{enumerate}

\paragraph{Differential graded Lie algebras and $L_\infty$-algebras}

\begin{enumerate}
\item[--] We denote by $\lie_\kk$ the model category of (unbounded) diffential graded Lie algebras over $\kk$ (i.e. Lie algebra objects in $\mod_\kk$). The associated $\infty$-category is denoted by $\LIE_{\kk}$.
\item[--] A differential graded Lie algebra over $\kk$ will be called a ``dgla''. 
\item[--] Remark that $\LIE_\kk$ is equivalent to the localization of the category of $L_\infty$-algebras, with morphisms being $\infty$-morphisms, with respect to $\infty$-quasi-isomorphisms (\textit{see e.g.} \cite{Val}). 
\end{enumerate}


\section{Introduction to pointed formal moduli problems for commutative algebras}

\subsection{Small augmented algebras} \label{SAA}
For any dg-algebra $A$ and any $A$-module $M$, we can form the square zero extension of $A$ by $M$; we denote it by $A \oplus M$ where there is no possible ambiguity. We set first some crucial definitions for the rest of the paper: 
\begin{Def} Recall that $\CALG_\kk$ is the $\infty$-category of augmented cdgas.
\begin{itemize}
\item[--] A morphism in $\CALG_\kk$ is called \textit{elementary} if it is a pull-back\footnote{Since we are working in the $\infty$-categorical framework, pullback means homotopy pullback.} of $\kk \to \kk \oplus \kk[n]$ for some $n\geq1$, where $\kk \to \kk \oplus \kk[n]$ is the square zero extension of $\kk$ by $\kk[n]$.
\item[--] A morphism in $\CALG_\kk$ is called \textit{small} if it is a finite composition of elementary morphisms. 
\item[--] An object in $\CALG_\kk$ is called \textit{small} if the augmentation morphism $\epsilon \colon A \to \kk$ is small. 
\end{itemize}
\end{Def}
We denote by $\CALGSM_\kk$ the full\footnote{Hence we allow all morphisms in the category $\CALGSM_\kk$, not only small ones.} sub-$\infty$-category of small objects of $\CALG_\kk$. Small objects admits various concrete equivalent algebraic characterizations: 

\begin{Prop}[{\cite[Proposition 1.1.11 and Lemma 1.1.20]{DAG-X}}]\label{prop-referee}
An object $A$ of $\CALG_\kk$ is small if and only if the three following conditions hold: 
\begin{itemize}
\item[--] $H^n(A)=\{0\}$ for $n$ positive and for $n$ sufficiently negative. 
\item[--] All cohomology groups $H^n(A)$ are finite dimensional over $\kk$.
\item[--] $H^0(A)$ is a local ring with maximal ideal $\mathfrak{m}$, and the morphism $H^0(A)/\mathfrak{m} \to \kk$ is an isomorphism.
\end{itemize}
Moreover, a morphism $A \to B$ between small objects is small if and only if $H^0(A) \to H^0(B)$ is surjective. 
\end{Prop}
\begin{Rem} \label{bottin}
Let us make a few comments on this statement, in order to explain its meaning and link it to classical results in commutative algebra and deformation theory.
\begin{itemize}
\item[--] Observe that small algebras are nothing but dg-artinian algebras concentrated (cohomologically) in nonpositive degree.
\item[--] To get a practical grasp to the definitions of elementary and small morphisms, it is necessary to be able to compute homotopy pullbacks in the model category $\calg_\kk$. This is a tractable problem since the model structure on $\cdga_\kk$ is fairly explicit, and $\calg_\kk$ is a slice category of $\cdga_\kk$.
\item[--] If a small object $A$ is concentrated in degree zero, the theorem says that $A$ is small if and only if $A$ is a local artinian algebra with residue field $\kk$. Let us explain concretely why this holds (the argument is the same as in the general case). If $A$ is a local artinian algebra, then $A$ can be obtained from the residue field as a finite sequence of (classical) small extensions, that is extensions of the form
\[
0 \rightarrow (t) \rightarrow R_2 \rightarrow R_1 \rightarrow 0
\]
where $R_1$, $R_2$ are local artinian with residue field $\kk$, and $(t)$ is the ideal generated by a single element $t$ annihilated by the maximal ideal of $R_2$ (hence it is isomorphic to the residue field $\kk$). In this way we get a cartesian diagram\footnote{The cdga structure we put on the cone is the natural one: we require that $(t)$ has square zero.}
\[
\xymatrix{
R_2 \ar[r]  \ar[d] & \kk \ar[d] \\
\mathrm{cone} \left( (t) \to R_2 \right) \ar[r] & \kk \oplus \kk[1]
}
\]
in $\calg_\kk$. Since the bottom horizontal map is surjective in each degree, it is in particular a fibration. Therefore this diagram is also cartesian in $\CALG_\kk$, and is isomorphic to a cartesian diagram of the form
\[
\xymatrix{
R_2 \ar[r]  \ar[d] & \kk \ar[d] \\
R_1  \ar[r] & \kk \oplus \kk[1]
}
\]
in $\CALG_\kk$. Hence the morphism $R_2 \to R_1$ is elementary.
\item[--] An important part of classical deformation theory in algebraic geometry is devoted to \textit{formal} deformations of algebraic schemes. For a complete account, we refer the reader to the book \cite{Sernesi}. Following the beginning of \cite{DAG-X}, we will explain quickly how small morphisms fit in this framework. Given a algebraic scheme $Z$ over $\kk$ (that will be assumed to be smooth for simplicity), the formal deformation theory of $Z$ deals with equivalence classes of cartesian diagrams
\[
\xymatrix{
Z \ar[r] \ar[d] & \mathfrak{Z} \ar[d] \\
Spec(\kk) \ar[r] & Spec(A)
}
\]
where $A$ is a local artinian algebra with residue field $k$. This construction defines a deformation functor $\mathrm{Def}_Z$ from the category of local artinian algebras to sets.
\par \medskip 
It is possible to refine the functor $\mathrm{Def}_Z$ to a groupoid-valued functor that associates to any $A$ the groupoid of such diagrams. Since any groupoid defines a homotopy type (by taking the nerve), we get an enriched functor $\mathcal{D}ef_Z$ from local artinian algebras to the homotopy category of topological spaces, such that $\pi_0 (\mathcal{D}ef_Z)=\mathrm{Def}_Z$.
\par \medskip
The first important case happens when $A=\kk[t]/t^2$. In this case, Kodaira-Spencer theory gives a bijection between isomorphism classes of deformations of $X$ over $Spec (\kk[t]/t^2)$ and the cohomology group $\mathrm{H}^1(Z, \mathrm{T}_Z)$. In other words, $\mathrm{H}^1(Z, \mathrm{T}_Z)$ is the tangent space to the deformation functor $\mathrm{Def}_Z$. The next problem of the theory is the following: \textit{when can an infinitesimal deformation $\theta$ in $\mathrm{H}^1(Z, \mathrm{T}_Z)$ be lifted to $Spec(k[t]/t^3)$}? The answer is: exactly when $[\theta, \theta]$ vanishes in $\mathrm{H}^2(Z, \mathrm{T}_Z)$. This can be interpreted in the framework of derived algebraic geometry as follows: the extension
\[
0 \to (t^2) \to \kk[t]/t^3 \to \kk[t]/t^2 \to 0
\]
yields a cartesian diagram
\[
\xymatrix{
\kk[t]/t^3 \ar[r]  \ar[d] & \kk \ar[d] \\
\kk[t]/t^2 \ar[r] & \kk \oplus \kk[1]
}
\]
in $\CALG_\kk$. This gives a fiber sequence of homotopy types
\[
\mathcal{D}ef_Z (\kk[t]/t^3) \to \mathcal{D}ef_Z (\kk[t]/t^2)  \to  \mathcal{D}ef_Z (\kk \oplus \kk[1])
\]
hence a long exact sequence
\[
\cdots\rightarrow \pi_0(\mathcal{D}ef_Z (\kk[t]/t^3)) \rightarrow \pi_0 (\mathcal{D}ef_Z (\kk[t]/t^2)) \rightarrow \pi_0 (\mathcal{D}ef_Z (\kk \oplus \kk[1]).
\]
It turns out that the set $\pi_0 (\mathcal{D}ef_Z (\kk \oplus \kk[1])$ of equivalence classes of deformations of $Z$ over the derived scheme $Spec(\kk \oplus \kk[1])$ is isomorphic to $\mathrm{H}^2(Z, \mathrm{T}_Z)$. Hence the obstruction class morphism
\begin{eqnarray*}
\mathrm{Def}_Z(\kk[t]/t^2)\cong \mathrm{H}^1(Z, \mathrm{T}_Z) & \longrightarrow & \mathrm{H}^2(Z, \mathrm{T}_Z) \\
\theta & \longmapsto & [\theta, \theta]
\end{eqnarray*}
can be entirely understood by writing $\kk[t]/t^3 \to \kk[t]/t^2$ as an elementary morphism.
\end{itemize}
\end{Rem}

\subsection{The $\infty$-category of formal moduli problems}
We start by introducing formal moduli problems in the case of cdgas:
\begin{Def}
A \textit{formal moduli problem} (we write \textit{fmp}) is an $\infty$-functor $X: \CALGSM_{\kk} \to \SSET$ satisfying the following two properties: 
\begin{itemize}
\item[--]$X(\kk)$ is contractible.
\item[--] $X$ preserves pull-backs along small morphisms.
\end{itemize}
\end{Def}
The second condition means that given a cartesian diagram
\[
\xymatrix{
N \ar[r] \ar[d] & A \ar[d] \\
M \ar[r] & B
}
\]
in $\CALGSM_{\kk}$ where $A \to B$ is small, then 
\[
\xymatrix{
X(N) \ar[r] \ar[d] & X(A) \ar[d] \\
X(M) \ar[r] & X(B)
}
\]
is cartesian. 
\begin{Rem}
Observe that the second condition is stable under composition and pullback. Hence it is equivalent to replace in this condition small morphisms with elementary morphisms. We claim that we can even replace elementary morphisms with the particular morphisms $\kk \to \kk\oplus \kk[n]$ for every $n\geq 1$. Indeed, consider a cartesian diagram
\[
\xymatrix{
N \ar[d] \ar[r] & A \ar[d]^-{f}\\
M \ar[r] & B
}
\]
where $f$ is elementary, that is given by a cartesian diagram
\[
\xymatrix{
A \ar[d]_-{f} \ar[r] & \kk \ar[d] \\
B \ar[r] & \kk \oplus \kk[n]
}
\]
If we look at the diagram 
\[
\xymatrix{
X(N) \ar[r] \ar[d] & X(A) \ar[r] \ar[d] & X(*) \ar[d] \\
X(M) \ar[r] & X(B) \ar[r] & X(\kk \oplus \kk[n])
}
\]
and assume that $X$ preserves pullbacks along the morphisms $\kk \rightarrow \kk \oplus \kk[n]$, then the right square is cartesian, so the left square is cartesian if and only if the big square is cartesian (which is the case).
\end{Rem}
\begin{Cor}
A functor $X: \CALGSM_\kk \to \SSET$ is a fmp if and only if $X(k)$ is contractible and preserves pull-backs whenever morphisms in the diagram are surjective on $H^0$. 
\end{Cor}
\begin{proof}
Assume that $X$ preserves pull-backs whenever morphisms in the diagram are surjective on $H^0$. 
Consider a cartesian diagram of the following type, with $n\geq1$: 
\[
\xymatrix{
N \ar[r] \ar[d] & \kk \ar[d] \\
M \ar[r] & \kk \oplus \kk[n]
}
\]
Since $M$ is an augmented $\kk$-algebra, the map $M \to \kk \oplus \kk[n]$ is surjective on $H^0$. Hence 
\[
\xymatrix{
X(N) \ar[r] \ar[d] & X(\kk) \ar[d] \\
X(M) \ar[r] & X(\kk \oplus \kk[n])
}
\]
is cartesian. According to the preceding remark, this implies that $X$ is a fmp.

The reverse implication is obvious thanks to the last sentence of Proposition \ref{prop-referee}. 
\end{proof}

We write $\bold{FMP}_{\kk}$ for the full sub $\infty$-category of $Fun(\CALGSM_{\kk}, \SSET)$ consisting of formal moduli problems. 

\subsection{A glimpse at the description of $\bold{FMP}_{\kk}$}

In this section, we explain some heuristical aspects of the proof of the following theorem:
\begin{Thm}[\cite{DAG-X}, \cite{Pri}] \label{boss}
There is an equivalence of $\infty$-categories $\LIE_{\kk} \to \bold{FMP}_{\kk}$. 
\end{Thm}
\begin{Rem} Again, we discuss various points in this theorem related to more classical material.
\begin{itemize}
\item[--] Let us first give a naive idea about how the $\infty$-functor can be defined on a ``sufficiently nice'' dgla. The procedure is rather classical; we refer the reader to \cite{HiSc1} and \cite{HiSc2} for further details. If $\mathfrak{g}$ is a dgla of over $\kk$, we can consider its (discrete) Maurer--Cartan set
\[
\mathrm{MC}(\mathfrak{g})=\{ x \in \mathfrak{g}^1 \,\, \textrm{such that}\,\, dx+[x, x]_{\mathfrak{g}}=0 \}.
\]
The ``good'' object attached to $\mathfrak{g}$ is the quotient of $\mathrm{MC}(\mathfrak{g})$ under gauge equivalence. This can be better formulated using a simplicial enrichment as follows: for any $n \geq 0$, let $\Omega^{\bullet}(\Delta_n)$ be the cdga of polynomial differential forms on the $n$-simplex $\Delta_n$. The collection of the $\Omega^{\bullet}(\Delta_n)$ defines a simplicial cdga. Then we define the simplicial set $\mathscr{MC}(\mathfrak{g})$ as follows\footnote{The $\pi_0$ of $\mathscr{MC}(\mathfrak{g})$ is exactly the quotient of $\mathrm{MC}(\mathfrak{g})$ under gauge equivalence.}: 
\[
\mathscr{MC}(\mathfrak{g})_n=\mathrm{MC}(\mathfrak{g} \otimes \Omega^{\bullet}(\Delta_n)).
\]
This being done, we can attach to $\mathfrak{g}$ a deformation functor $\mathrm{Def}_{\mathfrak{g}} \colon \CALGSM \to \SSET$ defined by
\[
\mathrm{Def}_{\mathfrak{g}}(A)=\mathscr{MC}(\mathfrak{g} \otimes_{\kk} \mathrm{I}_A)
\]
where $I_A$ is the augmentation ideal of $A$. This defines (again in good cases) a formal moduli problem.
\item[--] \raisebox{0.07cm}{\danger} The main problem of the Maurer--Cartan construction is that the functor $\mathfrak{g} \rightarrow \mathrm{Def}_{\mathfrak{g}}$ does not always preserve weak equivalences. For dglas that satisfy some extra conditions (like for instance nipoltence conditions), $\mathrm{Def}_{\mathfrak{g}}$ will be exactly the fmp we are seeking for. We will see very soon how Lurie circumvents this problem using the Chevalley--Eilenberg complex.
\item[--] To illustrate an example where the Maurer--Cartan construction appears, let us come back to deformation theory of algebraic schemes in a slightly more differential-geometric context: instead of algebraic schemes we deform compact complex manifolds. We can attach to a complex compact manifold $Z$ the Dolbeault complex of the holomorphic tangent bundle $\mathrm{T}_Z$, which is the complex
\[
0 \rightarrow \mathscr{C}^{\infty}(\mathrm{T}_Z) \xrightarrow{\overline{\partial}} \mathscr{A}^{0, 1}(\mathrm{T}_Z) \xrightarrow{\overline{\partial}} \cdots
\]
We see this complex as a dgla over $\mathbb{C}$, the Lie structure being given by the classical Lie bracket of vector fields and the wedge product on forms. Then it is well known (\textit{see e.g.} \cite[Lemma 6.1.2]{Huybrechts}) that deformations of $Z$ over an artinian algebra $A$ yield $\mathrm{I}_A$-points in the set-valued deformation functor associated with the dgla $(\mathscr{A}^{0, \bullet}(\mathrm{T}_Z), \overline{\partial})$. 
\end{itemize}
\end{Rem}
Let us now explain the good construction of the equivalence from  $\LIE_{\kk}$ to $\bold{FMP}_{\kk}$. First we start by recalling the following standard definitions:
\begin{Def} For any dgla $\mathfrak{g}$, we define the homological and cohomological Chevalley--Eilenberg complexes $\mathrm{CE}_{\bullet}(\mathfrak{g})$ and $\mathrm{CE}^{\bullet}(\mathfrak{g})$ as follows: 
\begin{itemize}
\item[--] As a graded vector space, $\mathrm{CE}_{\bullet}(\mathfrak{g})=\mathrm{S}\,(\mathfrak{g}[1])$ is the (graded) symmetric algebra of $\mathfrak{g}[1]$. The differential is obtained by extending, as a degree $1$ graded coderivation, the sum ot the differential $\mathfrak{g}[1]\to\mathfrak{g}[2]$ with the Lie bracket $\mathrm{S}^2 (\mathfrak{g}[1])\to\mathfrak{g}[2]$. Jacobi identity and Leibniz rule ensure that this coderivation squares to zero. 
The complex $\mathrm{CE}_{\bullet}(\mathfrak{g})$ is actually an (coaugmented, counital, and conilpotent) cocommutative coalgebra object in the category of complexes.
\item[--] $\mathrm{CE}^{\bullet}(\mathfrak{g})$ is the linear dual of $\mathrm{CE}_{\bullet}(\mathfrak{g})$, it is an augmented cdga.
\end{itemize}
\end{Def}
\begin{Rem} \label{muraillette}
~
\begin{itemize}
\item[--] Observe that the above definition still makes sense for an $L_\infty$-algebra $\mathfrak{g}$. Indeed, an $L_\infty$-algebra structure on $\mathfrak{g}$ is defined as a degree $1$ graded codifferential that makes $\mathrm{S}\,(\mathfrak{g}[1])$ a coaugmented counital cocommutative differential graded coalgebra. 
\item[--] It is possible to prove that $\mathrm{CE}_{\bullet} (\mathfrak{g})\simeq \kk \overset{\mathbb{L}}{\otimes}_{\mathrm{U}(\mathfrak{g})} \kk$ and $\mathrm{CE}^{\bullet}(\mathfrak{g}) \simeq \mathrm{RHom}_{\mathrm{U}(\mathfrak{g})}(\kk, \kk)$. 
\item[--] Let $V$ be an object of $\mod_{\kk}$, and let $\textit{free}\, (V)$ be the free dgla generated by $V$. Then $\mathrm{CE}_{\bullet}(\textit{free}\, (V))$ and $\mathrm{CE}^{\bullet}(\textit{free}\, (V))$ are quasi-isomorphic to the square zero extensions $\kk \oplus V[1]$ and $\kk \oplus V^*[-1]$ respectively.
\item[--] 
For any cdga $A$ and any dgla $\mathfrak{g}$ of finite dimension over $\kk$, there is a map
\begin{equation} \label{plot}
\mathrm{Hom}_{\cdga_{\kk}}(\mathrm{CE}^{\bullet}(\mathfrak{g}), A) \rightarrow \mathrm{MC}(\mathfrak{g} \otimes I_A)
\end{equation}
which is in good cases an isomorphism (where $I_A$ is the augmentation ideal of $A$). To see how the map is constructed, let $\varphi$ be an element in $\mathrm{Hom}_{\cdga_{\kk}}(\mathrm{CE}^{\bullet}(\mathfrak{g}), A)$. Forgetting the differential, it defines an algebra morphism from the completed algebra $\widehat{\mathrm{S}} \,(\mathfrak{g}^*[-1])$ to $A$, and in fact to $I_A$ (since the morphism is pointed).
In particular, we have a map $\phi \colon \mathfrak{g}^*[-1] \rightarrow I_A$, and since $\mathfrak{g}$ is finite-dimensional, we can see the morphism $\phi$ as a map $\kk[-1] \rightarrow \mathfrak{g} \otimes I_A$, hence an element $x$ of $(\mathfrak{g} \otimes I_A)^{1}$. Now we have a commutative diagram
\[
\xymatrix{
\mathfrak{g}^*[-1] \ar[r]^-{\phi} \ar[d]^-{[\,.\, , \, . \,]^*} & A \ar[r]^-{d_A} &A \\
\mathfrak{g}^*[-1] \otimes_{\kk} \mathfrak{g}^*[-1] \ar[rr]^-{\phi \otimes \phi} && A \otimes_{\kk} A \ar[u]^-{\star_A}
}
\]
Unwrapping what it means expressing $\phi$ with $x$, we end up exactly with the Maurer--Cartan equation $dx+[x, x]=0$. The reason why the map \eqref{plot} is not always bijective is that $\mathrm{CE}^{\bullet}(\mathfrak{g})$ is not the symmetric algebra of $\mathfrak{g}^*[-1]$, but it is the \textit{completed} symmetric algebra. 
\item[--] There is a simplicially enrichement of \eqref{plot}, given by 
\[
\mathrm{Hom}_{\CDGA_{\kk}}(\mathrm{CE}^{\bullet}(\mathfrak{g}), A) \rightarrow \mathscr{MC}(\mathfrak{g} \otimes I_A). 
\]
Apart from the completion issue that we have already discussed, this morphism may not be an equivalence as $\mathrm{CE}^{\bullet}(\mathfrak{g})$ may not be a cofibrant object in $\cdga_{\kk}$.
\end{itemize}
\end{Rem}
\begin{proof}[Sketch of the construction of the equivalence]
~
\begin{itemize}
\item[--] The Chevalley--Eilenberg construction preserves weak equivalences, hence defining an $\infty$-functor
\[
\mathrm{CE}^{\bullet} \,\colon \LIE_{\kk}^{op} \rightarrow \CALG_{\kk}. 
\]
The functor $(\mathrm{CE}^{\bullet})^{op}$ commutes with small colimits (\textit{see} \cite[Proposition 2.2.17]{DAG-X}), so since $\LIE_{\kk}$ is presentable, $\mathrm{CE}^{\bullet}$ admits a left adjoint. We call this adjoint $\mathfrak{D}$. Hence we have an adjunction
\[
\mathfrak{D} \colon \CALG_{\kk} \adjointHA \LIE_{\kk}^{op} \colon \mathrm{CE}^{\bullet}
\]
that can be seen as some version of Koszul duality. The main point in this step is that the Chevalley--Eilenberg functor does only commute with small \textit{homotopy} limits, not usual small limits. Hence the adjoint functor is only defined in the $\infty$-categorical setting (\textit{i.e.} it does not come from a Quillen adjunction).
\item[--] 
We define an $\infty$-functor from $\LIE_{\kk}$ to $\mathbf{Fun}(\CALG_{\kk}, \SSET)$ as follows:
\[
\Delta (\mathfrak{g})= \mathrm{Hom}_{\LIE_{\kk}^{op}}\big(\mathfrak{g}, \mathfrak{D}(-)\big)=\mathrm{Hom}_{\LIE_{\kk}}\big(\mathfrak{D}(-), \mathfrak{g}\big).
\]
The functor $\Delta$ will define the equivalence we are seeking for.
\item[--] Let us explain why $\Delta$ factors through $\bold{FMP}_{\kk}$. We introduce the notion of \textit{good} object: a dg-Lie algebra $\mathcal{L}$ is good if there exists a finite chain
$0=\mathcal{L}_0 \rightarrow \mathcal{L}_1 \rightarrow \ldots \rightarrow \mathcal{L}_n=\mathcal{L}$ such that each of these morphisms appears in a pushout diagram
\[
\xymatrix{\textit{free}\, \kk[-n_i-1] \ar[r]\ar[d] & \mathcal{L}_i \ar[d] \\
\{0\} \ar[r] & \mathcal{L}_{i+1}
}
\]
in $\LIE_{\kk}$, or equivalently a pullback diagram
\[
\xymatrix{ \mathcal{L}_{i+1} \ar[r]\ar[d] & \{0\} \ar[d] \\
\mathcal{L}_i \ar[r] &\textit{free}\,\, \kk[-n_i-1] 
}
\]
in $\LIE_{\kk}^{op}$. We denote by $\LIE_{\kk}^{gd}$ the full subcategory of $\LIE_{\kk}^{op}$ consisting of good objects. We see that good objects are formally the same as small ones in $\LIE_{\kk}^{op}$, using the sequence of objects $\textit{free}\, \kk[-n-1]$ instead of $\kk \oplus \kk[n]$. This will be formalized using the various notions of deformation contexts developed in the next section.
\item[--] The next step consists of proving that if $\mathfrak{g}$ is good, the counit morphism $\mathfrak{D} \mathrm{CE}^{\bullet}(\mathfrak{g}) \rightarrow \mathfrak{g}$ in $\LIE_{\kk}^{op}$ is an equivalence. This is the crucial technical input and will be proved in Proposition \ref{opinel}. By purely formal arguments (\textit{see} Proposition \ref{yolo}), this implies that the adjunction $\mathfrak{D} \adjointHA \mathrm{CE}^{\bullet}$ defines an equivalence of categories between $\CALGSM_\kk$ and $\LIE_{\kk}^{gd}$.
\item[--] Using this, if we have a cartesian diagram
\[
\xymatrix{
N \ar[r] \ar[d] & \kk \ar[d] \\
M \ar[r] & \kk \oplus \kk[n]
}
\]
where $N$ and $M$ are small, then 
\[
\xymatrix{
\mathfrak{D}(N) \ar[r] \ar[d] & \{0\} \ar[d] \\
\mathfrak{D}(M) \ar[r] & \mathfrak{D}(\kk \oplus \kk[n])
}
\]
is cartesian in $\LIE_{\kk}^{op}$, and therefore
\[
\xymatrix{
\Delta(\mathfrak g)(N) \ar[r] \ar[d] & {*} \ar[d] \\
\Delta(\mathfrak g)(M) \ar[r] & \Delta(\mathfrak g)(\kk \oplus \kk[n])
}
\]
is also cartesian in $\SSET$. This implies that $\Delta$ is an object of $\bold{FMP}_{\kk}$. Hence $\Delta$ factors through the category $\bold{FMP}_{\kk}$.
\end{itemize}
\end{proof}

\section{General formal moduli problems and Koszul duality}

\subsection{Deformation contexts and small objects} \label{oups}
In this section, we will explain how the notions of small and elementary morphism make sense in a broader categorical setting. We start by some general facts on $\infty$-categories.
\begin{itemize}
\item[--] If $\bold C$ is an $\infty$-category with finite limits\footnote{This hypothesis will be always implicit in the sequel. It implies among other things that $\bold C$ has a terminal object $*$.}, its \textit{stabilization} $\mathrm{Stab}(\bold C)$ can be described as the $\infty$-category of spectrum objects (also called infinite loop objects) in $\bold C$. An object of $\mathrm{Stab}(\bold C)$ is a sequence $E=(E_n)_{n \in \mathbb{Z}}$ of pointed objects\footnote{The $\infty$-category of pointed objects is the coslice $\infty$-category $ ^{*/}\!\bold C$ under the terminal object.} together with weak equivalences $E_n \to \Omega E_{n+1}$, where $\Omega:=\Omega_*$ denotes the based loop functor: $\Omega(c)=*\underset{c}{\times}*$. We often write $E_n=\Omega^{\infty-n}E$.
\item[--] If $\bold C$ is stable, then $\mathrm{Stab}(\bold C)$ is naturally equivalent to $C$ via the map sending $(E_n)_n$ to $E_0$.
\item[--] If $\bold C$ is an $\infty$-category with finite limits and $c$ is an object of $\bold C$, then the stabilization $\mathrm{Stab}(^{c/}\!\bold C)$ of its coslice $\infty$-category $ ^{c/}\!\bold C$ is equivalent to $\mathrm{Stab}(\bold C)$. 
Indeed, the $\infty$-categories of pointed objects in $\bold C$ and $ ^{c/}\!\bold C$ are themselves equivalent. 
\item[--] If $\bold C$ is an $\infty$-category with finite limits and $c$ is an object of $\bold C$, then the stabilization $\mathrm{Stab}(\bold C_{/c})$ of its slice $\infty$-category $\bold C_{/c}$ is the category of spectrum objects in the $\infty$-category $ ^{id_c/}\!(\bold C_{/c})$ of sequences $c\to d\to c$ such that the composition is the identity. 
\item[--] If $\mathbf{C}=\SSET$, then $\mathrm{Stab}(\mathbf{C})$ is the $\infty$-category $\SPA$ of spectra (that is spectrum objects in spaces).
\end{itemize}
\begin{Def}
A pair $(\bold C,E)$, where $\bold C$ is a presentable $\infty$-category with finite limits and $E$ is an object of $\mathrm{Stab}(\bold C)$, is called a \textit{deformation context}. 
Given a deformation context $(\bold C,E)$: 
\begin{itemize}
\item[--] A morphism in $\bold C$ is \textit{elementary} if it is a pull-back of $* \to \Omega^{\infty-n}E$ for $n\geq 1$ (where $*$ is a terminal object in the category $C$). 
\item[--] A morphism in $\bold C$ is \textit{small} if it can be written as a finite sequence of elementary morphisms. 
\item[--] An object $c$ is \textit{small} if the morphism $c \to *$ is small. 
\end{itemize}
We let $(\bold C,E)^{sm}$ be the full subcategory of $\bold C$ spanned by the small objects. When it is clear from the context, we may abuse notation and write $\bold C^{sm}:=(\bold C,E)^{sm}$. 
\end{Def}
Let us give two examples of deformation contexts:
\begin{Exa}\label{examod}
If $\bold C=\MOD_\kk$, which is already stable, then we have an equivalence 
\begin{eqnarray*}
\MOD_{\kk} & \tilde\longrightarrow & \mathrm{Stab}(\bold C) \\
M & \longmapsto & (M[n])_n .
\end{eqnarray*}
In this context we will mainly consider the spectrum object $E=(\kk[n+1])_{n\in\mathbb{Z}}$.
\end{Exa}
\begin{Rem}\label{remmod}
Instead of working over the ground field $\kk$, we can work over an arbitrary cdga $A$. 
Then we can take $\bold C=\MOD_A$ and $E=(A[n+1])_{n\in\mathbb{Z}}$ in $\mathrm{Stab}(\bold C)$.
\end{Rem}
\begin{Exa} \label{gioberney}
The category $\mathrm{Stab}(\CALG_\kk)$ is equivalent to $\MOD_\kk$. 
Indeed, we have an equivalence $\CALG_\kk\simeq \CALGNU_\kk$, and the based loop functor $\Omega_0$ in $\CALGNU_\kk$ ($0$ is the terminal nonunital cdga) sends a non-unital algebra $R$ to $R[-1]$ equipped with the trivial product; hence the equivalence 
\begin{eqnarray*}
\MOD_\kk & \tilde\longrightarrow & \mathrm{Stab}(\CALG_\kk) \\
M & \longmapsto & (\kk \oplus M[n])_{n\in\mathbb{Z}} .
\end{eqnarray*}
In this context we will mainly consider the spectrum object $E=(\kk\oplus\kk[n])_{n\in\mathbb{Z}}$. 
\end{Exa}
\begin{Rem}\label{remalgA}
One can prove in a similar way that the stabilization $\mathrm{Stab}(\CDGA_{A/A})$ of the $\infty$-category $\CDGA_{A/A}$ of $A$-augmented $A$-algebras, where $A$ is in $\CDGA_\kk$, is equivalent to $\MOD_A$. 
Moreover, if $\CDGA_{\kk/A}$ is the $\infty$-category of $A$-augmented $\kk$-algebras, then 
\[
\mathrm{Stab}\,(\CDGA_{\kk/A}) \simeq \mathrm{Stab}\,(\CDGA_{A/A})\simeq \MOD_A.
\]
In this case a natural spectrum object to consider is $(A\oplus A[n])_{n\in\mathbb{Z}}$. 
\end{Rem}

We have the following obvious, though very useful, lemma, which allows to transfer deformation contexts along adjunctions: 
\begin{Lem}
Suppose $(\bold C, E)$ is a deformation context. If we are given an adjunction $T' \colon \bold C' \adjointHA \bold C \colon T$,, then $T$ preserves small limits, so that $(\bold C', T(E))$ is a deformation context whenever $\bold C'$ is presentable. 
Besides, in this case $T$ induces a functor from $\bold C^{sm}$ to $(\bold C')^{sm}$. $\Box$
\end{Lem}

Let us give four examples of this transfer principle\footnote{In all these examples, though, the use of transfer is not strictly necessary. Indeed, all the categories involved have a stabilization that is equivalent to $\MOD_A$.}
\begin{Exa}\label{exacoslice}
Given an $A$-module $L$, the push-out functor $-\coprod_L0:~^{L/}\!\MOD_A\to\MOD_A$ along the zero map $L\to 0$ admits a right adjoint, being the functor sending an $A$-module $M$ to the zero map $L\overset{0}{\to}M$. 
Hence the deformation context from Example \ref{examod} and Remark \ref{remmod} can be transferred to $ ^{L/}\!\MOD_A$. 
\end{Exa}

\begin{Exa}\label{exa-AA}
The relative cotangent complex functor $\mathbb{L}_{A/-}:\CALG_A\to\MOD_A$ admits a right adjoint, being the split square zero extension functor $M\mapsto A\oplus M$. 
Hence the deformation context from Example \ref{gioberney} and Remark \ref{remalgA} can be obtained by transfer from the one given in Example \ref{examod} and Remark \ref{remmod}. 
\end{Exa}

\begin{Exa}\label{exa-kA}
The relative cotangent complex functor $\mathbb{L}_{A/-}:\CDGA_{\kk/A}\to ~^{{\mathbb{L}_{A/\kk}}_/}\!\MOD_A$ admits a right adjoint, being the functor sending a morphism $\mathbb{L}_{A/\kk}\overset{d}{\to}M$ in $\MOD_A$ to the non-necessarily split square zero extension $A\underset{d}{\oplus}M[-1]$ that it classifies. 
Hence the deformation context from Example \ref{exacoslice} (with $L=\mathbb{L}_{A/\kk}$) can be transferred to $\CDGA_{\kk/A}$. 
\end{Exa}

\begin{Exa}
The forgetful functor $\CALG_A\to \CDGA_{\kk/A}$ is a right adjoint: its left adjoint is $A\otimes-$. 
Hence the deformation context from Example \ref{exa-kA} can also be obtained by transfer from the one given in Example \ref{exa-AA}. 
\end{Exa}

The main observation is that formal moduli problems make sense with $\CALG_{\kk}$ being replaced by any deformation context $(\bold C,E)$. We write $\bold{FMP}(\bold C,E)$ for the $\infty$-category of formal moduli problems associated with it. Then $\bold{FMP}\big(\CALG_{\kk}, (\kk \oplus \kk[n])_n\big)=\bold{FMP}_{\kk}$. 

\subsection{Dual deformation contexts and good objects}
We introduce the dual concept of a deformation context:
\begin{Def}
A pair $(\bold D,F)$, where $\bold D$ is a presentable $\infty$-category and $F$ is an object of $\mathrm{Stab}(\bold D^{op})$, is called a \textit{dual deformation context}. 
\end{Def}
\raisebox{0.07cm}{\danger} If $(\bold D,F)$ is a dual deformation context, $(\bold D^{op},F)$ is not in general a deformation context because $\bold D^{op}$ is almost never presentable.
\begin{Exa}
The first easy example is for $D=\MOD_A$. Nevertheless, this example is of high interest (\textit{see} Lemma \ref{bitch} below). Since $\MOD_A$ is stable, so is $\MOD_A^{op}$, and the looping and delooping functors are swapped when passing to the opposite category. Hence spectrum objects in $\MOD_A^{op}$ are of the form $(M[-n])_n$ for some object $M$. In particular, $\big(\MOD_A, (A[-n-1])_n\big)$ is a dual deformation context. 
\end{Exa}
Just like deformation contexts, dual deformation contexts can be transported using adjunctions: assume to be given a dual deformation context $(\bold D, F)$ as well as an adjunction 
$T \colon \bold D\,\begin{matrix}{{\longrightarrow}} \\[-0.3cm] {\longleftarrow}\end{matrix}\, \bold D' \colon \iota$.
Then it is easy to see that $(\bold D', T(F))$ is a dual deformation context on $\bold D'$ whenever $\bold D'$ is presentable, since $T$ preserves colimits (hence limits in the opposite category). 
Using this, one can build a lot of dual deformation contexts starting from $\MOD_A$. 
\begin{Exa}
For instance, the adjunction 
\[
\textit{free} \colon \MOD_A \adjointHA \LIE_A \colon \textit{forget}
\]
yields the dual deformation context $\big(\LIE_A, \textit{free}(A[-n-1])_n\big)$. 
\end{Exa}
\begin{Exa}\label{example-dualdefo-over}
Let $T$ be an $A$-module. Then the pull-back functor $-\times_T0:\MOD_{A/T}\to\MOD_A$ along the zero morphism $0\to T$ admits a left adjoint: it is the functor sending an $A$-module $S$ to the zero morphism $S\overset{0}{\to} T$. 
This yields the dual deformation context $\big(\MOD_{A/T},(A[-n-1]\overset{0}{\to}T)_n\big)$. 
\end{Exa}

\begin{Def}
Given a dual deformation context $(\bold D, F)$, an object (resp.~ morphism) of $\bold D$ is \textit{good} if it is small when considered as an object (resp.~morphism) of $\bold D^{op}$. More explicitely, if $\varnothing$ denotes the initial element of $\bold D$, an object $b$ of $\bold D$ is good if there is a finite sequence of morphisms 
\[
\varnothing=b_m \xrightarrow{f_m}\cdots \xrightarrow{f_2} b_1 \xrightarrow{f_1} b_0\cong b
\]
that are pushouts along 
$F_n\to \varnothing$ in $\bold D$: i.e.~each $f_i$ fits into a pushout square 
\[
\xymatrix{
F_n \ar[d]\ar[r]  & \ar[d]^{f_i} b_i  \\
\varnothing \ar[r] & b_{i-1} 
}
\]
in $\bold D$ for some $n\geq1$. 
We let $(\bold D,F)^{gd}$ be the full subcategory of $\bold D$ spanned by the good objects, 
which we might simply denote $\bold D^{gd}$ if there is no ambiguity. 
\end{Def}
Let us give a somehow nontrivial example. Let $A$ be a $\kk$-algebra\footnote{$A$ is concentrated in degree $0$. }, and consider the dual deformation context $\big(\MOD_{A}, (A[-n-1])_n\big)$.
\begin{Lem} \label{bitch}
An element of $\MOD_{A}$ is good for the dual deformation context $\big(\MOD_{A}, (A[-n-1])_n\big)$ if and only if it is perfect and cohomologically concentrated in positive degrees.
\end{Lem}
\begin{proof}
Recall that a perfect complex is the same as a dualizable object, which means a complex quasi-isomorphic to a finite complex consisting of projective $A$-modules of finite type. 
Let $K$ be a perfect complex concentrated in positive degree and prove, by induction on the amplitude, that $K$ is good. 
If the amplitude is $0$ then $K$ is quasi-isomorphic to $P[-n]$, for $n \geq 1$, with $A^r=P\oplus A^q$ . We have the following push-out square in $\MOD_A$:
\[
\xymatrix{
A^q[-n-2] \ar[d] \ar[r] & 0 \ar[d] \\
0  \ar[r] & A^q[-n-1]
}
\]
Hence $A^q[-n-1]$ is good. We then also have another push-out square: 
\[
\xymatrix{
A^r[-n-1] \ar[d] \ar[r] & A^q[-n-1] \ar[d] \\
0  \ar[r] & P[-n]
}
\]
Hence the morphism $A^q[-n-1]\to P[-n]$ is good, and since $A^q[-n-1]$ is good, then $P[-n]$ is good as well. 
\par \medskip
Performing the induction step is now easy: let $K$ be a positively graded complex of finitely generated projective $A$-modules of some finite amplitude $d>0$, let $n>0$ be the index where $K$ starts and let $P=K_n$. 
We have a push-out square:
\[
\xymatrix{P[-n-1] \ar[d] \ar[r] & \tau^{>n} K \ar[d] \\
0 \ar[r] & K
}
\] 
where $\tau^{>n}$ is the stupid truncation functor, and using again that $A^r=P\oplus A^q$ we get another push-out square: 
\[
\xymatrix{A^r[-n-1] \ar[d] \ar[r] & A^q[-n-1]\oplus\tau^{>n} K \ar[d] \\
0 \ar[r] & K
}
\]
Hence the morphism $A^q[-n-1]\oplus\tau^{>n} K\to K$ is good. But $\tau^{>n} K$ is good (by induction on the amplitude) and $A^q[-n-1]$ is good as well, thus so is $K$. This finishes the induction step.
\par \medskip
For the converse statement, it suffices to observe that, given a push-out square
\[
\xymatrix{A[-n-1] \ar[r] \ar[d] & K \ar[d] \\
0 \ar[r] & L
}
\]
in $\MOD_A$, where $n\geq1$, and $K$ is perfect and concentrated in positive degrees, then so is $L$. .
\end{proof}
Similarly, one can prove that small objects for the deformation context $\big(\MOD_{A}, (A[n+1])_n\big)$ are perfect complexes of $A$-modules cohomologicaly concentrated in negative degrees. 
We leave it as an exercise to the reader. 
\begin{Rem} \label{cheat}
If we replace the $\kk$-algebra $A$ by a bounded cdga concentrated in non-positive degrees, then one can still prove that good objects are perfect $A$-modules that are cohomologically generated in positive degree. 
In other words, a good object is quasi-isomorphic, as an $A$-module, to an $A$-module $P$ having the following property: as a graded $A$-module, $P$ is a direct summand of $A\otimes V$, where $V$ is a finite dimensional positively graded $\kk$-module. 
\end{Rem}

\subsection{Koszul duality contexts}
We now introduce the main notion that is needed to state Lurie's theorem on formal moduli problems in full generality: 

\begin{Def}
(1) A \textit{weak Koszul duality context} is the data of :
\begin{itemize}
\item[--] a deformation context $(\bold C,E)$, 
\item[--] a dual deformation context $(\bold D,F)$, 
\item[--] an adjoint pair $\mathfrak D:\bold C\,\begin{matrix}\longrightarrow \\[-0.3cm] \longleftarrow\end{matrix}\,\bold D^{op}:\mathfrak D'$, 
\end{itemize}
such that for every $n\geq0$ there is an equivalence $E_n \simeq \mathfrak D' F_n$. \\[0.2cm]
\indent (2) A \textit{Koszul duality context} is a weak Kozsul duality context satisfying the two additional properties: 
\begin{itemize}
\item[--] For every good object $d$ of $\bold D$, the counit morphism $\mathfrak D\mathfrak D' d \rightarrow d$ is an equivalence.
\item[--] The functor
\[
\Theta \colon \mathrm{Hom}_{\mathbf{D}}(F_n,-) \colon \mathbf{D} \rightarrow \mathrm{Stab}(\SSET)=\SPA
\]
is conservative and preserves small sifted colimits.
\end{itemize}
With a slight abuse of notation, we will denote the whole package of a (weak) Koszul duality context by 
\[
\mathfrak D \colon (\bold C, E)\,\begin{matrix}\longrightarrow \\[-0.3cm] \longleftarrow\end{matrix}\,(\bold D^{op}, F) \colon \mathfrak D'
\]
\end{Def}
\begin{Rem}
The name Kosul duality context has been chosen in agreement with the first non-trivial example we will deal with (\textit{see} Proposition ref{opinel}), as it reflects the well-known Koszul duality between commutative and Lie algebras. Similarly, the characterization of associative formal moduli problems \cite[\S3]{DAG-X}, resp.~$\mathbb{E}_n$ formal moduli problems \cite[\S4]{DAG-X}, can be proven with the help of a Koszul duality context that reflects the Koszul duality for associative algebras, resp.~$\mathbb{E}_n$-algebras. We actually expect that a pair $(\mathcal O,\mathcal P)$ of Koszul dual operads always lead to a Koszul duality context between $\mathcal O$-algebras and $\mathcal P$-algebras, allowing to show that formal moduli problems for $\mathcal O$-algebras are $\mathcal P$-algebras.  
\end{Rem}
Observe that there may be more objects of $\bold D$ for which the counit morphism  is an equivalence, than just good objects. We call them \textit{reflexive objects}. 
\begin{Exa}\label{Koszul-modules}
The only elementary Koszul duality context we can give at this stage is the following: if $A$ is a bounded cdga concentrated in nonpositive degrees, 
\[
(-)^\vee=\mathrm{Hom}_A(-, A) \colon \big(\MOD_{A}, (A[n+1])_n\big) \,\begin{matrix}\longrightarrow \\[-0.3cm] \longleftarrow\end{matrix}\,\big(\MOD_{A}^{op}, (A[-n-1])_n\big) \colon\mathrm{Hom}_A(-, A)=(-)^\vee.
\]
To prove that it is indeed a Koszul duality context, we use Lemma \ref{bitch} and Remark \ref{cheat}: a good object in $\MOD_{A}^{op}$ is a perfect $A$-module (generated in non-negative degrees), so it is isomorphic to its bidual. Lastly, the functor $\Theta$ is simply the forgetful functor
\[
\MOD_A \rightarrow \MOD_{\kk} \simeq \mathrm{Stab}(\MOD_{\kk}) \rightarrow \mathrm{Stab}(\SSET) \simeq \SPA
\]
which is conservative. 
Note that there are strictly more reflexive objects than good ones: for instance, $\oplus_{n\geq0}A[-n]$ is reflexive, but not good (it is an example of an almost finite cellular object in the terminology of \cite{He}). 
\end{Exa}
There is a peculiar refinement of the above example, that will be useful for later purposes. 
\begin{Exa}\label{Koszul-mod-over-under}
Let $A$ be a bounded cdga concentrated in nonpositive degrees, and let $L$ be an $A$-module. 
We have a weak Koszul duality context\footnote{Observe that, contrary to what one could think, it is not required that $L$ is perfect, or even just reflexive. 
Indeed, for $L\overset{\ell}{\to} M$ and $K\overset{k}{\to} L^\vee$, the following commuting diagrams completely determine each other: 
\[
\xymatrix{
K \ar[r]^-{f^\vee}\ar[d]^-{k} & M^\vee  \ar[dl]^-{\ell^\vee} \\
L^\vee & 
}
\qquad\mathrm{and}\qquad
\xymatrix{
L \ar[r]\ar[d]^-{\ell} & (L^\vee)^\vee\ar[d]^-{k^\vee} \\
M \ar[r]^-f & K^\vee 
}
\]
} 
\[
(-)^\vee \colon \big(^{L/}\!\MOD_{A}, (L\overset{0}{\to}A[n+1])_n\big) \,\begin{matrix}\longrightarrow \\[-0.3cm] \longleftarrow\end{matrix}\,\big((\MOD_{A/L^\vee})^{op}, (A[-n-1]\overset{0}{\to}L^\vee)_n\big) \colon(-)^\vee\,,
\]
where the right adjoint sends $K\overset{k}{\longrightarrow} L^\vee$ to the composed morphism $L\to (L^\vee)^\vee\overset{k^\vee}{\longrightarrow} K^\vee$. 
We claim that this is actually a Koszul duality context. 
Indeed, one first observes that good objects are given by morphisms $K\to L^\vee$ where $K$ is a good object in $\MOD_A$. Hence they are again isomorphic to their biduals. 
Lastly, the functor $\Theta$ is the composition of the pull-back functor $-\underset{L^\vee}{\times}0:(\MOD_{A/L^\vee})^{op}\to\MOD_A$ along the zero morphism $0\to L^\vee$, with the forgetful functor $\MOD_A \to \SPA$ from the previous example. They are both conservative, hence $\Theta$ is.  
\end{Exa}

There are a few properties that can be deduced from the definition, which are absolutely crucial.
\begin{Prop}[{\cite[Proposition 1.3.5]{DAG-X}}] \label{yolo}
Assume that we are given a Koszul duality context
\[
\mathfrak D \colon (\bold C, E) \adjointHA (\bold D^{op}, F) \colon \mathfrak D'
\]
\begin{itemize}
\item[\textbf{\emph{(A)}}] For every $n\geq0$, $\mathfrak{D} E_n \simeq F_n$.
\item[\textbf{\emph{(B)}}] For every small object $\mathcal M$ in $\bold C$, the unit map $\mathcal{M} \rightarrow \mathfrak{D}' \mathfrak{D}(\mathcal{M})$ is an equivalence.
\item[\textbf{\emph{(C)}}] There is an equivalence of categories
$
\xymatrix{
(\bold C , E)^{sm} \ar@<2pt>[r]^{\mathfrak{D}} & \ar@<2pt>[l]^{\mathfrak{D}'} (\bold D^{op} , F)^{gd}.
}
$
\item[\textbf{\emph{(D)}}] Consider a pullback diagram
\[
\xymatrix{
N \ar[r] \ar[d] & A \ar[d]^-{f} \\
M \ar[r] & B
}
\]
where $f$ is small and $M$ is small. Then the image of this diagram by $\mathfrak{D}$ is still a pullback diagram.
\end{itemize}
\end{Prop}
\begin{proof} The proof is clever but completely formal, and doesn't require any extra input. 
\par \medskip
\textbf{(A)}  This is straightforward: $\mathfrak{D} E_n \simeq \mathfrak{D} \mathfrak{D}' F_n \simeq F_n$ since $F_n$ is good. 
\par \medskip
\textbf{(B)}  We proceed in two steps.
First we prove it if $\mathcal{M}$ is for the form $\mathfrak{D}'(\mathcal{N})$ for some good object $\mathcal{N}$. This is easy:
the composition
\[
\mathfrak{D}' \cong \mathrm{id}  \circ \mathfrak{D}'  \Rightarrow  (\mathfrak{D}' \circ  \mathfrak{D}) \circ  \mathfrak{D}' \cong  \mathfrak{D}' \circ (\mathfrak{D} \circ  \mathfrak{D}') \Rightarrow \mathfrak{D}' \circ \mathrm{id} \cong \mathfrak{D}'
\]
is an equivalence (it is homotopic to the identity). This means that the composition
$\mathfrak{D}' \mathfrak{D} (\mathfrak{D}' \mathcal{N}) \rightarrow \mathfrak{D}' \mathcal{N}$ is obtained up to equivalence by applying $\mathfrak{D}'$ to the equivalence $\mathfrak{D} \mathfrak{D}'(\mathcal{N}) \rightarrow \mathcal{N}$, hence it is also an equivalence.
\par \medskip
Now we prove that every small object can be written as $\mathfrak{D}'(X)$ where $X$ is good. This will in particular imply that $\mathfrak{D}$ maps small objects to good objects\footnote{This is nontrivial: $\mathfrak{D}$ preserves \textit{colimits}, and small objects are defined using \textit{limits}.}. Every small
morphism can be written as a composition of a finite number $i$ of elementary morphisms, we will argue by induction on the number $i$. If $i=0$, then since $\mathcal{D}'$ is a right adjoint it preserves limits. Hence $\mathfrak{D}'(\varnothing)={*}$, where $\varnothing$ is the initial element of $\bold{D}$. Since $\varnothing$ is good, we are done. We now deal with the induction step. Let us consider a cartesian diagram
\[
\xymatrix{\mathcal{M}_{i+1} \ar[d] \ar[r] & {*} \ar[d] \\
\mathcal{M}_i \ar[r] & E_n
}
\]
in $\bold C$ such that $\mathcal{M}_i \rightarrow *$ is obtained by composition of at most $i$ elementary morphisms. By induction, $\mathcal{M}_{i+1}=\mathfrak{D}'(\mathcal{N})$ for some good object $\mathcal{N}$. This implies that $\mathfrak{D}(\mathcal{M}_{i})$ is good. Let $X$ be the pullback the object of $\bold{D}$ making the diagram
\[
\xymatrix{X \ar[r] \ar[d] & {*} \ar[d] \\
\mathfrak{D}(\mathcal{M}_{i}) \ar[r] &  \mathfrak{D}(E_n)
}
\]
cartesian in $\bold{D}^{op}$. Since $\mathfrak{D}(\mathcal{M}_{i})$ is good and $\mathfrak{D}(E_n) \simeq F_n$, $X$ is good. Now we apply $\mathfrak{D}'$, which preserves limits. We get a cartesian diagram isomorphic to 
\[
\xymatrix{\mathfrak{D}'(X) \ar[r] \ar[d] & {*} \ar[d] \\
\mathcal{M}_{i} \ar[r] & E_n
}
\]
This proves that $\mathcal{M}_{i+1}$ is isomorphic to $\mathfrak{D}'(X)$.
\par \medskip
\textbf{(C)} This is a direct consequence of \textbf{(B)}.
\par \medskip
\textbf{(D)} We can reduce to diagrams of the form
\[
\xymatrix{
N \ar[r] \ar[d] & {*} \ar[d] \\
M \ar[r] & E_n
}
\]
Then the property follows from \textbf{(C)}.
\end{proof}

\subsection{Morphisms of (weak) Koszul duality contexts}

We will now introduce the notion of morphisms between weak Koszul duality contexts. It will be extremely useful in the sequel.
\begin{Def} \label{giletjaunes}
Let \[
\mathfrak D_1 \colon (\bold C_1, E_1)\,\begin{matrix}\longrightarrow \\[-0.3cm] \longleftarrow\end{matrix}\,(\bold D_1^{op}, F_1) \colon \mathfrak D'_1
\]
and
\[
\mathfrak D_2 \colon (\bold C_2, E_2)\,\begin{matrix}\longrightarrow \\[-0.3cm] \longleftarrow\end{matrix}\,(\bold D^{op}_2, F_2) \colon \mathfrak D'_2
\]
be two weak Koszul duality contexts. 
\par \medskip
\noindent{}A \textit{weak morphism} between these two duality contexts consists of two additional pairs of adjoint functors appearing (vertically) in the diagram
\[
\xymatrix@C=60pt@R=30pt{
\bold C_1 \ar@<-2pt>[d]_-{S} \ar@<2pt>[r]^-{\mathfrak{D}_1} &   \ar@<2pt>[l]^-{\mathfrak{D}'_1} \bold{D}_1^{op} \ar@<-2pt>[d]_-{Y} \\
\bold C_2\ar@<-2pt>[u]_-{T} \ar@<2pt>[r]^-{\mathfrak{D}_2}  &  \bold{D}_2^{op}  \ar@<2pt>[l]^-{\mathfrak{D}'_2} \ar@<-2pt>[u]_-{Z}
}
\]
such that:
\begin{itemize}
\item[\textbf{{(A)}}] The diagram consisting of right adjoints $\xymatrix{ \bold C_1 & \ar[l]_-{\mathfrak{D}'_1} \bold D_1^{op}
\\ \bold C_2 \ar[u]^-{T} & \ar[u]_-{Z} \ar[l]^-{\mathfrak{D}'_2} \bold D_2^{op}
}$ commutes.
\item[\textbf{{(B)}}] We have equivalences $Z(F_{2, n}) \simeq F_{1, n}$.
\end{itemize}
\end{Def}
\begin{Exa}\label{ex-tricky}
Let $L$ be an $A$-module. We then have the following weak morphism from the Koszul duality context of Example \ref{Koszul-mod-over-under} to the one of Example \ref{Koszul-modules}: 
\[
\xymatrix@C=60pt@R=30pt{
^{L/}\!\MOD_A \ar@<-2pt>[d]_-{\textit{cofib}} \ar@<2pt>[r]^-{(-)^\vee} &   \ar@<2pt>[l]^-{(-)^\vee} (\MOD_{A/L^\vee})^{op} \ar@<-2pt>[d]_-{\textit{fib}} \\
\MOD_A\ar@<-2pt>[u]_-{M\mapsto (L\overset{0}{\to}M)} \ar@<2pt>[r]^-{(-)^\vee}  &  \MOD_A^{op}  \ar@<2pt>[l]^-{(-)^\vee} \ar@<-2pt>[u]_-{K\mapsto (K\overset{0}{\to} L^\vee)}
}
\] 
where $\textit{cofib}=0\underset{L}{\coprod}-$ and $\textit{fib}=0\underset{L^\vee}{\times}-$. 
\end{Exa}
Notice that the natural equivalence \raisebox{-1ex}{\begin{tikzpicture}
\draw [>=stealth,->] (0.4,0.4) to (0,0.4) ;
\draw [>=stealth,->] (0.4,0) to (0.4,0.4) ;
\draw [>=stealth,thin,double,<->] (0.5,0.2) to (0.9,0.2);
\draw [>=stealth,->] (1,0) to (1,0.4) ;
\draw [>=stealth,->] (1.4,0) to (1,0) ;
\end{tikzpicture}} from condition \textbf{(A)} above has a \textit{mate} \raisebox{-1ex}{\begin{tikzpicture}
\draw [>=stealth,->] (0.4,0.4) to (0,0.4) ;
\draw [>=stealth,->] (0,0.4) to (0,0) ;
\draw [>=stealth,thin,double,->] (0.5,0.2) to (0.9,0.2) node[above left]{{\tiny $\theta$}};
\draw [>=stealth,->] (1.4,0.4) to (1.4,0) ;
\draw [>=stealth,->] (1.4,0) to (1,0) ;
\end{tikzpicture}} which we define pictorially as the composition 
\[
\begin{tikzpicture}
\draw [>=stealth,->] (0.4,0.4) to (0,0.4) ;
\draw [>=stealth,->] (0,0.4) to (0,0) ;
\draw [>=stealth,thin,double,->] (0.5,0.2) to (0.9,0.2) ;
\draw [>=stealth,->] (1.3,0.4) to (1,0.4) ;
\draw [>=stealth,->] (1,0.4) to (1,0) ;
\draw [>=stealth,->] (1.4,0.4) to (1.4,0) to (1.3,0) to (1.3,0.4) ;
\draw [>=stealth,thin,double,<->] (1.5,0.2) to (1.9,0.2) ;
\draw [>=stealth,->] (2.4,0.4) to (2.4,0) ;
\draw [>=stealth,->] (2.4,0) to (2.1,0) ;
\draw [>=stealth,->] (2.1,0) to (2.1,0.4) to (2,0.4) to (2,0) ;
\draw [>=stealth,thin,double,->] (2.5,0.2) to (2.9,0.2) ;
\draw [>=stealth,->] (3.4,0.4) to (3.4,0) ;
\draw [>=stealth,->] (3.4,0) to (3,0) ;
\end{tikzpicture}
\]
where 
\raisebox{-1ex}{\begin{tikzpicture}
\node (A) at (-0.2,0.2) {$id$};
\draw [>=stealth,thin,double,->] (0.1,0.2) to (0.5,0.2);
\draw [>=stealth,->] (0.7,0.4) to (0.7,0) to (0.6,0) to (0.6,0.4) ;
\end{tikzpicture}} and \raisebox{-1ex}{\begin{tikzpicture}
\node (A) at (0.7,0.2) {$id$};
\draw [>=stealth,thin,double,->] (0.1,0.2) to (0.5,0.2);
\draw [>=stealth,->] (0,0) to (0,0.4) to (-0.1,0.4) to (-0.1,0) ;
\end{tikzpicture}} are used to depict units and counits of vertical adjunctions. In other words, the mate $\theta$ is the composition 
$S\mathfrak{D}_1'\Rightarrow S\mathfrak{D}_1'ZY\cong ST\mathfrak{D}_2'Y\Rightarrow \mathfrak{D}_2'Y$. 
\begin{Rem}
The commutativity of the square of right adjoints implies the commutativity of the square of left adjoints. Hence there is a natural equivalence $\mathfrak{D}_2S\cong Y\mathfrak{D}_1$ that pictorially reads \raisebox{-1ex}{\begin{tikzpicture}
\draw [>=stealth,->] (0,0.4) to (0.4,0.4) ;
\draw [>=stealth,->] (0.4,0.4) to (0.4,0) ;
\draw [>=stealth,thin,double,<->] (0.5,0.2) to (0.9,0.2);
\draw [>=stealth,->] (1,0.4) to (1,0) ;
\draw [>=stealth,->] (1,0) to (1.4,0) ;
\end{tikzpicture}}, and the mate $\theta$ can also be identified with the following composition:
\[
S\mathfrak{D}_1'\Rightarrow \mathfrak{D}_2'\mathfrak{D}_2S\mathfrak{D}_1\cong  \mathfrak{D}_2'Y\mathfrak{D}_1\mathfrak D_1'\Rightarrow \mathfrak{D}_2'Y
\qquad\textrm{or, pictorially}\qquad
\begin{tikzpicture}
\draw [>=stealth,->] (0.4,0.4) to (0,0.4) ;
\draw [>=stealth,->] (0,0.4) to (0,0) ;
\draw [>=stealth,thin,double,->] (0.5,0.2) to (0.9,0.2) ;
\draw [>=stealth,->] (1.4,0.4) to (1,0.4) ;
\draw [>=stealth,->] (1,0.4) to (1,0.1) ;
\draw [>=stealth,->] (1,0.1) to (1.4,0.1) to (1.4,0) to (1,0) ;
\draw [>=stealth,thin,double,<->] (1.5,0.2) to (1.9,0.2) ;
\draw [>=stealth,->] (2.4,0.4) to (2,0.4) to (2,0.3) to (2.4,0.3) ;
\draw [>=stealth,->] (2.4,0.3) to (2.4,0);
\draw [>=stealth,->] (2.4,0) to (2,0) ;
\draw [>=stealth,thin,double,->] (2.5,0.2) to (2.9,0.2) ;
\draw [>=stealth,->] (3.4,0.4) to (3.4,0) ;
\draw [>=stealth,->] (3.4,0) to (3,0) ;
\end{tikzpicture}\,.
\]
\end{Rem}
\begin{Def}
One says that the above commuting square of right ajoints satisfies the \textit{Beck--Chevalley condition locally at $d$}, where $d$ is an object of $\mathbf{D}_1$, 
if the mate $\theta_d:S\mathfrak{D}_1'(d)\to \mathfrak{D}_2'Y(d)$ is an equivalence. 
\end{Def}

\begin{Def}\label{def-Koszulmorphism}
A \textit{morphism} of weak Koszul duality contexts is a weak morphism such that, borrowing the above notation: 
\begin{itemize}
\item[\textbf{{(C)}}] The functor $Y$ is conservative, preserves small sifted limits, and sends good objects to reflexive objects.
\item[\textbf{{(D)}}] The commuting square of right adjoints satisfies the Beck--Chevalley condition locally at good objects. 
\end{itemize}
\end{Def}
 
The main feature of this definition is a result allowing to transfer Koszul duality contexts along morphisms:

\begin{Prop}[Transfer theorem] \label{suppositoire}
Assume to be given a morphism between two weak Koszul duality contexts. If the target deformation context is a Koszul duality context, then the source is also a Koszul duality context.
\end{Prop}

\begin{proof}
We take the notation of Definition \ref{giletjaunes}. We have two functors $\Theta_1$ and $\Theta_2$, defined on $\bold{D}_1$ and $\bold{D}_2$ respectively and with values in $\SPA$, defined by
\[
\begin{cases}
\Theta_1(d)_n=\mathrm{Hom}_{\bold{D}_1}(F_{1, n}, d) \\
\Theta_2(d)_n=\mathrm{Hom}_{\bold{D}_2}(F_{2, n}, d) \\
\end{cases}
\]

For any object $d$ in $\bold D_1$, we have 
\begin{align*}
\Theta_1(d)_n&=\mathrm{Hom}_{\bold{D}_1}(F_{1, n}, d) \simeq \mathrm{Hom}_{\bold{D}_1^{op}}(d, Z(F_{2, n}))  \\
&\simeq \mathrm{Hom}_{\bold{D}_2^{op}}(Y(d), F_{2, n}) \simeq \mathrm{Hom}_{\bold{D}_2}(F_{2, n}, Y^{op}(d)) \\
&\simeq \Theta_2(Y^{op}(d))_n
\end{align*}
So $\Theta_1=\Theta_2 \circ Y^{op}$, and thus $\Theta_1$ is conservative (because $Y$ and $\Theta_2$ are). 
Moreover, $Y$ preserves small sifted limits, so $Y^{op}$ preserves small sifted colimits, and hence so does $\Theta_1$ (because $\Theta_2$ does too). Lastly, we consider the following diagram in $\bold D_2^{op}$, where $d$ is still an object in $\bold D_1$: 
\[
\xymatrix@C=60pt{
Y \mathfrak{D}_1 \mathfrak{D}'_1 \ar@{<=>}[d]_-{\sim}  \ar@{=>}[r]^-{Y \circ \,\textrm{co-unit} \,} & Y \\
\mathfrak{D}_2 S  \mathfrak{D}'_1 \ar@{=>}[r]_-{\mathfrak{D}_2\circ\theta} & \mathfrak{D}_2 \mathfrak{D}'_2 Y \ar@{=>}[u]_-{\textrm{co-unit}\,\circ Y}.
}
\]
We claim that it commutes. Indeed, writing the mate explicitely gives a diagram
\[
\xymatrix@C=60pt{
Y \mathfrak{D}_1 \mathfrak{D}'_1 \ar@{<=>}[d]_-{\sim}  \ar@{=>}[r]^-{Y \circ \,\textrm{co-unit} \,} & Y \\
\mathfrak{D}_2 S  \mathfrak{D}'_1 \ar@{=>}[d]_-{\mathfrak{D}_2\circ\,\textrm{unit}\,\circ S  \mathfrak{D}'_1} & \mathfrak{D}_2 \mathfrak{D}'_2 Y \ar@{=>}[u]_-{\textrm{co-unit}\,\circ Y} \\
\mathfrak{D}_2\mathfrak{D}_2'\mathfrak{D}_2 S  \mathfrak{D}'_1 \ar@{<=>}[r]^-{\sim} & \mathfrak{D}_2\mathfrak{D}_2'Y\mathfrak{D}_1\mathfrak{D}_1' \ar@{=>}[u]_-{\mathfrak{D}_2\mathfrak{D}_2'Y\circ\,\textrm{co-unit}}
\ar@{=>}[uul]|-{\textrm{co-unit}\,\circ Y\mathfrak{D}_1\mathfrak{D}_1'}
}
\]
where the two triangles commute. 
Given now a good object $d$ in $\bold D_1$, we have that $Y(d)$ is reflexive (after \textbf{(C)} in Definition \ref{def-Koszulmorphism}): thus the co-unit morphism of the adjunction $(\mathfrak{D}_2, \mathfrak{D}'_2)$ is an equivalence on $Y(d)$. Besides, $\theta_d$ is also an equivalence (thanks to \textbf{(D)} in Definition \ref{def-Koszulmorphism}). Therefore $Y\circ \textrm{co-unit}(d)$ is an equivalence and thus, by conservativity of $Y$, the co-unit morphism $\mathfrak{D}_1\mathfrak{D}_1'(d)\to d$ is an equivalence.
\end{proof}
\begin{Exa}\label{ex-tricky-strikes-back}
Going back to Example \ref{ex-tricky}, one sees that the weak morphism is a morphism if and only if $L$ is reflexive (in which case, $L^\vee$ is so as well). Indeed: 
\begin{itemize}
\item[--] The functor $\textit{fib}$ is conservative. 
\item[--] If $K\to L^\vee$ is good then $K$ itself is good, thus perfect, in $\MOD_A$, and thus $\textit{fib}(K\to L^\vee)$ is reflexive as soon as $L^\vee$ is so.  
\item[--] the mate $\theta_{K\to L^\vee}$ is the natural morphism $\textit{cofib}(L\to K^\vee)\to \textit{fib}(K\to L^\vee)^\vee$, which is an equivalence if and only if the unit morphism $L\to (L^\vee)^\vee$ is an equivalence. 
\end{itemize}
Hence the Koszul duality context from Example \ref{Koszul-mod-over-under} is in general not obtained by transfer from the one of \ref{Koszul-modules}, but it is in the case when $L$ is reflexive.
\end{Exa}
\begin{Prop} \label{opinel}
The adjunction
\[
\mathfrak D \colon (\CALG_{\kk}, \kk \oplus \kk[n]) \adjointHA (\LIE_{\kk}^{op}, \textit{free}\, \kk[-n-1]) \colon \mathrm{CE}^{\bullet}
\]
defines a Koszul duality context.
\end{Prop}
\begin{proof}
We consider the diagram (\textit{see} Remark \ref{muraillette}):
\[
\xymatrix@C=80pt@R=40pt{
\CALG_{\kk} &   \ar[l]_-{\mathrm{CE}^{\bullet}} \LIE_{\kk}^{op} \\
\mod_{\kk}\ar[u]^-{V \mapsto \kk \oplus V[-1]} & \mod_{\kk}^{op}  \ar[l]_-{\mathrm{Hom}(-, \kk)} \ar[u]_-{\textit{free}}
}
\]
We can fill it with left adjoints everywhere. This gives the following (nice!) diagram:
\[
\xymatrix@C=100pt@R=50pt{
\CALG_{\kk} \ar@<-2pt>[d]_-{\mathfrak{L}} \ar@<2pt>[r]^-{\mathfrak{D}}&   \ar@<2pt>[l]^-{\mathrm{CE}^{\bullet}} \LIE_{\kk}^{op} \ar@<-2pt>[d]_-{\textit{forget}} \\
\MOD_{\kk}\ar@<-2pt>[u]_-{V \mapsto \kk \oplus V[-1]} \ar@<2pt>[r]^-{\mathrm{Hom}(-, \kk)}  & \MOD_{\kk}^{op}  \ar@<2pt>[l]^-{\mathrm{Hom}(-, \kk)} \ar@<-2pt>[u]_-{\textit{free}}
}
\]
where $\mathfrak{L}(R) =\mathbb{L}_{\kk/R}\simeq \mathbb{L}_{R/\kk}\otimes_R\kk[1]$. 

We claim that these four adjunctions define a morphism of weak Koszul duality contexts. Properties \textbf{(A)} and \textbf{(B)} are true, so that we have a weak morphism. 
We will now prove that it is actually a morphism, so that we get the result, using Proposition \ref{suppositoire} and the fact that the bottom adjunction is a Koszul duality context (Example \ref{Koszul-modules}).  
\par \medskip
The functor $\textit{forget}$ is conservative. Let us now prove that, if $\mathfrak{g}$ is good, then it is equivalent to a \textit{very good} dgla: a good dgla with underlying graded Lie algebra being generated by a finite dimensional graded vector space sitting in positive degrees. 
\begin{Lem}\label{verygoodlemma}
Any good dgla $\mathfrak{g}$ is quasi-isomorphic to a very good one. 
\end{Lem}
\begin{proof}[Proof of the lemma]
We first observe that $0$ is very good. We then proceed by induction: assume that $\mathfrak{g}$ is very good, and consider a dgla $\mathfrak{g}'$ obtained by a pushout 
\[
\xymatrix{\textit{free}\,\, \kk[-n-1]   \ar[r] \ar[d] & \mathfrak{g} \ar[d] \\
0 \ar[r]& \mathfrak{g}'
}
\]
in $\LIE_{\kk}$. The first trick is that $\textit{free}\,\, \kk[-n-1]$ is cofibrant and $\mathfrak{g}$ is fibrant (every object in $\lie_\kk$ is). Hence the morphism $\textit{free}\,\, \kk[-n-1]$ in the $\infty$-category $\LIE_{\kk}$ can be represented by a honest morphism in $\lie_{\kk}$. The next step consists in picking a cofibrant replacement of the left vertical arrow. A cofibrant replacement is given by the morphism
\[
\textit{free}\,\, \kk[-n-1] \rightarrow \textit{free}\,\, \left\{\textrm{cone}\,(\kk \xrightarrow{\mathrm{id}} \kk)[-n-1] \right\}.
\]
Hence our pushout in $\LIE_\kk$ is represented by the following honest non-derived pushout in $\lie_\kk$: 
\[
\xymatrix{\textit{free}\,\, \kk[-n-1]   \ar[r] \ar[d] & \mathfrak{g} \ar[d] \\
\textit{free}\,\, \left\{\textrm{cone}\,(\kk \xrightarrow{\mathrm{id}} \kk)[-n-1] \right\} \ar[r]& \mathfrak{g}'
}
\]
This pushout is obtained by making the free product of $\mathfrak{g}$ and $\textit{free}\,\, \left\{\textrm{cone}\,(\kk \xrightarrow{\mathrm{id}} \kk)[-n-1] \right\}$, and then by taking the quotient by the image of the ideal generated by $\textit{free}\,\, \kk[-n-1]$. This completes the induction step: $\mathfrak{g}'$ is still very good. 
\end{proof}
This in particular shows that $\textit{forget}$ sends good objects to reflexive objects for the dual deformation context $(\MOD_{\kk}, \kk[-n-1])$. Lastly, thanks to \cite[Proposition 2.1.16]{DAG-X}, the forgetful functor from $\LIE_{\kk}^{op} $ to $\MOD_{\kk}^{op}$ preserves small sifted limits. Thus \textbf{(C)} holds. It remains to prove \textbf{(D)}, which is the main delicate point of the proof. Recall for that purpose that, for a dgla $\mathfrak{g}$, $\theta_{\mathfrak{g}}$ is defined as the following composition, where we omit the $\textit{forget}$ functor (its appearance being obvious): 
\[
\mathfrak{L}\big(\mathrm{CE}^{\bullet}(\mathfrak{g})\big)\to \mathfrak{L}\big(\mathrm{CE}^{\bullet}(\textit{free}\,\mathfrak{g})\big)\to \mathfrak{L}(\kk\oplus\mathfrak{g}^*[-1])\to\mathfrak{g}^*\,.
\]
At this point it is important to make a rather elementary observation: the composed cdga morphism $\mathrm{CE}^{\bullet}(\mathfrak{g})\to \mathrm{CE}^{\bullet}(\textit{free}\,\mathfrak{g})\to \kk\oplus\mathfrak{g}^*[-1]$ is nothing but the projection onto the quotient by the square $\mathcal I^2$ of the augmentation ideal $\mathcal I=\ker(\mathrm{CE}^{\bullet}(\mathfrak{g})\to\kk)$ whenever $\mathfrak{g}$ is very good\footnote{Indeed, in this case the natural map $S^k(\mathfrak g^*[-1])\to S^k(\mathfrak{g}[1])^*$ is an isomorphism. Thus, forgetting the differential, $\mathrm{CE}^\bullet(\mathfrak g)=\hat{S}(\mathfrak g^*[-1])$. }. 

We now introduce the \textit{uncompleted} Chevalley--Eilenberg cdga $\widetilde{\mathrm{CE}}^{\bullet}(\mathfrak g)$ as a sub-cdga $S(\mathfrak{g}^*[-1])\subset\mathrm{CE}^{\bullet}(\mathfrak g)$ (it is obviously a graded subalgebra, and it can be easily checked as an exercise that it is stable under the differential). The quotient by the square of its augmentation ideal is still $\kk\oplus\mathfrak{g}^*[-1]$, and we have the following commuting diagram and its image through the left-most vertical adjunction of our square: 
\[
\xymatrix{ 
\widetilde{\mathrm{CE}}^{\bullet}(\mathfrak{g}) \ar[d]_{\iota_{\mathfrak{g}}}\ar[rd] & \\
\mathrm{CE}^{\bullet}(\mathfrak{g}) \ar[r] & \kk\oplus\mathfrak{g}^*[-1]
}
\qquad
\xymatrix{
\mathfrak{L}\big(\widetilde{\mathrm{CE}}^{\bullet}(\mathfrak{g})\big) \ar[d]_{\mathfrak{L}(\iota_{\mathfrak{g}})}\ar[rd]^{\widetilde{\theta}_{\mathfrak{g}}} & \\
\mathfrak{L}\big(\mathrm{CE}^{\bullet}(\mathfrak{g})\big) \ar[r]_-{\theta_{\mathfrak g}} & \mathfrak{g}^*
}
\]
In order to prove that $\theta_{\mathfrak{g}}$ is an equivalence when $\mathfrak{g}$ is good, we will prove that both $\widetilde{\theta}_{\mathfrak{g}}$ and $\mathfrak{L}(\iota_{\mathfrak{g}})$ are. 
Let us start with the following lemma: 
\begin{Lem}\label{augmentedlemma}
Let $A$ be an augmented cdga, with augmentation ideal $\mathcal J$, that is cofibrant as a cdga. 
Then the morphism $\mathfrak{L}(A)\to \mathcal J/\mathcal J^2[1]$ associated with the projection $A\to \kk\oplus \mathcal J/\mathcal J^2$ is an equivalence. 
\end{Lem}
\begin{proof}
First of all, the projection $A\to \kk\oplus \mathcal J/\mathcal J^2$ is an actual morphism in the category $\cdga_\kk$, so that we have a factorization 
\[
\mathfrak{L}(A)\simeq \mathbb{L}_{A/\kk}\otimes_A\kk[1]\to \Omega^1_{A/\kk}\otimes_A\kk[1]\to \mathcal J/\mathcal J^2[1]\,, 
\]
where: 
\begin{itemize}
\item[--] The morphism $\mathbb{L}_{A/\kk}\otimes_A\kk\to \Omega^1_{A/\kk}\otimes_A\kk$ is an equivalence because $A$ is cofibrant. 
\item[--] The second morphism is an \textbf{isomorphism} in $\mod_\kk$. 
\end{itemize}
The lemma is proved. 
\end{proof}
We then observe that when $\mathfrak{g}$ is very good (which we can always assume without loss of generality when dealing with good dglas), then $\widetilde{\mathrm{CE}}^{\bullet}(\mathfrak{g})$ is cofibrant and thus $\widetilde{\theta}_{\mathfrak g}$ is an equivalence\footnote{Here is another approach, avoiding the use of very good models. If $\mathfrak{g}$ is a dgla then there is an $L_\infty$-structure on $H^*(\mathfrak{g})$ that makes it equivalent to $\mathfrak{g}$ in $\LIE_\kk$. If moreover $\mathfrak{g}$ is good then $H^*(\mathfrak g)$ is finite dimensional and concentrated in positive degree, so that the $L_\infty$-structure has only finitely many non-trivial structure maps. Thus the uncompleted Chevalley-Eilenberg cdga $\widetilde{\mathrm{CE}}^{\bullet}\big(H^*(\mathfrak g)\big)$ can be defined, is cofibrant, and is equivalent to $\widetilde{\mathrm{CE}}^{\bullet}(\mathfrak g)$. }. 

Lastly, it can be shown that $\mathrm{CE}^{\bullet}(\mathfrak{g})$ is flat over $\widetilde{\mathrm{CE}}^{\bullet}(\mathfrak{g})$, which implies that the natural map 
\[
\mathbb{L}_{\widetilde{\mathrm{CE}}^{\bullet}(\mathfrak{g})/\kk}\otimes_{\widetilde{\mathrm{CE}}^{\bullet}(\mathfrak{g})} {\mathrm{CE}}^{\bullet}(\mathfrak{g})\longrightarrow\mathbb{L}_{{\mathrm{CE}}^{\bullet}(\mathfrak{g})/\kk}
\]
is also an equivalence, which shows (after applying $-\otimes_{{\mathrm{CE}}^{\bullet}(\mathfrak{g})}\kk$) that $\mathfrak{L}(\iota_{\mathfrak g})$ is an equivalence. 
\end{proof}

\subsection{Description of general formal moduli problems}
Assume that we are given a Koszul duality context
\[
\mathfrak D \colon (\bold C, E) \adjointHA (\bold D^{op}, F) \colon \mathfrak D'
\]
We can define a functor $\Psi \colon \bold{D} \rightarrow \bold{Fun}(\bold{C}, \SSET)$ by the composition
\[
\Psi \colon \bold{D} \xrightarrow{\mathscr{Y}} \mathbf{Fun}(\bold{D}^{op}, \SSET) \xrightarrow{\circ \,\mathfrak{D}} \mathbf{Fun}(\bold{C}, \SSET).
\]
where $\mathscr{Y}$ is the Yoneda functor $d \rightarrow \mathrm{Hom}_{\bold{D}^{op}}(d,-)$.
\begin{Thm}[Lurie \cite{DAG-X}] \label{bomb}
Given a Koszul duality context
\[
\mathfrak D \colon (\bold C, E) \adjointHA (\bold D^{op}, F) \colon \mathfrak D',
\]
The functor $\Psi$ factors through $\mathbf{FMP}(\bold{C}, E)$ and the induced functor
\[
\Psi \colon \bold{D} \rightarrow \mathbf{FMP}(\bold{C}, E)
\] 
is an equivalence.
\end{Thm}

\begin{proof}[Sketch of proof]
The first point is a direct consequence of Proposition \ref{yolo} (\textbf{D}). The proof that $\Psi$ is an equivalence proceeds on several steps.  \vspace{0.2cm}
\begin{enumerate}
\item[--] The first step consists of proving that $\Psi$ is conservative. This follows almost immediately from the hypotheses. Indeed, let $f \colon X \rightarrow Y$ an arrow in $\bold{D}$ inducing isomorphic formal moduli problems. We have
$\Psi(X)=\mathrm{Hom}_{\bold{D}}(\mathfrak{D}(*), X)$ and similarly for $Y$. Since all $E_n$ are small and $\mathfrak{D}(E_n) \simeq F_n$, $f$ induces an equivalence of spectra
\[
\mathrm{Hom}_{\bold{D}}(F_n, X) \simeq \mathrm{Hom}_{\bold{D}}(F_n, Y) .
\]
Since $\mathrm{Hom}_{\mathbf{D}}(F_n,-) \colon \mathbf{D} \rightarrow \SPA$ is conservative, $f$ is an equivalence.
\vspace{0.2cm}
\item[--] The next step consists in proving that $\Psi$ commutes with limits and accessible colimits, so it has a left adjoint. This part is elementary, and uses the same techniques as in Proposition \ref{yolo}. We denote this left adjoint by $\Phi$.
\vspace{0.2cm}
\item[--] The functor $\Psi$ being conservative, it suffices to prove that the unit $\mathrm{id}_{\mathbf{FMP}(\bold{C}, E)} \Rightarrow \Psi \circ \Phi $ is an equivalence. This is the most technical part in the proof: it involves hypercoverings to reduce to pro-representable moduli problems; this is where the condition on \textit{sifted} colimits plays a role. Here pro-representable means small limit of fmp representables by small objects.
\vspace{0.2cm}
\item[--] In the representable case, we can explain what happens: we have an adjunction diagram
\[
\xymatrix{
\bold{D} \ar@<-2pt>[rr]_-{\Psi} && \ar@<-2pt>[ll]_-{\Phi} \bold{FMP}(\bold{C}, E) \\
& \bold{C}^{op} \ar[ur]_-{\mathscr{Y}}  \ar@<-2pt>[lu]^-{\mathfrak{D}^{op}}  &
}
\]
where $\mathscr{Y}$ is the Yoneda functor $c \rightarrow \mathrm{Hom}_{\bold{C}}(c,-)$. We claim that $\Phi \circ \mathscr{Y}\simeq\mathfrak{D}^{op}$ on small objects, that is on $(\bold{C}^{sm})^{op}$. Indeed, for any objects $c$ and $d$ of $\bold{C}^{sm}$ and $\bold{D}$ respectively,
\begin{align*}
\mathrm{Hom}_{\bold{D}}(\Phi(\mathscr{Y}_c), d) &\simeq \mathrm{Hom}_{\bold{FMP}(\bold{C}, E)}(\mathscr{Y}_c, \Psi(d)) \\
&= \mathrm{Hom}_{\bold{Fun}({\bold{C}}^{sm}, E)}(\mathscr{Y}_c, \Psi(d)) \\
&= \psi(d)(c) \\
&= \mathrm{Hom}_{\bold{D}^{op}}(d, \mathfrak{D}(c)) \\
&=  \mathrm{Hom}_{\bold{D}}(\mathfrak{D}^{op}(c), d) 
\end{align*}
Using this, the unit map of $\mathscr{Y}_c$ is given by the unit map of the adjunction between $\mathfrak{D}$ and $\mathfrak{D}'$ via the natural equivalence
\[
\mathscr{Y}_c \rightarrow \Psi \circ \Phi (\mathscr{Y}_c) \simeq\Psi (\mathfrak{D}^{op}(c))=\mathrm{Hom}_{\bold D^{op}}(\mathfrak{D}(c), \mathfrak{D}(-)) \simeq \mathrm{Hom}_{\bold{C}}(c, \mathfrak{D}' \mathfrak{D}(-)).
\]
Using Proposition \ref{yolo} \textbf{(B)}, we see that this map is an equivalence of formal moduli problems.
\end{enumerate}
\end{proof}

\subsection{The tangent complex}
In this section, we introduce the tangent complex associated to a formal moduli problem. We start with a very general definition.

\begin{Def}
Let $(\bold C, E)$ be a deformation context. For any fmp $X$ in $\bold{FMP}(\bold{C}, E)$, its tangent complex $\mathbb{T}_X$ is the spectrum $X(E)=X(E_n)_n$.
\end{Def}

\begin{Rem}
There is a slight subtelty in the definition of $\mathbb{T}_X$, as $X(E_n)$ is only defined for $n \geq 0$. 
However, this suffices to define it uniquely as a spectra, by putting $(\mathbb{T}_X)_{-m}:=\Omega^{m}_*\big((\mathbb{T}_X)_0\big)$.
\end{Rem}
Assume now to be given a Koszul duality context
\[
\mathfrak D \colon (\bold C, E) \adjointHA (\bold D^{op}, F) \colon \mathfrak D'.
\]
Then we have the following result: 

\begin{Prop} \label{olan}
The following diagram commutes: 
\[
\xymatrix{
\bold D \ar[rr]^-{\Psi}  \ar[rd]_-{\Theta} && \mathbf{FMP}(\bold C, E) \ar[ld]^-{\mathbb{T}}  \\
&\SPA^{\vphantom{a}}&
}
\]
\end{Prop}

\begin{proof}
This is straightforward: given an object $d$ in $D$, $\Psi(d)=\mathrm{Hom}_{\bold{D}^{op}}(d, \mathfrak D(*))$, so
\[
\mathbb{T}_{\Psi(d)}=\mathrm{Hom}_{\bold{D}^{op}}(d,\mathfrak D(E_n))=\mathrm{Hom}_{\bold{D}}(F_n, d)=\Theta(d).
\] 
\end{proof}

\begin{Rem}
If $\bold{D}$ is $\kk$-linear (resp.~$A$-linear), then $\Theta$ actually lifts to $\MOD_{\kk}$ (resp.~$\MOD_A$): indeed, replacing $\mathrm{Hom}_{\bold D}(F_n,-)$ by its enriched version $\mathrm{HOM}_{\bold D}(F_n,-)$ gives us the lift of $\Theta$. 
\end{Rem}
\begin{Exa}
Going back to Example \ref{Koszul-modules} we get that, for an $A$-module $M$, 
$$
\Theta(M)=\mathrm{Hom}_{\MOD_A}(A[-n-1],M)_n\,.
$$
In the $\MOD_A$-enriched version, we have 
$\Theta(M)=(M[n+1])_n$ which gives $\Theta(M)\simeq M[1]$ \textit{via} the canonical identification $\mathrm{Stab}(\MOD_A)\simeq \MOD_A$. 
\end{Exa}
We now deal with functoriality: 
\begin{Prop}\label{functorialprop}
Assume to be given a weak morphism 
\[
\xymatrix@C=60pt@R=30pt{
\bold C_1 \ar@<-2pt>[d]_-{S} \ar@<2pt>[r]^-{\mathfrak D_1} &   \ar@<2pt>[l]^-{\mathfrak D_1'} \bold{D}_1^{op} \ar@<-2pt>[d]_-{Y} \\
\bold C_2\ar@<-2pt>[u]_-{T} \ar@<2pt>[r]^-{\mathfrak D_2}  &  \bold{D}_2^{op}  \ar@<2pt>[l]^-{\mathfrak D_2'} \ar@<-2pt>[u]_-{Z}
}
\]
between two Koszul duality contexts. Then there is an induced commuting diagram
\[
\xymatrix@C=40pt@R=20pt{
\bold D_1 \ar[r]^-{\sim} \ar[dd]^{Y^{op}}& \bold{FMP}\,(\bold{C}_1, E_1) \ar[dd]^-{\circ T} \ar[rd]^-{\mathbb{T}}& \\
&&\SPA \\
\bold D_2 \ar[r]^-{\sim} & \bold{FMP}\,(\bold{C}_2, E_2) \ar[ru]_-{\mathbb{T}}&
}
\]
In particular, the functor $\bold{FMP}\,(\bold{C}_1, E_1) \rightarrow \bold{FMP}\,(\bold{C}_2, E_2)$ is conservative.
\end{Prop}

\begin{proof}
We begin by proving that $T^*:=-\circ T:\mathbf{Fun}(\bold C_1,\mathbf{sSet})\to \mathbf{Fun}(\bold C_2,\mathbf{sSet})$ indeed defines a functor $\bold{FMP}\,(\bold{C}_1, E_1)\to  \bold{FMP}\,(\bold{C}_2, E_2)$: 
\begin{itemize}
\item[(1)] First, for every $n\geq0$: $T(E_{n,2})\simeq T\mathfrak{D}'_2(F_{n,2})\simeq \mathfrak{D}_1'Z(F_{n,2})\simeq\mathfrak{D}_1'(F_{n,1})\simeq E_{n,1}$. 
\item[(2)] Then, $T$ being a right adjoint it preserves in particular pull-backs along $* \to E_{n,2}$, and thus send them to pull-backs along $*\to E_{n,1}$ whenever $n\geq0$. Hence it sends small objects to small objects. 
\item[(3)] Finally, let $F$ be an fmp for $(\bold{C}_1,E_1)$. Then $F\circ T(*)\simeq F(*)\simeq *$ (because $T$ is a right adjoint and $F$ is an fmp), and $F\circ T$ preserves pull-back along $* \to E_{n,2}$ (thanks to the second point and that $F$ is an fmp).  
\end{itemize}
Note that $Z$ also sends good objects to good objects (the proof is the same as for the second point above). 

We now come to the proof of the commutativity of the square, which is essentially based on the following Lemma: 
\begin{Lem}
There is a natural transformation $\mathfrak{D}_1T\Rightarrow Z\mathfrak{D}_2$ that is an equivalence on small objects. 
\end{Lem}
\begin{proof}[Proof of the Lemma]
The commutativity of the square of right adjoints tells us there is a natural equivalence $T\mathfrak{D}_2'\cong \mathfrak{D}_1'Z$. Hence we have a mate $\mathfrak{D}_1T\Rightarrow Z\mathfrak{D}_2$ defined as the composition 
$$
\mathfrak{D}_1T\Rightarrow \mathfrak{D}_1T\mathfrak{D}_2'\mathfrak{D}_2\cong\mathfrak{D}_1\mathfrak{D}_1'Z\mathfrak{D}_2\Rightarrow Z\mathfrak{D}_2\,,
$$
that can also be depicted as \raisebox{-1ex}{~\begin{tikzpicture}
\draw [>=stealth,<-] (0.4,0.4) to (0,0.4) ;
\draw [>=stealth,<-] (0,0.4) to (0,0) ;
\draw [>=stealth,thin,double,->] (0.5,0.2) to (0.9,0.2) ;
\draw [>=stealth,<-] (1.4,0.4) to (1,0.4) ;
\draw [>=stealth,<-] (1,0.4) to (1,0.1) ;
\draw [>=stealth,<-] (1,0.1) to (1.4,0.1) to (1.4,0) to (1,0) ;
\draw [>=stealth,thin,double,<->] (1.5,0.2) to (1.9,0.2) ;
\draw [>=stealth,<-] (2.4,0.4) to (2,0.4) to (2,0.3) to (2.4,0.3) ;
\draw [>=stealth,<-] (2.4,0.3) to (2.4,0);
\draw [>=stealth,<-] (2.4,0) to (2,0) ;
\draw [>=stealth,thin,double,->] (2.5,0.2) to (2.9,0.2) ;
\draw [>=stealth,<-] (3.4,0.4) to (3.4,0) ;
\draw [>=stealth,<-] (3.4,0) to (3,0) ;
\end{tikzpicture}~}. 
We then observe that 
\begin{itemize}
\item[--] on small objects, the unit $id\Rightarrow \mathfrak{D}_2'\mathfrak{D}_2$ is an equivalence. 
\item[--] $Z\mathfrak{D}_2$ sends small objects to good objects ($\mathfrak{D}_2$ realizes an equivalence between smalls and goods, and $Z$ preserves the goods). 
\item[--] on good objects, the co-unit $\mathfrak{D}_1\mathfrak{D}_1'\Rightarrow id$ is an equivalence. 
\end{itemize}
Hence the mate $\mathfrak{D}_1T\Rightarrow Z\mathfrak{D}_2$ is an equivalence on small objects. 
\end{proof}
The commutativity of the square then reads as follows (recall that we are reasoning on the category of small objects): 
\begin{eqnarray*}
T^*\Psi_1 & = & \mathrm{Hom}_{\bold D_1^{op}}\big(-,\mathfrak D_1T(-)\big)\quad\textrm{(by definition)} \\
& \simeq & \mathrm{Hom}_{\bold D_1^{op}}\big(-,Z\mathfrak D_2(-)\big)\quad\textrm{(using the mate)} \\
& \simeq & \mathrm{Hom}_{\bold D_2^{op}}\big(Y(-),\mathfrak D_2(-)\big)\quad\textrm{(by adjunction)} \\
& = &  \Psi_2Y^{op}\quad\textrm{(by definition).}
\end{eqnarray*}
Finally, we have to prove that the triangle commutes. This is obvious: $X\circ T(E_{n,2})\simeq X(E_{n,1})$. 
\end{proof}
\begin{Exa}
Let us see what it implies for our preferred morphism of Koszul duality context:
\[
\xymatrix@C=100pt@R=50pt{
  \CALG_{\kk} \ar@<-2pt>[d]_-{\mathfrak{L}} \ar@<2pt>[r]^-{\mathfrak{D}}&   \ar@<2pt>[l]^-{\mathrm{CE}^{\bullet}} \LIE_{\kk}^{op} \ar@<-2pt>[d]_-{\textit{forget}} \\
\MOD_{\kk}\ar@<-2pt>[u]_-{V \mapsto \kk \oplus V[-1]} \ar@<2pt>[r]^-{\mathrm{Hom}(-, \kk)}  & \MOD_{\kk}^{op}  \ar@<2pt>[l]^-{\mathrm{Hom}(-, \kk)} \ar@<-2pt>[u]_-{\textit{free}}
}
\]
Note that we are in the $\kk$-linear situation, hence viewing the tangent complex as an object in $\MOD_\kk$. Let $\vartheta$ denote the functor $V \mapsto \kk \oplus V[-1]$. 
If $X$ is in $\bold{FMP}_{\kk}$, then $X\circ\vartheta$ belongs to $\bold{FMP}\big(\MOD_\kk,(\kk[n+1])_n\big)$, and thus $\mathbb{T}_X=\mathbb{T}_{X\circ\vartheta}$. Now: 
\[
\mathbb{T}_{X\circ \vartheta}=\Theta_2\Phi_2(X\circ \vartheta)=\Phi_2(X\circ \vartheta)[1]=\textit{forget}\,\Phi_1(X)[1]\,.
\] 
Hence we obtain the following beautiful result, explaining one of the phenomenon mentionned in the introduction: 
\par \medskip
\begin{center}
\fbox{\textit{The underlying complex of the dgla $\mathfrak{g}_X:=\Phi_1(X)$ attached to $X$ is $\mathbb{T}_X[-1]$.}}
\end{center}
\end{Exa}
The following proposition describes the tangent complex of a representable moduli problem: 
\begin{Prop}\label{prop-tangent-repr}
Let $\underline{A}$ be a representable element in $\bold{FMP}_\kk$. Then its tangent complex $\mathbb{T}_{\underline{A}}$, considered in $\MOD_\kk$, is equivalent to $\mathbb{T}_{A}\otimes_A\kk$, with $\mathbb{T}_{A}:=(\mathbb{L}_{A/k})^\vee$. 
\end{Prop}
\begin{proof}
It is a simple calculation: 
\begin{align*}
(\mathbb{T}_{\underline{A}})_n:=\mathrm{Hom}_{\CALG_\kk}(A, \kk \oplus \kk[n]) &\simeq \mathrm{Hom}_{\MOD_\kk}(\mathbb{L}_{\kk/A}, \kk[n+1]) \\
& \simeq  \mathrm{Hom}_{\MOD_\kk}(\mathbb{L}_{A/\kk}\otimes_A\kk[1], \kk[n+1]) \\
& \simeq \mathbb{T}_A\otimes_A\kk[n]
\end{align*}
Hence, using the canonical identification $\mathrm{Stab}(\MOD_\kk)\simeq \MOD_\kk$ we get that $\mathbb{T}_{\underline{A}}\simeq \mathbb{T}_A\otimes_A\kk$. 
\end{proof}
As a consequence, we get that $\mathbb{T}_{\kk/A}\simeq\mathbb{T}_A\otimes_A\kk[-1]$ carries a dgla structure. 

\section{DG-Lie algebroids and formal moduli problems under $Spec(A)$}

In this part, $A$ will denote a fixed cdga over $\kk$ that is concentrated in non-positive degree and cohomologically bounded. In what follows the boundedness hypothesis is important; we will explain precisely where it has to be used.

\subsection{Split formal moduli problems under $Spec(A)$} \label{easy}

One of the main purpose of the work \cite{He} (in the local case) is to prove that the equivalence provided by Theorem \ref{yolo} can be extended when replacing the ground field $\kk$ by $A$. We have (\textit{see} Lemma \ref{gioberney})
\[
\mathrm{Stab}\,(\CDGA_{A/A}) \simeq \MOD_A.
\]
We consider the deformation context $\big(\CDGA_{A/A}, (A \oplus A[n])_n\big)$. Then Hennion's result runs as follows:
\begin{Thm}[Hennion \cite{He}]
If $A$ is a cohomologically bounded cdga concentrated in non-positive degree, there is an isomorphism
\[
\bold{FMP}_{A/A} \simeq \LIE_{A}\,,
\]
where $\bold{FMP}_{A/A}:=\bold{FMP}\big(\CDGA_{A/A}, (A \oplus A[n])_n\big)$.
\end{Thm}
\begin{Rem}
In \cite{He}, there is the additional assumption that $H^0(A)$ is noetherian, but it does not appear to be necessary. 
\end{Rem}
\begin{proof}[Hints of proof]
The strategy is to produce a Koszul duality context
\[
(\CDGA_{A/A}, A \oplus A[n]) \adjointHA (\LIE_A^{op}, \textit{free}\, \, A[-n-1])
\]
and then to apply Theorem \ref{bomb}. The Chevalley--Eilenberg construction can be performed over $A$, but we must take its derived version (because the functor $M \rightarrow \mathrm{S}_A(M)$ does not respect weak equivalences). Concretely, assuming that $A$ is cofibrant, the Chevalley Eilenberg is defined in the usual way on $A$-dg-Lie algebras that are projective as $A$-modules. Since every element in $\lie_A$ is isomorphic to a dg-Lie algebra whose underlying $A$-module is projective, we get a
functor $\mathrm{CE}^{\bullet}_A \colon \LIE_A^{op} \rightarrow \CDGA_{A/A}$, and it admits a left adjoint $\mathfrak{D}_A$. We consider again a diagram of the form
\[
\xymatrix@C=100pt@R=50pt{
\CDGA_{A/A} \ar@<-2pt>[d]_-{\mathcal{L}_{A}} \ar@<2pt>[r]^-{\mathfrak{D}_A}&   \ar@<2pt>[l]^-{\mathrm{CE}^{\bullet}_A} \LIE_{A}^{op} \ar@<-2pt>[d]_-{\textit{forget}} \\
\MOD_{A}\ar@<-2pt>[u]_-{V \rightarrow A \oplus V[-1]} \ar@<2pt>[r]^-{\mathrm{Hom}_A(-, A)}  & \MOD_{A}^{op}  \ar@<2pt>[l]^-{\mathrm{Hom}_A(-, A)} \ar@<-2pt>[u]_-{\textit{free}}
}
\]
where $\mathfrak{L}_A(R)=\mathbb{L}_{A/R}$ and try to construct a morphism between weak Koszul duality contexts. This works exactly as in the proof of Proposition \ref{opinel}, but some aspects have to be taken care of in condition \textbf{(C)} of morphisms between weak Koszul duality contexts: the fact that theforgetful functor commutes with with sifted limits is explained in the proof of \cite[prop. 1.2.2]{He}. On the other hand, the cohomological boundedness of $A$ is crucial, otherwise the forgetful functor (although conservative) wouldn't send good objects to reflexive objects. 
More precisely, even if Lemma \ref{verygoodlemma} still holds in this context\footnote{A very good dgla over $A$ is a good dgla over $A$ such that 
\begin{itemize}
\item[--] the underlying $A$-module is projective. 
\item[--] the underlying graded Lie algebra is freely generated over $A$ by finitely many generators in positive degree. 
\end{itemize}}, very good dglas over $A$ are not necessarily reflexive as $A$-modules (for instance if $A$ is not a bounded $\kk$-algebra, then $\textit{free}\,\, A[-2]$ is not quasi-isomorphic to its double dual as a complex of $A$-modules). But they are whenever $A$ is bounded. 
\end{proof}

All results of the previous Section remain true if one replaces $\kk$ with a bounded $A$. For instance, we have the following analog of Proposition \ref{prop-tangent-repr}: 
\begin{Prop}
Let $\underline{B}$ be a representable element in $\bold{FMP}_{A/A}$. Then its tangent complex $\mathbb{T}_{\underline{B}}$, considered in $\MOD_A$, is equivalent to $\mathbb{T}_{B/A}\otimes_BA$, with $\mathbb{T}_{B/A}:=(\mathbb{L}_{B/A})^\vee$. $\Box$
\end{Prop}
In particular if $B=A \otimes_{\kk} A$, with $A$-augmentation being given by the product, then 
$$
\mathbb{L}_{A \otimes_{\kk} A/A}\otimes_{A \otimes_{\kk} A}A\simeq \mathbb{L}_{A}
$$
and thus $\mathbb{T}_{\underline{{A \otimes_{\kk} A}}}\simeq\mathbb{T}_{A}$. 
\begin{Cor} \label{swan}
The $A$-module $\mathbb{T}_A[-1]$ is a Lie algebra object in $\MOD_A$. 
\end{Cor}
In \cite{He}, Hennion proves global versions of the above results, and shows in particular that if $X$ is an algebraic derived stack locally of finite presentation, then $\mathbb{T}_X[-1]$ is a Lie algebra object in $\mathbf{QCoh}(X)$. This again explains (and generalizes) a phenomenon that we mentioned in the introduction. 

\subsection{DG-Lie algebroids}

In this section, we introduce the notion of dg-Lie algebroids, which is the dg-enriched version of Lie algebroids. This will be used to construct a Koszul duality context for $A$-augmented $k$-algebras in the next section. Informally, a (dg)-Lie algebroid is a Lie algebra $L$ over $\kk$ such that $L$ is an $A$-module, $A$ is a $L$-module, both structures being compatible. More precisely: 
\begin{Def}
A \textit{$(\kk, A)$-dg-Lie algebroid} is the data of a dgla $L$ over $\kk$ endowed with an $A$-module structure, as well as an action of $L$ on $A$, satisfying the following conditions: 
\begin{itemize}
\item[--] $L$ acts on $A$ by derivations, meaning that the action is given by a $A$-linear morphism of $\kk$-dglas $\rho \colon L \rightarrow \mathrm{Der}_{\kk}(A)$, called the \textit{anchor map}.
\item[--] The following Leibniz type rule holds for any $a\in A$ and any $\ell_1,\ell_2\in L$: 
\[
[\ell_1\,,\,a\ell_2]=(-1)^{|a| \times |\ell_1|} \, a[\ell_1\,,\,\ell_2] + \rho(\ell_1)(a)\ell_2
\]
\end{itemize}
Morphisms of $(\kk, A)$-dg-Lie algebroids are $A$-linear morphisms of dglas over $\kk$ commuting with the anchor map.
\end{Def}
\begin{Rem} $ $ \par
\begin{itemize}
\item[--] Every $A$-linear dgla defines a dg-Lie algebroid: it suffices to keep the same underlying object and to set the anchor map to zero. On the other hand, if $L$ is a $(\kk, A)$-dg-Lie algebroid, the kernel of the anchor map is a true $A$-linear dgla. These two constructions are adjoint. 
\item[--] Given a pair $(\kk, A)$, it is possible (\textit{see} \cite{kapranov_free_2007}) to attach to any object $V$ of $\mod_{A / \mathrm{Der}_{\kk}(A)}$ a free $(\kk, A)$-dg-Lie algebroid, denoted by $\textit{free}\, (V)$. The functor $V \rightarrow \textit{free}\,\,(V)$ is the left adjoint to the forgetful functor from $(\kk, A)$-dg-Lie algebroids to $A$-modules lying over $\mathrm{Der}_{\kk}(A)$. 
\item[--] The Chevalley--Eilenberg construction can be performed for $(\kk, A)$-dg-Lie algebroids. If $L$ is a $(\kk, A)$-dg-Lie algebroid, then 
the $A$-module underlying $\mathrm{CE}_{\kk/A}^{\bullet}(L)$ is the $A$-dual of $S_A(L[1])$ and the differential reads as follows (omitting signs):  
\begin{align*}
d_{\mathrm{CE}}\omega(\ell_0,\dots,\ell_n)= & \pm\sum_{i=0}^n\omega\big(\ell_0,\dots,d_L(\ell_i),\dots,\ell_n\big)\\
& \pm\sum_{i=0}^n\rho(\ell_i)\big(\omega(\ell_0,\dots\hat{\ell_i},\dots,\ell_n)\big)\pm\sum_{0\leq i<j\leq n}\omega([\ell_i,\ell_j],\ell_0,\dots\hat{\ell_i},\dots,\hat{\ell_j},\dots,\ell_n)\,.
\end{align*}
However, the same problem appearing for the Chevalley-Eilenberg complex of $A$-dg-Lie algebras happens: this functor must be derived. Hence, the above definition is set only on $(\kk, A)$-dg-Lie algebroids that are projective as $A$-modules (again $A$ is assumed to be cofibrant), and since every $(\kk, A)$-dg-Lie algebroid is quasi-isomorphic to one whose underlying $A$-module is projective, this defines a functor $\mathrm{CE}_{\kk/A}^{\bullet}$ from $(\kk, A)$ dg-Lie algebroids to $A$-augmented $\kk$-cdgas.
\item[--] If two cdgas $A$ and $A'$ are equivalent, then the $\infty$-categories $\LIE_A$ and $\LIE_{A'}$ are equivalent. It is impossible to expect this kind of result for $(\kk, A)$-dg-Lie algebroids, due to the presence of the derivations of $A$ (which is not an $\infty$-functor). This motivates the next forthcoming definition. 
\end{itemize}
\end{Rem}
\begin{Def}
A derived $(\kk, A)$-dg-Lie algebroid is a $(\kk, QA)$-dg-Lie algebroid, where $QA$ is a cofibrant replacement of $A$. We denote by $\lie_{\kk/A}$ the category of $(\kk, QA)$-dg-Lie algebroids.
\end{Def}
It can be proved that the category $\lie_{\kk/A}$ has a naturel model structure obtained by transferring the model structure of $\mod_{QA/\mathrm{Der}_\kk(QA)}$ 
via the adjunction 
\[
{\textit{free}} \colon \mod_{QA/\mathrm{Der}_\kk(QA)} \adjointHA \lie_{\kk/A} \colon \textit{forget}
\]
(\textit{see} \cite{Ve}). 
Hence we get an $\infty$-category $\LIE_{\kk/A}$ together with an adjunction 
\[
\MOD_{A/\mathbb{T}_{A}}\,\begin{matrix}\longrightarrow \\[-0.3cm] \longleftarrow\end{matrix}\,\LIE_{\kk/A}\,.
\]
where $\mathbb{T}_{A} \simeq \mathrm{Der}_{\kk}(QA)$ is the tangent complex of $A$. Fibrant $(\kk, A)$-dg-Lie algebroids are those with surjective anchor map.
\begin{Rem}[Corrigendum] \label{corrigendum}
Actually, the transferred structure is only a semi-model structure\footnote{The main problem is that the initial object $0$ is not fibrant.}, as shown in \cite{Nuit1}. One can consider the under-category $^{\mathfrak h/}\!\lie_{\kk/A}$, where $\mathfrak h$ is a fibrant-cofibrant replacement of the initial Lie algebroid $0$. It happens to be a genuine combinatorial model category that is obviously Quillen equivalent to $\lie_{\kk/A}$ (\textit{see} \cite[Remark 2.5]{Nuit2}). Hence, from our point of view everything is still fine, as the $\infty$-category $\LIE_{\kk/A}$ is indeed presentable. 
\end{Rem}
In the sequel, we will always assume that $A$ is cofibrant (which is possible after taking a cofibrant replacement). In this way, derived $(\kk,A)$-dg-Lie algebroids will be usual $(\kk,A)$-dg-Lie algebroids. 

Observe that $\LIE_{\kk/A}$ can be made into a dual deformation context: namely, we transfer the dual deformation context $\big(\MOD_{A/\mathbb{T}_A},(A[-n-1]\overset{0}\to\mathbb{T}_A)_n\big)$ from Example \ref{example-dualdefo-over} along the adjunction
\[
\textit{forget} \,\, \colon \LIE_{\kk/A}^{op} \colon\adjointHA  \MOD_{A/\mathbb{T}_A}^{op}  \colon  {\textit{free}}\,.
\]

\subsection{Formal moduli problems under $Spec(A)$}

In this section, we are looking at a more general situation as in \S \ref{easy}: we look at $A$-augmented $\kk$-algebras, where $A$ is a bounded cdga in non-positive degrees, which we assume to be cofibrant. 
Recall from Remark \ref{remalgA} that we have a deformation context $\big(\CDGA_{\kk/A},(A \oplus A[n])_n\big)$. The aim of this section is to give a description of formal moduli problems for $A$-augmented cdgas over $\kk$. 
The motto behind what follows will be the following one:
\par \medskip
\begin{center}
\textit{Going from $A$-augmented $A$-algebras to $A$-augmented $\kk$-algebras corresponds via Koszul duality to go from Lie algebras over $A$ to $(\kk, A)$ Lie algebroids}. 
\end{center}
\par \medskip
\noindent{}The precise result we want to prove is: 
\begin{Thm}[Nuiten \cite{Nuit2}]\label{thm-tooooop}
Given a cdga $A$ as above, there is a Koszul duality context 
\[
\big(\CDGA_{\kk /A}, (A \oplus A[n])_n\big) \adjointHA \big(\LIE_{\kk/A}^{op}, (\textit{free} \,(A[-n-1]\overset{0}{\to}\mathbb{T}_A)_n)
\]
This in particular implies that there is an equivalence 
\[
\bold{FMP}(\CDGA_{\kk/A}, A \oplus A[n]) \simeq \LIE_{\kk/A}\,.
\]
\end{Thm}

\begin{proof}
The first step consists in building the weak Koszul duality context. A reasonable candidate for the right adjoint functor is the Chevalley--Eimenberg functor 
\[
\mathrm{CE}^{\bullet}_{\kk/A}\colon\LIE_{\kk/A}^{op}\longrightarrow \CDGA_{\kk/A}\,.
\]
Here are a few properties of the Chevalley--Eilenberg functor that carry on to the Lie algebroid setting: 
\begin{itemize}
\item[--] $\mathrm{CE}^\bullet_{\kk/A}(L)\simeq \mathrm{RHom}_{U(A,L)}(A,A)$, where $U(A,L)$ is the universal envelopping algebra of $L$ (\textit{see e.g.} \cite{Rine} for the definition of $U(A,L)$)\footnote{In \cite{Rine}, $(\kk,A)$-Lie algebroids are named $(\kk,A)$-Lie algebras, and everything that is done makes sense in the differential graded setting as well.}. 
\item[--] To every $M\overset{f}{\to}\mathbb{T}_A=\mathrm{Der}(A)$ one can associate a derivation $d_f:A\to M^\vee$ and thus a \textbf{non-split} square zero extension $A\underset{d_f}{\oplus}M^\vee[-1]$. It is a fact that $\mathrm{CE}^\bullet_{\kk/A}\big(\textit{free}\,(f)\big)\simeq A\underset{d_f}{\oplus}M^\vee[-1]$. 
\item[--] $\mathrm{CE}^\bullet_{\kk/A}$ commutes with small limits (viewed as a functor on $\LIE_{\kk/A}^{op}$) and thus, since $\LIE_{\kk/A}$ is presentable, admits a left adjoint, that we denote by $\mathfrak{D}_{\kk/A}$. 
\end{itemize}
Thus we have a weak Koszul duality context 
\[
\mathfrak{D}_{\kk/A}\colon\big(\CDGA_{\kk /A}, (A \oplus A[n])_n\big) \adjointHA \big(\LIE_{\kk/A}^{op}, (\textit{free} \,(A[-n-1]\overset{0}{\to}\mathbb{T}_A)_n)\colon\mathrm{CE}^\bullet_{\kk/A}
\]
The next step is to prove that it is a Koszul duality context, by following the strategy of Proposition \ref{opinel}. We first consider the diagram 
\[
\xymatrix@C=80pt@R=40pt{
\CDGA_{\kk /A} & \ar[l]_-{\mathrm{CE}^{\bullet}_{\kk/A}}  \LIE_{\kk/A}^{op} \\
^{\mathbb{L}_A/}\!\MOD_A \ar[u]_-{(\mathbb{L}_A\overset{d}{\to}M)\mapsto A\underset{d}{\oplus}M[-1]} & \ar[l]_-{(-)^\vee}  (\MOD_{A/\mathbb{T}_A})^{op} \ar[u]_-{\textit{free}}
}
\]
that one can fill everywhere with left adjoints, giving rise to the following weak morphism of weak Koszul duality contexts:
\[
\xymatrix@C=80pt@R=40pt{
\CDGA_{\kk /A}  \ar@<-2pt>[d]_-{\mathfrak{L}: (B\to A) \mapsto (\mathbb{L}_A\to \mathbb{L}_{A/B})} \ar@<2pt>[r]^-{\mathfrak{D}_{\kk/A}} & \ar@<2pt>[l]^-{\mathrm{CE}^{\bullet}_{\kk/A}}  \LIE_{\kk/A}^{op} \ar@<-2pt>[d]_-{\textit{forget}} \\
^{\mathbb{L}_A/}\!\MOD_A   \ar@<2pt>[r]^-{(-)^\vee}  \ar@<-2pt>[u]_-{d\mapsto A\underset{d}{\oplus}M[-1]} & \ar@<2pt>[l]^-{(-)^\vee}  (\MOD_{A/\mathbb{T}_A})^{op} \ar@<-2pt>[u]_-{\textit{free}}
}
\]
Knowing, from Example \ref{Koszul-mod-over-under}, that the bottom weak Kozul duality context is actually a Koszul duality context, according to Proposition \ref{suppositoire} it remains to prove that this weak morphism is a morphism. The functor $\textit{forget}$ is conservative, and it is proved in \cite{Nuit1} that it preserves sifted limits. Let us now show that, if $L$ is good, then it is equivalent to a \textit{very good} dg-Lie algebroid: \textit{i.e.} a good dg-Lie algebroid that is of the form $\mathfrak g\overset{0}{\to} \mathbb{T}_A$ for a very good dgla $\mathfrak g$ over $A$. 
\begin{Lem}
Any good dg-Lie algebroid $L$ is quasi-isomorphic to a very good one. 
\end{Lem}
\begin{proof}[Proof of the lemma]
We first observe that $0$ is very good. We then proceed by induction: assume that $L$ is very good, and consider a dg-Lie algebroid $L'$ obtained by a pushout 
\[
\xymatrix{
\textit{free}\,(A[-n-1]\overset{0}{\to}\mathbb{T}_A)   \ar[r] \ar[d] & L \ar[d] \\
0 \ar[r] & L'
}
\]
in $\LIE_{\kk/A}$. Since $\textit{free}\,(A[-n-1]\overset{0}{\to}\mathbb{T}_A)$ is cofibrant, then a morphism $\textit{free}\,(A[-n-1]\overset{0}{\to}\mathbb{T}_A)\to L$ in $\LIE_{\kk/A}$ is represented by a morphism 
$\textit{free}\,(A[-n-1]\overset{0}{\to}\mathbb{T}_A)\to \widetilde{L}$ in $\lie_{\kk/A}$, where $\widetilde{L}$ is a chosen fibrant replacement of $L$. Now, we observe that one can choose $\widetilde{L}=L\oplus\mathfrak{h}$, where $\mathfrak{h}$ is a fibrant replacement of $0$, so that the replacement morphism $L\to \widetilde{L}$ splits. All in all, any morphism $\textit{free}\,(A[-n-1]\overset{0}{\to}\mathbb{T}_A)\to L$ in $\LIE_{\kk/A}$ can be represented by an actual morphism $\textit{free}\,(A[-n-1]\overset{0}{\to}\mathbb{T}_A)\to L$ in $\lie_{\kk/A}$. Next, a cofibrant replacement of the left vertical arrow is given by the morphism
\[
\textit{free}\,(A[-n-1]\overset{0}{\to} \mathbb{T}_A)\rightarrow \textit{free}\,\left\{\textrm{cone}\,(A \xrightarrow{\mathrm{id}} A)[-n-1]\overset{0}{\to}\mathbb{T}_A \right\}.
\]
Hence our pushout in $\LIE_{\kk/A}$ is represented by the following honest non-derived pushout in $\lie_{\kk/A}$: 
\[
\xymatrix{
\textit{free}\,(A[-n-1]\overset{0}{\to} \mathbb{T}_A)   \ar[r] \ar[d] & L \ar[d] \\
\textit{free}\,\left\{\textrm{cone}\,(A \xrightarrow{\mathrm{id}} A)[-n-1]\overset{0}{\to}\mathbb{T}_A)  \right\} \ar[r] & L'
}
\]
Now recall that, by induction, $L=(\mathfrak{g}\overset{0}{\to}\mathbb{T}_A)$ for some very good dgla $\mathfrak g$ over $A$. 
We finally observe that the functor 
\[
(-)\overset{0}{\to}\mathbb{T}_A\colon\lie_{A}\longrightarrow\lie_{\kk/A}
\]
commutes with colimits, so that $L'=(\mathfrak{g}'\overset{0}{\to}\mathbb{T}_A)$ for a very good dgla over $A$. 
\end{proof}
This in particular shows that, again under the boundedness assumption on $A$, $\textit{forget}$ sends good objects to reflexive objects; hence \textbf{(C)} holds. It remains to prove \textbf{(D)}. 

We consider the natural morphism $\theta_L$ defined as the following composition: 
\[
\mathfrak{L}\big(\mathrm{CE}^{\bullet}_{\kk/A}(L)\big)\to \mathfrak{L}\big(\mathrm{CE}^{\bullet}_{\kk/A}(\textit{free}\,\textit{forget}\,L)\big)\to \mathfrak{L}\big(A\underset{d_{\rho}}{\oplus}(\textit{forget}\,L)^\vee[-1]\big)\to (\textit{forget}\,L)^\vee\,, 
\]
where $\rho$ is the anchor map\footnote{In the above, $(\textit{forget}\,L)^\vee$ shall be understood as $\mathbb{L}_A\overset{\rho^\vee}{\to}L^\vee$. } and $\mathfrak{L}$ is the left adjoint to the twisted square zero extension functor given by the formula $\mathfrak{L}(B\to A) = (\mathbb{L}_A\to \mathbb{L}_{A/B}).$ Again, the cdga morphism 
\[
\mathrm{CE}^{\bullet}_{\kk/A}(L)\to \mathrm{CE}^{\bullet}_{\kk/A}(\textit{free}\,\textit{forget}\,L)\to A\underset{d_{\rho}}{\oplus}L^\vee[-1]
\]
is nothing but the obvious projection (onto the quotient by the square of the augmentation ideal whenever $L$ is very good). 

As in the proof of Proposition \ref{opinel} we consider the \textit{uncompleted} Chevalley--Eilenberg sub-cdga $\widetilde{\mathrm{CE}}^{\bullet}_{\kk/A}(L)\hookrightarrow \mathrm{CE}^{\bullet}_{\kk/A}(L)$. 
The quotient by the square of its augmentation ideal is again $A\underset{d_{\rho}}{\oplus}L^\vee[-1]$, and we have the following commuting diagram and its image through the left-most vertical adjunction of our square: 
\[
\xymatrix{ 
\widetilde{\mathrm{CE}}_{\kk/A}^{\bullet}(L) \ar[d]_{\iota_{L}}\ar[rd] & \\
\mathrm{CE}^{\bullet}_{\kk/A}(L) \ar[r] & A\underset{d_{\rho}}{\oplus}(\textit{forget}\,L)^\vee[-1]
}
\qquad
\xymatrix{
\mathfrak{L}\big(\widetilde{\mathrm{CE}}^{\bullet}_{\kk/A}(L)\big) \ar[d]_{\mathfrak{L}(\iota_{L})}\ar[rd]^{\widetilde{\theta}_{L}} & \\
\mathfrak{L}\big(\mathrm{CE}^{\bullet}_{\kk/A}(L)\big) \ar[r]_-{\theta_{L}} & (\textit{forget}\,L)^\vee
}
\]
In order to prove that $\theta_L$ is an equivalence when $L$ is good, it suffices to prove that both $\widetilde{\theta}_L$ and $\mathfrak{L}(\iota_L)$ are. 
Let us start with the following variation on Lemma \ref{augmentedlemma}: 
\begin{Lem}\label{2augmentedlemma}
Let $B$ be an $A$-augmented $\kk$-cdga, with augmentation ideal $\mathcal J$, that is cofibrant as a $\kk$-cdga, and such that the augmentation morphism is a fibration. 
Then $\mathcal{J}/\mathcal J^2[1]$ is naturally an $A$-module under $\mathbb{L}_A$, and there is an equivalence 
$\mathfrak{L}(B)\simeq(\mathbb{L}_A\to\mathcal{J}/\mathcal J^2[1])$ in $ ^{\mathbb{L}_A/}\!\MOD_A$. 
\end{Lem}
\begin{proof}[Proof of the lemma]
First of all, observe that the exact sequence
\[
0 \to \mathcal{J}/\mathcal{J}^2 \to B/\mathcal{J}^2 \to A \to 0
\]
splits in $\mod_{\kk}$. 
Hence there is a $\kk$-derivation $\delta :A\to\mathcal J/\mathcal J^2[1]$ such that $B/\mathcal J^2$ is isomorphic to $A\underset{\delta}{\oplus}\mathcal J/\mathcal J^2$ in $\cdga_{\kk/A}$. 
Hence, 
$$
\mathrm{Hom}_{\cdga_{\kk/A}}(B, B/\mathcal J^2) 
= \mathrm{Hom}_{\cdga_{\kk/A}}(B,  A \underset{\delta}{\oplus} \mathcal J/\mathcal J^2)
\simeq \mathrm{Hom}_{\, ^{\mathbb{L}_A/}\!\MOD_A}(\mathfrak{L}(B), \mathcal J/\mathcal J^2[1])\,.
$$
Therefore, the quotient morphism $B\to B/\mathcal J^2$ gives us a morphism $\mathfrak{L}(B)\to (\mathbb{L}_A\to \mathcal J/\mathcal J^2[1])$. 
To show that this is an equivalence it is sufficient to prove that the morphism $\mathbb{L}_{A/B}\to \mathcal J/\mathcal J^2[1]$ is an equivalence. Now, since $A$ and $B$ are cofibrant, and $B\to A$ is fibrant, 
$$
\mathbb{L}_{A/B}\simeq cone(\Omega^1_B\otimes_BA\to \Omega^1_A)\simeq \ker(\Omega^1_B\otimes_BA\to \Omega^1_A)[1]\,.
$$
The morphism $\mathbb{L}_{A/B}\to \mathcal J/\mathcal J^2[1]$ is thus modeled by the shift of the natural map 
$$
\ker(\Omega^1_B\otimes_BA\to \Omega^1_A)\to \mathcal J/\mathcal J^2\,,
$$
which is an isomorphism. 
%
\end{proof}
We then observe that 
\begin{itemize}
\item[--] the augmentation map $\widetilde{\mathrm{CE}}^{\bullet}_{\kk/A}(L)\to A$ is obviously a fibration. 
\item[--] when $L$ is very good (which we can always assume without loss of generality when dealing with good dg-Lie algebroids), then $\widetilde{\mathrm{CE}}^{\bullet}_{\kk/A}(L)$ is cofibrant. 
\end{itemize}
Thus $\widetilde{\theta}_{L}$ is an equivalence if $L$ is good. 
We again conclude with the very same flatness argument as in the proof of Proposition \ref{opinel} in order to get that $\mathfrak{L}(\iota_L)$ is an equivalence. 
\end{proof}

\subsection{The relative tangent complex}

We want to compute the underlying anchored module of the dg-Lie algebroid associated with a representable fmp under $Spec(A)$. 
We have an equivalence
$$
\Phi_{\kk/A}\colon\mathbf{FMP}(\CDGA_{\kk/A})\tilde{\longrightarrow}\LIE_{\kk/A}
$$
\begin{Prop}\label{prop-TAB}
Let $\underline{B}$ be the representable fmp associated with an $A$-augmented $\kk$-algebra $B$. Then the image of $\Phi_{\kk/A}(\underline{B})$ along the forgetful functor $\LIE_{\kk/A}\to \MOD_{A/\mathbb{T}_A}$ is $\mathbb{T}_{A/B}\to\mathbb{T}_A$. 
\end{Prop}
\begin{proof}[Sketch of proof]
Using Proposition \ref{functorialprop} together with Theorem \ref{thm-tooooop}, one sees that the image of $\Phi_{\kk/A}(\underline{B})$ along the forgetful functor is equivalent to the image of the representable fmp $\underline{{\mathbb{L}_A}\to \mathbb{L}_{A/B}}$ on $^{\mathbb{L}_A/}\!\MOD_A$ along the equivalence $\mathbf{FMP}(^{\mathbb{L}_A/}\!\MOD_A)\tilde\to\MOD_{A/\mathbb{T}_A}$. One easily sees that the image of this representable fmp is indeed 
$(\mathbb{L}_A\to \mathbb{L}_{A/B})^\vee=\mathbb{T}_{A/B}\to\mathbb{T_A}$. 
\end{proof}
According to the above proposition, the image of $\Phi_{\kk/A}(X)$ of an fmp under $Spec(A)$ along the forgetful functor $\LIE_{\kk/A}\to \MOD_{A/\mathbb{T}_A}$ deserves to be denoted $\mathbb{T}_{Spec(A)/X}$ and called the ``relative tangent complex of $Spec(A)$ over $X$''. 
\begin{Rem}
Unsurprisingly, an easy calculation allows to prove that the tangent complex to an fmp under $Spec(A)$ is 
$$
\mathbb{T}_X\simeq\textit{cofib}(\mathbb{T}_{Spec(A)/X}\to\mathbb{T}_A)\,, 
$$
where $\mathbb{T}_X$ can be viewed in $\MOD_A$ because $\LIE_{\kk/A}$ is $A$-linear. 

In the case $A=k$, we get back the Lie structure\footnote{A dg-Lie $(\kk,\kk)$-algebroid is nothing but a dgla. } on 
$$
\mathbb{T}_X[-1]\simeq\textit{fib}(0\to\mathbb{T}_X)=\textit{fib}(\mathbb{T}_\kk\to\mathbb{T}_X)\simeq \mathbb{T}_{Spec(\kk)/X} \,.
$$
\end{Rem}

Even if $A$ is not $\kk$ itself, we have a fiber sequence of Lie algebroids
$$
\mathbb{T}_X[-1]\to \mathbb{T}_{Spec(A)/X}\to\mathbb{T}_A
$$
in $\MOD_A$, where $\mathbb{T}_{Spec(A)/X}$ and $\mathbb{T}_A$ are $(\kk,A)$-dg-Lie algebroids, and $\mathbb{T}_{Spec(A)/X}\to\mathbb{T}_A$ is a morphism in $\LIE_{\kk/A}$. 
The forgetful functor $\LIE_{\kk/A}\to \MOD_{A/\mathbb{T}_A}$ being a right adjoint, there is a $(\kk,A)$-dg-Lie algebroid structure on $\mathbb{T}_X[-1]$ that allows to upgrade the above sequence into a fiber sequence of $(\kk,A)$-dg-Lie algebroids. One can then show that the functor ``taking the fiber of the anchor'' actually factors through dglas over $A$. Here is a question we would like to answer : 
\par \medskip
\begin{center}
\fbox{\textit{What is the split formal moduli probem under $Spec(A)$ having $\mathbb{T}_X[-1]$ as associated dgla?}}
\end{center}

We consider the commuting diagram: 
\[
\xymatrix@C=80pt@R=40pt{
\CDGA_{\kk /A} & \ar[l]_-{\mathrm{CE}^{\bullet}_{\kk/A}}  \LIE_{\kk/A}^{op} \\
\CDGA_{A/A} \ar[u]_-{\textit{forget}} & \ar[l]_-{\mathrm{CE}^{\bullet}_A}  \LIE_{A}^{op} \ar[u]_-{\mathfrak{g}\mapsto(\mathfrak{g}\overset{0}{\to}\mathbb{T}_A)}
}
\]
It can be filled with left adjoints in the following manner: 
\[
\xymatrix@C=80pt@R=40pt{
\CDGA_{\kk /A}  \ar@<-2pt>[d]_-{V \rightarrow V \otimes_{\kk} A} \ar@<2pt>[r]^-{\mathfrak{D}_{\kk/A}} & \ar@<2pt>[l]^-{\mathrm{CE}^{\bullet}_{\kk/A}}  \LIE_{\kk/A}^{op} \ar@<-2pt>[d]_-{\Xi} \\
\CDGA_{A/A}   \ar@<2pt>[r]^-{\mathfrak{D}_{A}}  \ar@<-2pt>[u]_-{\textit{forget}} & \ar@<2pt>[l]^-{\mathrm{CE}^{\bullet}_A}  \LIE_{A}^{op} \ar@<-2pt>[u]_-{\mathfrak{g}\mapsto(\mathfrak{g}\overset{0}{\to}\mathbb{T}_A)}
}
\]
Here the right-most vertical adjunction is obtained from a Quillen adjunction: in particular $\Xi$ is the left\footnote{Recall that we are working with opposite categories: $\Xi:\LIE_{\kk/A}^{op}\to\LIE_A^{op}$. } derived functor of the functor given on models by the kernel of the anchor map. 
\begin{Rem}
Note that the forgetful functor 
\[
\CDGA_{A/A} \rightarrow \CDGA_{k/A}
\] 
maps $\CDGA_{A/A}^{sm}$ to $\CDGA_{k/A}^{sm}$; however it is not an equivalence: $\CDGA_{k/A}^{sm}$ is way bigger than $\CDGA_{A/A}^{sm}$. Typically, non-split square-zero extensions do not belong to $\CDGA_{A/A}^{sm}$. 
\end{Rem}
\begin{Lem}
The above square of adjunctions is a weak morphism of Koszul duality contexts. 
\end{Lem}
\begin{proof}
The only thing left to prove is condition \textbf{(B)}, which is obvious: indeed, $\textit{forget}\,(A\oplus A[n])$ is precisely $A\oplus A[n]$, viewed as a $\kk$-algebra. 
\end{proof}
Consequently, Proposition \ref{functorialprop} tells us that we have the following commuting diagram 
\[
\xymatrix@C=40pt@R=20pt{
\LIE_{\kk/A} \ar[r]^-{\Psi_{\kk/A}} \ar[dd]^{\Xi}& \bold{FMP}_{\kk/A} \ar[dd]^-{\circ \textit{forget}} \ar[rd]^-{\mathbb{T}}& \\
&&\MOD_A \\
\LIE_A \ar[r]^-{\Psi_A} & \bold{FMP}_{A/A} \ar[ru]_-{\mathbb{T}}&
}
\]
Let $X$ be a formal moduli problem under $Spec(A)$, and let $L_X:=\Phi_{\kk/A}(X)$ be its \textit{relative tangent Lie algebroid}. We then have: 
$$
\Xi(L_X)\simeq \mathfrak g_{X\circ\textit{forget}}\,.
$$
Hence the Lie dgla structure on $\mathbb{T}_X[-1]$ obtained by taking the fiber of the anchor map of the relative tangent Lie algebroid $L_X$ coincide with the one given on $\mathbb{T}_{X\circ\textit{forget}}$. 
This answers our question: the split formal moduli problem under $Spec(A)$ having $\Xi(L_X)$ as tangent Lie algebra is $X\circ\textit{forget}$. 
\begin{Rem}\label{heuristicremark}
Geometrically, and following the intuitive ideas presented in the introduction, one shall understand $X$ as a formal thickening $Spec(A)\to X$ of $Spec(A)$. 
The geometric interpretation of $X\circ\textit{forget}$ is then as the split formal thickening $Spec(A)\to Spec(A)\times X\to Spec(A)$ given by the graph\footnote{Actually, the formal neighborhood of the graph. } of the previous one. 
Making this idea more precise is a bit complicated as $Spec(A)$ is initial for the two $\infty$-categories of fmp's that we are considering. 
Hence $Spec(A)\times-$ does nothing, and shall be understood as composing with the forgetful functor. 
\end{Rem}

\subsection{Generalizations}\label{3.5}

Before going to the last, more geometric, Section of this survey, let us mention a few unnecessary assumptions that we have made for the sake of exposition: 
\begin{itemize}
\item[--] one can replace $\kk$ with any cdga sitting (cohomologically) in non-positive degree and containing $\mathbb{Q}$ (as a sub-cdga). 
\item[--] one can replace $A$ with any $\kk$-cdga with a reflexive cotangent complex $\mathbb{L}_{A/\kk}$ in $\MOD_A$. 
\item[--] one can replace $\kk$ and $A$ with presheaves of such cdgas on some given $\infty$-category. 
\end{itemize}
\begin{Rem}\label{comparewithcptvv}
One could ask if all this still works for commutative algebra objects in a gentle enough stable model category $M$, e.g.~satisfying the standing assumptions of \cite[\S1.1]{CPTVV}. 
Ideally, we would like to be able to allow $M$ to be the category of graded mixed modules (\textit{see} \S\ref{appendix-A.1}), and cover the following situation: $\kk=\mathbb{D}_{X_{DR}}$ is the mixed crystalline structure sheaf of a derived Artin stack $X$ and $A=\mathcal{B}_X$ is the mixed principal parts of $X$ (after \cite[Definition 2.4.11]{CPTVV}), which are graded mixed cdgas (\textit{see} \S\ref{appendix-A.1}, again). We will come back to the comparison with the \cite{CPTVV} in the next Section. 
\end{Rem}
In particular, let $\kk\to B\to A$ be a sequence of cdgas, with $\kk$ and $B$ sitting cohomologically in non-positive degree, and both $\mathbb{L}_{A/B}$ and $\mathbb{L}_{A/\kk}$ reflexive in $\MOD_A$. One can show in a very similar fashion that the weak morphism of Koszul duality contexts from the previous subsection generalizes as follows: 
\[
\xymatrix@C=80pt@R=40pt{
\CDGA_{\kk /A}  \ar@<-2pt>[d]_-{V \rightarrow V \otimes_B A} \ar@<2pt>[r]^-{\mathfrak{D}_{\kk/A}} & \ar@<2pt>[l]^-{\mathrm{CE}^{\bullet}_{\kk/A}}  \LIE_{\kk/A}^{op} \ar@<-2pt>[d]_-{-\underset{\mathbb{T}_{A/\kk}}{\times}\mathbb{T}_{A/B}} \\
\CDGA_{B/A}   \ar@<2pt>[r]^-{\mathfrak{D}_{B/A}}  \ar@<-2pt>[u]_-{\textit{forget}} & \ar@<2pt>[l]^-{\mathrm{CE}^{\bullet}_{B/A}}  \LIE_{B/A}^{op} \ar@<-2pt>[u]_-{(L\to \mathbb{T}_{A/B})\mapsto(L\to\mathbb{T}_{A/\kk})}
}
\]
This leads (again after Proposition \ref{functorialprop}) to a commuting square 
\[
\xymatrix@C=40pt@R=20pt{
\LIE_{\kk/A} \ar[r]^-{\Psi_{\kk/A}} \ar[dd]^{-\underset{\mathbb{T}_{A/\kk}}{\times}\mathbb{T}_{A/B}}& \bold{FMP}_{\kk/A} \ar[dd]^-{\circ \textit{forget}} \ar[rd]^-{\mathbb{T}}& \\
&&\MOD_A \\
\LIE_{B/A} \ar[r]^-{\Psi_{B/A}} & \bold{FMP}_{B/A} \ar[ru]_-{\mathbb{T}}&
}
\]
If $X$ lies in $\bold{FMP}_{\kk/A}$, then it has a relative tangent Lie algebroid $\mathbb{T}_{Spec(A)/X}\to\mathbb{T}_{A/\kk}$. 
It follows from the above that the relative tangent Lie algebroid\footnote{Note that we slightly abuse the notation here. As we have seen in Proposition \ref{prop-TAB}, $\mathbb{T}_{A/B}$ is the underlying $A$-module of the $(\kk,A)$-dg-Lie algebroid $\Psi_{\kk/A}(\underline{B})$. } $\mathbb{T}_{Spec(B)/X\circ\textit{forget}}$ of $X\circ\textit{forget}$, viewed as a $(\kk,A)$-dg-Lie algebroid using the Lie algebroid morphism $\mathbb{T}_{A/B}\to\mathbb{T}_{A/\kk}$, is equivalent to the pull-back 
$$
\mathbb{T}_{Spec(A)/X}\underset{\mathbb{T}_{A/\kk}}{\times}\mathbb{T}_{A/B}\,.
$$
In the spirit of Remark \ref{heuristicremark}, one geometrically interprets $X\circ\textit{forget}$ as the following ``$Spec(B)$-split formal thickening of $Spec(A)$'': 
$$
Spec(A)\to Spec(B)\times X\,.
$$
\begin{Rem}
One can actually show that $\LIE_{B/A}\simeq (\LIE_{\kk/A})_{/\mathbb{T}_{A/B}}$. 
In particular, this shows that $\bold{FMP}_{B/A}\simeq (\bold{FMP}_{\kk/A})_{/\underline{B}}$. 

This shall be interpreted geometrically as follows. If we have a $Spec(B)$-split formal thickening of $Spec(A)$, then the map $X\to Spec(B)$ uniquely factors through $X\to\underline{B}$, which is terminal among such formal thickenings (being the formal completion of the map $Spec(A)\to Spec(B)$). 
\end{Rem}

There is yet another geometric situation one may want to understand in terms of Lie algebroids: given a fmp $Spec(B)\to X$ under $Spec(B)$, one can look at it under $Spec(A)$ by just composing\footnote{And, maybe, completing afterwards. } with $Spec(A)\to Spec(B)$. On the level of anchored modules, we have a functor 
\begin{eqnarray*}
\MOD_{B/\mathbb{T}_{B/\kk}} & \longrightarrow & \MOD_{A/\mathbb{T}_{A/\kk}} \\
L & \longmapsto & L\otimes_BA\underset{\mathbb{T}_{B/\kk}\otimes_BA}{\oplus}\mathbb{T}_{A/\kk}\,.
\end{eqnarray*}
Assume that $B\to A$ is a cofibration betwen cofibrant cdgas. Hence the morphism $\mathbb{T}_{B/\kk}\otimes_BA\to\mathbb{T}_{A/\kk}$ in $\MOD_A$ is represented by $d\phi:\mathrm{Der}_\kk(B)\otimes A\to\mathrm{Der}_\kk(A)$ in $\mod_A$, which is moreover a fibration. Hence the above fiber product in $\MOD_A$ can be computed as an ordinary fiber product in $\mod_A$: more precisely, it is $\ker(\rho\otimes_BA-d\phi)$, where $\rho:L\to\mathrm{Der}_\kk(B)$ is the anchor map. In this situation, it has been shown by Higgins and Mackenzie \cite{HM} that one can put a Lie bracket on $L$ that turns it into a Lie algebroid\footnote{In \cite{HM} the authors work in the differential setting and nothing is dg. Nevertheless, their construction carry on without any problem. Their only technical assumption is that the map along which they pull-back is submersive, which is fine in our setting thanks to the cofibrancy of $B\to A$. }. We therefore get a functor $\LIE_{\kk/B}\to \LIE_{\kk/A}$ that actually factors as 
$$
\LIE_{\kk/B}\longrightarrow\, ^{\mathbb{T}_{A/B}/}\!\LIE_{\kk/A} \longrightarrow \LIE_{\kk/A}
$$

The main upshot of this sketchy discussion is that if we are given a commuting square of cdgas
\[
\xymatrix{
B_1 \ar[r] \ar[d] & B_2 \ar[d] \\
A_1 \ar[r] & A_2
}
\] 
where $A_1$ and $B_1$ cohomologically sit in non-positive degree and contain $\mathbb{Q}$, and $\mathbb{L}_{A_2/B_1}$ and $\mathbb{L}_{A_2/A_1}$ are reflexive in $\MOD_{A_2}$, then we have a pull-back functor 
$$
\LIE_{B_1/B_2}\to \LIE_{B_1/A_2}\to\LIE_{A_1/A_2}\,.
$$

\section{Global aspects}

\subsection{Formal derived prestacks and formal thickenings}

We start with several definitions and statements, mainly extracted from \cite{CPTVV,GR1,GR2}. 

Let us denote by $\mathbf{dAff}_\kk^{f.p.}$ the opposite $\infty$-category of non-positively graded $\kk$-cdgas $A$ that are \textit{almost finitely presented}: $H^0(A)$ is finitely generated as a $\kk$-algebra and $H^i(A)$ is a finitely presented $H^0(A)$-module. We then define the $\infty$-category $\mathbf{dPrSt}_\kk^{f.p.}$ of \textit{locally almost finitely presented} derived prestacks over $\kk$ to be the $\infty$-category of presheaves on $\mathbf{dAff}_\kk^{f.p.}$. 

We say that a locally almost finitely presented derived prestack $F$ is \textit{nilcomplete} (\textit{convergent} in the terminology of \cite{GR1}) if the canonical map 
$$
F(A)\to \underset{n}{\lim}F(A_{\leq n})
$$
is an equivalence, where $A_{\leq n}$ denotes the $n$-th Postnikov truncation of $A$. 

The full sub-$\infty$-category $\mathbf{dPrSt}_\kk^{f.t.}$ of $\mathbf{dPrSt}_\kk^{f.p.}$ spanned by nilcomplete prestacks\footnote{The superscript ``f.t.'' stands for \textit{(locally almost of) finite type}. } is equivalent to the essential image of the right Kan extension from the full sub-$\infty$-category $\mathbf{dAff}_\kk^{f.t.}\subset \mathbf{dAff}_\kk^{f.p.}$ spanned by eventually coconnective cdgas (\textit{see} \cite[Proposition 1.4.7 of Chapter 2]{GR1}), where a cdga $A$ is called \textit{eventually coconnective} if $H^n(A)=0$ for $n<\!<0$. 

A \textit{formal derived prestack} is a nilcomplete locally almost finitely presented derived prestack\footnote{Nilcomplete locally almost finitely presented prestacks are called said to be \textit{locally almost of finite type (laft)} in \cite{GR1}. } $F$ with the following to properties: 
\begin{itemize}
\item[--] $F$ is infinitesimally cohesive: for any cartesian square of almost finitely presented non-positively graded $\kk$-cdgas 
$$
\xymatrix{
A \ar[r]\ar[d] & A_1 \ar[d] \\
A_2 \ar[r] & A_0
}
$$
such that each $H^0(A_i)\to H^0(A_0)$ is surjective with nilpotent kernel, then the induced square 
$$
\xymatrix{
F(A) \ar[r]\ar[d] & F(A_1) \ar[d] \\
F(A_2) \ar[r] & F(A_0)
}
$$
is cartesian as well. 
\item[--] $F$ admits a procotangent complex. Here we say that $F$ admits a procotangent complex if it admits a procotangent complex $\mathbb{L}_{F,x}$ at every point $x:Spec(A)\to F$, and for every morphism of points 
$$
\xymatrix{
Spec(A) \ar[rr]^u \ar[rd]^x && Spec(B)\ar[ld]^y \\\
& F & 
}
$$
the induced morphism $u^*:\mathbb{L}_{F,y}\otimes_BA\to \mathbb{L}_{F,x}$ is an equivalence. We say that $F$ admits a procotangent complex at a given point $x:Spec(A)\to F$ if the functor $\mathbb{D}er_F(x,-)$ of derivations is co-prorepresentable. 
\end{itemize}

We denote by $\bold{FdPrSt}_\kk$ the $\infty$-category of formal derived prestacks over $\kk$. 

For a derived prestack $F$, we denote by $F_{red}$ its associated \textit{reduced prestack} $F_{red}:=\underset{Spec(A)\to F}{colim}Spec(A_{red})$, where $A_{red}$ denotes the quotient of $H^0(A)$ by its maximal nilpotent ideal.  We also consider the \textit{de Rham prestack} $F_{DR}$ of $F$, that is defined as follows: $F_{DR}(A):=F(A_{red})$. We now summarize the properties of reduced and de Rham prestacks: 
\begin{itemize}
\item[--] there are canonical maps $F_{red}\to F\to F_{DR}$, inducing equivalences $(F_{red})_{DR}\tilde\to F_{DR}$ and $F_{red}\tilde\to (F_{DR})_{red}$. This implies in particular that a morphism $F\to G$ induces an equivalence $F_{red}\tilde\to G_{red}$ between their  reduced prestacks if and only if it induces an equivalence  $F_{DR}\tilde\to G_{DR}$ between their de Rham prestacks. 
\item[--] $(-)_{red}$ is left adjoint to $(-)_{DR}$. 
\end{itemize}

For any formal derived prestack $X$, one can consider the sub-$\infty$-category $\mathbf{Thick}(X)$ of $^{X/}\!\bold{FdPrSt}$ spanned by those morphisms $X\to F$ that induce an equivalence $X_{red}\to F_{red}$ between the associated reduced prestacks. We call it the $\infty$-category of formal thickenings of $X$. 

\begin{Prop}[\cite{GR2}, Chapter 5, Proposition 1.4.2]\label{propositionGR}
Let $A$ be a non-positively graded cdga that is almost finitely presented and eventually coconnective (\textit{i.e.} $A$ is almost of finite type). Then there is an equivalence of $\infty$-categories 
$$
\bold{FMP}_{\kk/A}\simeq \mathbf{Thick}\big(Spec(A)\big)\,.
$$ 
\end{Prop}
\begin{proof}[Idea of the proof]
We have a pull-back functor $\mathbf{Thick}\big(Spec(A)\big)\to \bold{FMP}_{\kk/A}$ along 
\[
\CDGA_{\kk/A}^{sm}\to \CDGA_\kk^{f.t.}:=\big(\mathbf{dAff}_\kk^{f.t.}\big)^{op}\,,
\]
that consists in evaluating our prestack on $Spec(A)\to Spec(B)$ with $B$ a small $A$-augmented $\kk$-cdga. This pull-back functor has a left adjoint, given by a left Kan extension\footnote{Observe that the left Kan extension of a functor on small $A$-augmented cdgas is a priori just a derived prestack locally almost of finite type. It is an exercise to check that left Kan extensions of fmps are sent to formal thickenings. The only non-obvious property to check that it admits a procotangent complex, which follows from the prorepresentability of fmps. }. 
One can prove (using infinitesimal cohesiveness) that a formal thickening is completely determined by its restriction on $Spec(B)$'s for which $B_{red}\simeq A_{red}$. We conclude by observing that the functor $\CDGA_{\kk/A}^{sm}\to \CDGA_\kk^{f.t.}$ obviously factors through the category of those cdgas of finite type with reduced algebra equivalent to $A_{red}$, and that this has a left adjoint: sending a $B$ such that $B_{red}\simeq A_{red}$ to $A\times_{A_{red}}B$. 
\end{proof}
As a consequence, we have the following useful Corollary: 
\begin{Cor}\label{crazycor}
For any non-positively graded cdga $A$ almost of finite type, there is an equivalence 
$$
\bold{FMP}_{\kk/A}\simeq~^{\underline{A}/}\!\bold{FMP}_{\kk/A_{red}}\,.
$$
Moreover, for any formal derived prestack $X\to Spec(A)$ over $Spec(A)$ that exhibits $Spec(A)$ as a formal thickening of $X$, we have an equivalence 
$$
\mathbf{Thick}(X)_{/Spec(A)}\simeq ~^{X/}\!\bold{FMP}_{A/A_{red}}\,.
$$
\end{Cor}
\noindent In the second part of the Corollary, we abusively denote by the same symbol the formal derived prestack $X$ and its associated formal moduli problem under $Spec(A_{red})$. 
\begin{proof}
Let $X\to Y$ be a morphism of formal derived prestacks that induces an equivalence $X_{red}\tilde\to Y_{red}$ between their associated reduced prestacks. Then we obviously have an equivalence $\mathbf{Thick}(Y)\simeq~^{Y/}\!\mathbf{Thick}(X)$. 

To prove the first part of the Corollary, we begin with the obvious observation that the map $Spec(A_{red})\to Spec(A)$ is a formal thickening of $Spec(A_{red})$. Hence we have an equivalence 
$$
\mathbf{Thick}\big(Spec(A)\big)\simeq~^{Spec(A)/}\!\mathbf{Thick}\big(Spec(A_{red})\big)\,.
$$
Then we notice that the fmp $\underline{A}\in\mathbf{FMP}_{\kk/A_{red}}$ is the pullback of $Spec(A)$ along 
$\CDGA_{\kk/A_{red}}^{sm}\to \CDGA_\kk^{f.t.}$.  
We finally use the above Proposition to get a chain of equivalences 
$$
\bold{FMP}_{\kk/A}\simeq\mathbf{Thick}\big(Spec(A)\big)\simeq~^{Spec(A)/}\!\mathbf{Thick}\big(Spec(A_{red})\big)\simeq~^{\underline{A}/}\!\bold{FMP}_{\kk/A_{red}}\,.
$$
For the second part of the statement, we observe that $X\to Spec(A)$ being a formal thickening means that the induced map $X_{red}\to Spec(A_{red})$ is an equivalence. This in turns teaches us that $Spec(A_{red})\simeq X_{red}\to X$ is a formal thickening, and thus can be understood as a fmp under $Spec(A_{red})$. Hence we have  
$$
\mathbf{Thick}(X)\simeq~^{X/}\!\mathbf{Thick}\big(Spec(A_{red})\big)\simeq~^{X/}\!\mathbf{FMP}_{\kk/A_{red}}\,,
$$
As a consequence we get that $\mathbf{Thick}(X)_{/Spec(A)}\simeq~^{X/}\!(\mathbf{FMP}_{\kk/A_{red}})_{/\underline{A}}$. We conclude by observing that $(\mathbf{FMP}_{\kk/A_{red}})_{/\underline{A}}\simeq \mathbf{FMP}_{A/A_{red}}$ (\textit{see} Subsection \ref{3.5}). 
\end{proof}

\subsection{DG-Lie algebroids and formal thickenings}

The consequences of Corollary \ref{crazycor} are of particular interest. Let $B\to A$ be a morphism in $\CDGA_\kk^{f.t.}$, and assume that the induced morphism $B_{red}\to A_{red}$ is an isomorphism. 
\begin{Prop}\label{gplusdenom}
We have equivalences: 
\begin{itemize}
\item[--] $\bold{FMP}_{\kk/A}\simeq ~^{\mathbb{T}_{A_{red}/A}/}\!\LIE_{\kk/A_{red}}$. 
\item[--] $\bold{FMP}_{B/A}\simeq ~^{\mathbb{T}_{B_{red}/A}/}\!\LIE_{B/B_{red}}$. 
\end{itemize}
\end{Prop}
\begin{proof}
This is a direct consequence of (the proof of) Corollary \ref{crazycor}, Theorem \ref{thm-tooooop}, and Proposition \ref{prop-TAB}. 
Indeed, we have two chains of equivalences: 
$$
\bold{FMP}_{\kk/A}\simeq~^{\underline{A}/}\!\bold{FMP}_{\kk/A_{red}}\simeq ~^{\mathbb{T}_{A_{red}/A}/}\!\LIE_{\kk/A_{red}}
$$
and
$$
\bold{FMP}_{B/A}\simeq~^{\underline{A}/}\!\bold{FMP}_{B/B_{red}}\simeq ~^{\mathbb{T}_{B_{red}/A}/}\!\LIE_{B/B_{red}}\,.
$$
\end{proof}
In the rest of this subsection we will provide (sketches of) two alternative proofs that do not rely on formal thickenings and Proposition \ref{propositionGR}. 
We actually prove the following more general result: 
\begin{Thm}\label{thm-lastminute}
Consider a sequence of $\kk$-cdgas $K\to B\to A$, all cohomologically bounded and concentrated in non-positive degree, and with $B$ being small in $\CDGA_{\kk/A}$. Then we have an equivalence 
$$
\bold{FMP}_{K/B}\simeq~^{\underline{B}/}\!\bold{FMP}_{K/A}\,.
$$
\end{Thm}
Observe that $B$ is small in $\CDGA_{\kk/A}$ if and only if it is so in $\CDGA_{K/A}$. Also note that if $B$ is almost of finite type and $A=B_{red}$, then $B$ is small in $\CDGA_{\kk/B_{red}}$ and thus we recover in particular Proposition \ref{gplusdenom}. 
\begin{proof}[Sketch of proof of the Theorem]
We first observe that there is an adjunction 
$$
\iota \colon \CDGA_{K/B} \adjointHA \CDGA_{K/A} \colon p\,,
$$
where $\iota(C)=C$ and $p(C)=C\underset{A}{\times}B$. 

Since $B$ is small, and $p$ commutes with limits (being a right adjoint), then $p$ preserves small morphisms and $p^*:Fun(\CDGA_{K/B},\SSET)\to Fun(\CDGA_{K/A},\SSET)$ sends fmps to fmps. Moreover, for every $X\in \bold{FMP}_{K/B}$, there is a natural morphism $\underline{B}\to p^*X$:  
$$
\underline{B}=Hom_{\CDGA_{K/A}}(B,-)\simeq Hom_{\CDGA_{K/A}}(B,-)\times X(B)\to X(-\underset{A}{\times}B)=p^*X\,.
$$
Hence we obtain a functor $\bold{FMP}_{K/B}\to~^{\underline{B}/}\!\bold{FMP}_{K/A}$, which we abusively still denote $p^*$. 

This functor has a right adjoint $p_*$, which we describe now. 
Since $B$ is small, and $\iota$ commutes with pull-backs, then $\iota$ preserves small morphisms. 
But it does not preserves fmps, as it doesn't preserve terminal objects. 
To cure this we define, for an $X\in~^{\underline{B}/}\!\bold{FMP}_{K/A}$, 
$$
p_*X:=\iota^*X\underset{X(B)}{\times}*\in\bold{FMP}_{K/B}\,,
$$
where $*\to X(B)$ is given by $\{id_B\}\to \underline{B}(B)\to X(B)$. 

Let us prove that the adjoint pair $(p^*,p_*)$ is actually an equivalence. 
We start by describing the unit of the adjunction: for every $X\in\bold{FMP}_{K/B}$, we have 
$$
p_*p^*X	=				\iota^*p^*X\underset{p^*X(B)}{\times}*
				\simeq 	X(-\underset{A}{\times}B)\underset{X(B\underset{A}{\times}B)}{\times}X(B)\,.
$$
Hence there is an obvious natural morphism $X\to p_*p^*X$. 
We claim that it is an equivalence. 
Indeed, for every small $C\to B$, the square 
$$
\xymatrix{
C \ar[r]\ar[d] & B \ar[d] \\
C\underset{A}{\times}B \ar[r] & B\underset{A}{\times}B
}
$$
is cartesian, and its bottom horizontal morphism is small. 
Therefore $X$ preserves this cartesian square, and thus the unit morphism $X\to p_*p^*X$ is an equivalence. 

We now turn to describing the counit of the adjunction. 
For every $X\in~^{\underline{B}/}\!\bold{FMP}_{K/A}$, we have 
$$
p^*p_*X = p_*X(-\underset{A}{\times}B)=X(-\underset{A}{\times}B)\underset{X(B)}*\,,
$$
and the counit is the composed morphism $p^*p_*X\to X(-\underset{A}{\times}B)\to X$. 
Thus we have, for every $C\to A$, a composition of pull-back squares 
$$
\xymatrix{
p^*p_*X(C) \ar[r]\ar[d] & {*} \ar[d] \\
X(C\underset{A}{\times}B) \ar[r]\ar[d] & X(B) \ar[d]\\
X(C) \ar[r] & X(A)\simeq*
}
$$
This shows that the counit morphism $p^*p_*X(C)\to X(C)$ is an equivalence, and ends the proof of the Theorem. 
\end{proof}

In the remainder of this subsection, we work under the hypothesis of Theorem \ref{thm-lastminute}. 
As a consequence of the Theorem we get (using Theorem \ref{thm-tooooop} and Proposition \ref{prop-TAB}) an equivalence
$$
\LIE_{K/B}\simeq ~^{\mathbb{T}_{A/B}/}\!\LIE_{K/A}\,.
$$
The functor $\LIE_{K/B}\to ~^{\mathbb{T}_{A/B}/}\!\LIE_{K/A}$ can be proven to be the derived version of Higgins--Mackenzie pull-back of Lie algebroids (\textit{see} the end of \S\ref{3.5}). 
The inverse functor is not so easy to describe purely in terms of Lie algebroids, but there is nevertheless a nice way of directly proving that there is an equivalence 
$$
\bold{FMP}_{K/B}\simeq ~^{\mathbb{T}_{A/B}/}\!\LIE_{K/A}
$$
that we explain now. Our strategy is (again) as follows: 
\begin{itemize}
\item[--] we first show that there is a weak Koszul duality context involving $\CDGA_{\kk/B}$ and $^{\mathbb{T}_{A/B}/}\!\LIE_{K/A}$. 
\item[--] we then show that it comes with a weak morphism to a Koszul duality context having the form of Example \ref{Koszul-mod-over-under}. 
\item[--] we finally claim that this weak morphism is a morphism, showing that the original weak Koszul duality context is a Koszul duality context (after Proposition \ref{suppositoire}). 
\item[--] we conclude using Theorem \ref{bomb}. 
\end{itemize}

\paragraph{The weak Koszul duality context}~\\
We begin with the following rather evident result: 
\begin{Lem}\label{lem-cptvv}
We have that $\mathfrak{D}_{K/A}(B\to A)\simeq \mathbb{T}_{A/B}$ and $CE^\bullet_{K/A}(\mathbb{T}_{A/B})\simeq B$. 
In particular, we have an adjunction 
$$
\mathfrak{D}_{K/A}\colon\CDGA_{K/B} \adjointHA \Big(^{\mathbb{T}_{A/B}/}\!\LIE_{K/A}\Big)^{op}\colon CE^\bullet_{K/A}\,.
$$
\end{Lem}
\begin{proof}
The first part of the statement is just Proposition \ref{prop-TAB}, and the second part follows from the fact that $B$ is small. 
Finally, one concludes using Theorem \ref{thm-tooooop} and the obvious observation that $\CDGA_{K/B}\simeq(\CDGA_{K/A})_{/B\to A}$.  
\end{proof}
The deformation context is still the one from Example \ref{exa-kA}: $\big(\CDGA_{K/B},(B\oplus B[n])_n\big)$, and the dual deformation context is $^{\mathbb{T}_{A/B}/}\!\LIE_{K/A}$ equipped with the spectrum object 
$$
\mathbb{T}_{A/B}\oplus \textit{free}\,(A[-n-1]\overset{0}{\to}\mathbb{T}_{A/K})\,.
$$
for its opposite $\infty$-category. 
\begin{Rem}
The reader can check directly that this is indeed a spectrum object, but this will also be a consequence of our construction of the weak morphism below (\textit{see} Lemma \ref{lem-en-retard}). 
\end{Rem}
Together, the deformation context and the dual deformation context define a weak Koszul duality context. Indeed: 
\begin{align*}
& \, CE^\bullet_{K/A}\big(\mathbb{T}_{A/B}\oplus \textit{free}\,(A[-n-1]\overset{0}{\to}\mathbb{T}_{A/K})\big) \\
\simeq & \, CE^\bullet_{K/A}(\mathbb{T}_{A/B})\underset{A}{\times}CE^\bullet_{A}\big(\textit{free}\,(A[-n-1])\big) \\
\simeq & \, B\underset{A}{\times}(A\oplus A[n]) \\
\simeq & \, B\oplus B[n]\,.
\end{align*}

\paragraph{The weak morphism}~\\
We already have the following adjunctions: 
\[
\xymatrix@C=80pt@R=40pt{
\CDGA_{K/B}  \ar@<-2pt>[d]_-{(C\to B) \mapsto (\mathbb{L}_{B/K}\to \mathbb{L}_{B/C})} \ar@<2pt>[r]^-{\mathfrak{D}_{K/A}} & \ar@<2pt>[l]^-{\mathrm{CE}^{\bullet}_{K/A}} \Big(^{\mathbb{T}_{A/B}/}\!\LIE_{K/A}\Big)^{op}  \\
^{\mathbb{L}_{B/K}/}\!\MOD_B   \ar@<2pt>[r]^-{(-)^\vee}  \ar@<-2pt>[u]_-{d\mapsto B\underset{d}{\oplus}M[-1]} & \ar@<2pt>[l]^-{(-)^\vee}  (\MOD_{B/\mathbb{T}_{B/K}})^{op} 
}
\]
\begin{Lem}\label{lem-en-retard}
There is an adjunction 
$$
\mathfrak{L}\colon\MOD_{B/\mathbb{T}_{B/K}} \adjointHA~^{\mathbb{T}_{A/B}/}\!\LIE_{K/A}\colon\mathfrak{R}
$$
such that 
$$
\mathfrak{L}(B[-n-1]\overset{0}{\to}\mathbb{T}_{B/K})\simeq \mathbb{T}_{A/B}\oplus \textit{free}\,(A[-n-1]\overset{0}{\to}\mathbb{T}_{A/K})\,.
$$
\end{Lem}
\begin{proof}[Sketch of proof]
We define $\mathfrak{L}$ as the composition of the following two functors: 
\begin{itemize}
\item[--] $\MOD_{B/\mathbb{T}_{B/K}}\to~^{\mathbb{T}_{A/B}/}\MOD_{A/\mathbb{T}_{A/K}}$, sending $M\to \mathbb{T}_{B/K}$ to 
$$
M':=M\otimes_{B}A\underset{\mathbb{T}_B\otimes_BA}{\times}\mathbb{T}_{A/K}\,.
$$
Note that $0'=\textit{fib}\,(\mathbb{T}_{A/K}\to\mathbb{T}_B\otimes_BA)\simeq\mathbb{T}_{A/B}$. 
\item[--] $~^{\mathbb{T}_{A/B}/}\MOD_{A/\mathbb{T}_{A/K}}\to~^{\mathbb{T}_{A/B}/}\!\LIE_{K/A}$, sending 
$\mathbb{T}_{A/B}\to N\to \mathbb{T}_{A/K}$ to 
$$
\mathbb{T}_{A/B}\underset{\textit{free}\,\mathbb{T}_{A/B}}{\coprod}\textit{free}\,N\,.
$$
\end{itemize}
One easily checks that $\mathfrak{L}(B[-n-1])\simeq \mathbb{T}_{A/B}\oplus \textit{free}\,(A[-n-1]\overset{0}{\to}\mathbb{T}_{A/K})$. 

We also define $\mathfrak{R}$ as the composition of two functors: 
\begin{itemize}
\item[--] first a functor $^{\mathbb{T}_{A/B}/}\!\LIE_{K/A}\to \MOD_{\mathbb{T}_{A/B}}$, where $\MOD_{\mathbb{T}_{A/B}}$ denotes the $\infty$-category of modules over the $(\kk,A)$-dg-Lie algebroid ${\mathbb{T}_{A/B}}$ (which are just modules over its universal enveloping algebra). Here the functor is defined as $\textit{cofib}\,(\mathbb{T}_{A/B}\to -)$, which naturally lands in $\mathbb{T}_{A/B}$-modules\footnote{To see this, one can use models and compute the cofiber in the case when the map out of $\mathbb{T}_{A/B}$ is injective. In this case the ${\mathbb{T}_{A/B}}$-module structure is rather classical (\textit{see e.g.} \cite[\S3.1]{calaque-rendiconti}).}.  
\item[--] then the functor $\MOD_{\mathbb{T}_{A/B}}\to \MOD_{B/\mathbb{T}_B}$, sending a ${\mathbb{T}_{A/B}}$-module $E$ to its Chevalley--Eilenberg complex $CE^\bullet(\mathbb{T}_{A/B},E)$ (\textit{see e.g.} \cite[\S2.2]{calaque-rendiconti} for a definition). Here the reader has to check that the $\mathbb{T}_{A/B}$-module structure on $\textit{cofib}\,(\mathbb{T}_{A/B}\to \mathbb{T}_{A/K})\simeq\mathbb{T}_{A/K}\otimes_BA$ is given by the action of $\mathbb{T}_{A/B}$ on $A$, so that 
$$
\mathfrak{R}(\mathbb{T}_{A/K})
\simeq CE^\bullet(\mathbb{T}_{A/B},\mathbb{T}_{B/K}\otimes_BA)
\simeq \mathbb{T}_{B/K}\otimes_B CE^\bullet(\mathbb{T}_{A/B})
\simeq \mathbb{T}_{B/K}\,.
$$
\end{itemize}
The details of the construction, and the proof that $\mathfrak{R}$ is right adjoint to $\mathfrak{L}$, will appear elsewhere. 
\end{proof}
\begin{Lem}\label{lem-gpasdenom}
The square of right adjoints 
\[
\xymatrix{ \CDGA_{K/B} & \ar[l]_-{CE^\bullet_{K/A}} \Big(^{\mathbb{T}_{A/B}/}\!\LIE_{K/A}\Big)^{op} \\ 
^{\mathbb{L}_{B/K}/}\!\MOD_B \ar[u]^-{d\mapsto B\underset{d}{\oplus}M[-1]} & \ar[u]_-{Z:=\mathfrak{L}^{op}} \ar[l]^-{(-)^\vee} (\MOD_{B/\mathbb{T}_{B/K}})^{op}
}
\]
commutes. 
\end{Lem}
\begin{proof}[Sketch of proof]
The functor $CE_{K/A}$ is right adjoint, and thus preserves limits. 
In particular, it sends push-outs in $\LIE_{K/A}$ to pull-backs in $\CDGA_{K/A}$. 
Hence we have that 
$$
CE^\bullet_{K/A}\mathfrak{L}^{op}(M)\simeq (A\oplus_{d'}(M')^\vee[-1])\underset{A\oplus_{can}\mathbb{L}_{A/B}[-1]}{\times}B\,,
$$
where $d':\mathbb{L}_{A/K}\to (M')^\vee$ is adjoint to the pullback of $M\otimes_BA\to \mathbb{T}_{B/K}\otimes_BA$, and  along $\mathbb{T}_{A/K}\to \mathbb{T}_{B/K}\otimes_BA$, and $can:\mathbb{L}_{A/K}\to\mathbb{L}_{A/B}$ is the canonical map. 
Hence $CE^\bullet_{K/A}\mathfrak{L}^{op}(M)\simeq B\oplus_dM^\vee[-1]$, where $d:\mathbb{L}_{B/K}\to M^\vee$ is adjoint to $M\to\mathbb{T}_{B/K}$. 
\end{proof}
Together, Lemma \ref{lem-en-retard} and Lemma \ref{lem-gpasdenom} imply that we have a weak morphism of weak Koszul duality contexts from $\CDGA_{K/B} \adjointHA\Big(^{\mathbb{T}_{A/B}/}\!\LIE_{K/A}\Big)^{op}$ to $^{\mathbb{L}_{B/K}/}\!\MOD_B\adjointHA(\MOD_{B/\mathbb{T}_{B/K}})^{op}$. 

The proof that this weak morphism is indeed a morphism will appear elsewhere. 

\subsection{Formal moduli problems under $X$}

Let $X$ be a formal derived prestack, and let $X\to Y$ be a formal thickening of $X$.  
\begin{Lem}
For every $Spec(A)\to Y$, $X_A:=X\times_YSpec(A)\to Spec(A)$ is a formal thickening of $X_A$. 
\end{Lem}
\begin{proof}
First of all observe that formal derived prestacks are stable by pullbacks, so that $X_A$ is a formal derived prestack. 
The only thing left to prove is that $X_A\to Spec(A)$ induces an equivalence on reduced prestacks, which is equivalent to require that it induces an equivalence on de Rham prestacks. The functor $(-)_{DR}$ being a right adjoint, we have that 
$(X_A)_{DR}\simeq X_{DR}\times_{Y_{DR}}Spec(A)_{DR}\simeq Spec(A)_{DR}$. 
\end{proof}
This allows us to build a presheaf of $\infty$-categories $A\mapsto \mathbf{Thick}(X_A)_{/Spec(A)}$ on $\mathbf{dAff}^{f.t.}$, and a commuting diagram of functors
$$
\xymatrix{
\mathbf{Thick}(X)_{/Y} \ar[r]\ar[d] & \underset{Spec(A)\to Y}{lim} \mathbf{Thick}(X_A)_{/Spec(A)} \ar[d] \\
^{X/}(\mathbf{dPrSt}_\kk^{f.t.})_{/Y} \ar[r] & \underset{Spec(A)\to Y}{lim} ~^{X_A/}(\mathbf{dPrSt}_\kk^{f.t.})_{/Spec(A)}  
}
$$
The bottom horizontal functor has a left adjoint 
$$
\underset{Spec(A)\to Y}{lim} ~^{X_A/}(\mathbf{dPrSt}_\kk^{f.t.})_{/Spec(A)} \longrightarrow~^{X/}(\mathbf{dPrSt}_\kk^{f.t.})_{/Y}
$$
that sends a diagram $\big(X_A\to Z(A)\to Spec(A)\big)_{Spec(A)\to Y}$ to its colimit 
$$
X\simeq\underset{Spec(A)\to Y}{colim} X_A\to \underset{Spec(A)\to Y}{colim} Z(A)\to \underset{Spec(A)\to Y}{colim} Spec(A)\simeq Y\,. 
$$
The counit of the adjunction 
$$
\underset{Spec(A)\to Y}{colim}\big(Z\underset{Y}{\times}Spec(A)\big)\to Z
$$
is an equivalence, and thus the bottom horizontal functor is fully faithful. 
Since the two vertical functors are also fully faithful, therefore so is the remaining horizontal one. 
The next Lemma tells us how far it is from being an equivalence. 
\begin{Lem}
The essential image of $\underset{Spec(A)\to Y}{lim} \mathbf{Thick}(X_A)_{/Spec(A)}$ lies in the essential image of $~^{X/}(\mathbf{dPrSt}_\kk^{f.t.})_{/Y}$. Moreover, it is included in the sub-$\infty$-category spanned by (images of) those $X\to Z\to Y$ inducing equivalences on associated reduced stacks and such that $Z$ is infinitesimally cohesive. 
\end{Lem}
\begin{proof}
The first part of the statement is obvious. 
Let $(Z_A)_{Spec(A)\to Y}$ be a diagram of formal derived stacks lying over the tautological diagram $\big(Spec(A)\big)_{Spec(A)\to Y}$, and such that $(Z_A)_{red}\to Spec(A_{red})$ is an equivalence. 
The functor $(-)_{red}$ being left adjoint, it commutes with colimits, and thus $Z_{red}\to Y_{red}$ is an equivalence (where 
$Z=\underset{Spec(A)\to Y}{colim}Z_A$). Moreover, filtered colimits do commute with finite limits and thus with pullbacks in particular, so that $Z$ is infinitesimally cohesive because every single $Z_A$ is.  
\end{proof}
If we could prove that derived prestacks in the essential image have a procotangent complex, then we would get that the functor $\mathbf{Thick}(X)_{/Y} \to \underset{Spec(A)\to Y}{lim} \mathbf{Thick}(X_A)_{/Spec(A)}$ is an equivalence. 
This is probably not true in full generality, but we think that it is for several examples: 
\begin{Conj}
If the functor $\bold{IndCoh}(Y)\to \bold{QCoh}(Y)$ is an equivalence, then the functor 
$$
\mathbf{Thick}(X)_{/Y}\to \underset{Spec(A)\to Y}{lim}\mathbf{Thick}(X_A)_{/Spec(A)}
$$
is an equivalence as well. 
\end{Conj}
Combining the above Lemma for $Y=X_{DR}$, together with the results of the previous Subsections, we get the following Theorem: 
\begin{Thm}\label{thm-final}
There is a fully faithful functor 
$$
\mathbf{Thick}(X) \to \LIE_{X_{DR}/X}:=\underset{Spec(A)\to X_{DR}}{lim}~^{\mathbb{T}_{Spec(A_{red})/X_A}/}\!\LIE_{A/A_{red}}\,.
$$
\end{Thm}
\begin{proof}
We simply have to prove that $\mathbf{Thick}(X)\simeq \mathbf{Thick}(X)_{/X_{DR}}$. 
This is true as $X_{DR}$ happens to be the terminal formal thickening of $X$. 
\end{proof}
We strongly believe that the functor appearing in the Theorem is actually an equivalence. This is in fact a consequence of our conjecture, as we know that $\bold{IndCoh}(X_{DR})\to \bold{QCoh}(X_{DR})$ is an equivalence (\textit{see} \cite[Proposition 2.4.4]{crystals}). 

\paragraph{Comparison with formal localization}~\\
Following Remark \ref{comparewithcptvv}, let us pretend that our results do hold in the more general graded mixed setting from \cite{CPTVV} (\textit{see} \S\ref{appendix-A.1} for a recollection), and consider the following situation: 
$\mathbf{K}=\mathbb{D}_{X_{DR}}$ is the mixed crystalline structure sheaf of a derived Artin stack $X$ and $\mathbf{A}=\mathcal{B}_X$ is the mixed principal parts of $X$ (after \cite[Definition 2.4.11]{CPTVV}). These are prestacks over $X_{DR}$ with values in graded mixed cdgas, and they are defined as follows: 
\begin{itemize}
\item[--] $\big(Spec(B)\to X_{DR}\big)\overset{\mathbf{K}}{\longmapsto}\mathbb{D}(B):=DR(B_{red}/B)\simeq CE^{\epsilon}(\mathbb{T}_{B{red}/B})$, with $|\mathbb{D}(B)|\simeq B$; 
\item[--] $\big(Spec(B)\to X_{DR}\big)\overset{\mathbf{A}}{\longmapsto}\mathbb{D}(X_B):=DR(X_B/B)$. 
\end{itemize}
We would thus have a prestack of $\infty$-categories $\LIE_{\mathbf{K}/\mathbf{A}}$ over $X_{DR}$, of which the global sections are denoted $\LIE^{f}_{X_{DR}/X}$. Presumably, this would be an ind-coherent version of $\LIE_{X_{DR}/X}$ defined above, which has also been defined as global sections of a prestack of $\infty$-categories. Let us explain why do believe so. 
\begin{itemize}
\item[--] First observe that, according to our results, we shall have an equivalence 
$$
\LIE_{\mathbb{D}(B)/\mathbb{D}(X_B)}\simeq~^{\mathbb{T}_{B_{red}/\mathbb{D}(X_B)}/}\!\LIE_{\mathbb{D}(B)/B_{red}}\,.
$$
\item[--] Then, the realization functor $|\cdot|$ and its left adjoint shall produce two adjunctions, related by a forgetful functor as follows: 
\[
\xymatrix@C=75pt@R=50pt{
~^{\mathbb{T}_{Spec(B_{red})/X_B}/}\!\LIE_{B/B_{red}} \ar@<-2pt>[d] \ar@<2pt>[r] &   \ar@<2pt>[l] ~^{\mathbb{T}_{B_{red}/\mathbb{D}(X_B)}/}\!\LIE_{\mathbb{D}(B)/B_{red}} \ar@<-2pt>[d] \\
\MOD_{B} \ar@<2pt>[r]  & \MOD_{\mathbb{D}(B)}  \ar@<2pt>[l] 
}
\]
\item[--] Finally, it has been proven in \cite[Proposition 2.2.8]{CPTVV} that, restricted to suitable subcategories of perfect modules, the adjunction 
$$
\mathbb{D}(X_B)\otimes_{B_{red}}-\colon\MOD_B \adjointHA\MOD_{\mathbb{D}(B)}\colon|\cdot|
$$
restricts to an equivalence. 
Moreover, it is conjectured that $\MOD_{\mathbb{D}(X_B)}$ is equivalent to $\mathbf{IndCoh}(X_B)$ (\textit{see e.g.} \cite[Remark 2.2.6]{CPTVV}). 	
\end{itemize}
Let us conclude with the following observation. Let $X\to Y$ be a formal thickening for which $Y$ is also a derived geometric stack. 
In particular we know that $X_{DR}\to Y_{DR}$ is an equivalence, and we thus get a sequence of graded mixed cdgas 
$$
\mathbf{K}\to\mathbf{B}:=\mathcal{B}_Y\to\mathbf{A}\,, 
$$
and hence graded mixed $(\mathbf{K},\mathbf{A})$-Lie algebroid $\mathbb{T}_{\mathbf{A}/\mathbf{B}}$. 
We therefore get a functor $\mathbf{ThickGeom}(X)\to \LIE^{f}_{X_{DR}/X}$, which shall fit into the following conjectural commuting diagram of functors: 
$$
\xymatrix{
\mathbf{ThickGeom}(X) \ar[r]\ar[d] & \mathbf{Thick}(X) \ar[d] \\
\LIE_{X_{DR}/X}^{f} \ar[r]\ar[d] & \LIE_{X_{DR}/X} \ar[d]\\
\mathbf{IndCoh}(X) \ar[r] & \mathbf{QCoh}(X)
}
$$
where $ThickGeom(X)$ denotes the subcategory of $Thick(X)$ consisting of geometric thickenings. 

\appendix

\section{Appendix}

In this appendix our aim is to compare the Chevalley-Eilenberg functor with the de Rham functor that is used in the formal localization procedure of \cite{CPTVV}. 
This is a very first step for a future comparison between the various approaches to formal geometry in the derived setting. 

\subsection{Recollection on graded mixed stuff}\label{appendix-A.1}

We fix a base cdga $A$ (not assuming anything at the moment), and start with Definitions. 
\begin{Def}~
\begin{itemize}
\item[--] A \textit{graded module} is a sequence $(M_n)_{n\in \mathbb{Z}}$ of $A$-modules. We refer to the index $n$ as the \textit{weight}. 
\item[--] We denote by $\mod_A^{gr}$ the category of graded modules. It comes equipped with the weight shifting functor $M\mapsto M(1)$, defined by $M(1)_n=M_{n+1}$. 
\item[--] A graded mixed module is a pair ($M,\epsilon)$, with $M$ a graded module and $\epsilon:M\to M[1](1)$ a morphism of graded modules such that $\epsilon[1](1)\circ\epsilon=0$ (which we will abusively write $\epsilon^2=0$). We call $\epsilon$ the \textit{mixed differential}. 
\item[--] We denote by $\mod_A^{\egr}$ the category of graded mixed modules. The weight shifting functor obviously lifts to $\mod_A^{\egr}$. 
\item[--] Both categories ($\mod_A^{gr}$ and $\mod_A^{\egr}$) are symmetric monoidal, so all categories of algebraic structures on modules that we considered for $\mod_A$ can be considered as well within them. We will add superscripts $gr$ and $\egr$ to the already introduce notation when emphasize which underlying symmetric monoidal category we are working with. For instance, $\cdga_{A}^{\egr}$ denotes the category of commutative monoids in $\mod_A^{\egr}$. 
\end{itemize}
\end{Def}
We equip $\mod_A^{gr}$ with the objectwise model structure from the one of $\mod_A$. We can transport this model structure along the forgetful functor 
$$
\mod_A^{\egr}\overset{(-)^\sharp}{\longrightarrow} \mod_A^{gr}
$$
and get a model structure\footnote{The injective one (though we will soon use the projective one, with which it is easier to compute mapping spaces, but for which the monoidal unit is not cofibrant). } on $\mod_A^{\egr}$ for which the monoidal unit $A$ is cofibrant. 
This model structure is nice enough so that all model category theoretic constructions based on $\mod_A$ still make sense on $\mod_A^{\egr}$, \textit{see e.g.} \cite[\S1.1, \S1.2 \& Appendix A]{CPTVV} and \cite[Variant 3.10]{Nuit1}. Thus we have for instance an $\infty$-category $\CDGA_A^{\egr}$ of graded mixed $A$-cdgas. 

Then recall from \cite[\S1.3]{CPTVV} that the Quillen adjunction 
$$
\mod_A \adjointHA \mod_A^{\egr}\colon\mathrm{HOM}_{\mod_A^{\egr}}(A,-)
$$
induces an adjunction 
$$
\MOD_A \adjointHA \MOD_A^{\egr}\colon|-|
$$
for which the right adjoint $|-|$, called \textit{(standard) realization}, is lax symmetric monoidal. This implies in particular that this realization functor is also well-defined on the other aforementioned categories carrying the superscript $\egr$. It has a very explicit description: 
\begin{Prop}[\cite{CPTVV}]\label{propstandardreal}
For a graded mixed module $(M,\epsilon)$, $|M|\simeq (\prod_{n\geq0}M_n,d_M+\epsilon)$, where $d_M:M\to M[1]$ is the (weight preserving) differential. 
\end{Prop}
\begin{proof}
Let us (temporarily) use the projective model structure on $\mod_A^\egr$. For this model structure, every object is fibrant. Moreover, one has the following explicit cofibrant replacement $\widetilde{A}$ for $A$, which is quasi-free (as a graded mixed $A$-module): 
\begin{itemize}
\item[--] let us consider the free $A$-module generated by $x_i$'s for $i\geq 0$ and $y_j$'s for $j\geq1$. 
\item[--] assign to $x_i$ cohomological degree $0$ and weight $i$, and to $y_j$ cohomological degree $1$ and weight $j$. 
\item[--] define $\epsilon(x_i)=y_{i+1}$. 
\item[--] modify the differential\footnote{Before doing that, our graded mixed module was free. } by imposing that $d(x_i)=y_i$ (by convention, $y_0=0$). 
\end{itemize}
Finally observe that $\mathrm{HOM}_{\MOD_A^{\egr}}(\tilde{A},M)=(\prod_{n\geq0}M_n,d_M+\epsilon)$. 
\end{proof}

\subsection{Graded mixed cdgas and dg-Lie algebroids}

Let $B\to A$ be a morphism of cdgas which, as usual, we assume to be cofibrant. 
\begin{Prop}
The Chevalley--Eilenberg functor $\mathrm{CE}^\bullet_{B/A}$ factors through a commuting diagram 
$$
\xymatrix{
& \CDGA_{B/A}^\egr \ar[rd]^-{|-|} & \\
\LIE_{B/A}^{op} \ar[rr]^-{\mathrm{CE}^\bullet_{B/A}} \ar[ru]^-{\mathrm{CE}^{\epsilon}_{B/A}} & & \CDGA_{B/A}
}
$$
\end{Prop}
\begin{proof}
For a dg-$(B,A)$-algebroid $L$ we set $\mathrm{CE}^\epsilon_{B/A}(L)_n:=S^n(L[1])^\vee$, together with 
$$
\epsilon(\omega)(\ell_0,\dots,\ell_n):=\pm\sum_{i=0}^n\rho(\ell_i)\big(\omega(\ell_0,\dots\hat{\ell_i},\dots,\ell_n)\big)\pm\sum_{0\leq i<j\leq n}\omega([\ell_i,\ell_j],\ell_0,\dots\hat{\ell_i},\dots,\hat{\ell_j},\dots,\ell_n)\,.
$$
We let the reader check that we indeed have that $|\mathrm{CE}^\epsilon_{B/A}(L)|\simeq \mathrm{CE}^\bullet_{B/A}(L)$ using Proposition\ref{propstandardreal}.\footnote{We simply observe that the discrepancy between $d_{\mathrm{CE}}$ and $\epsilon$ is resolved precisely thanks to the specific form of the realization functor. } 
\end{proof}

Hence, composing with the forgetful functor $\CDGA_{B/A}^{\egr}\to\CDGA_{B}^\egr$, we get a functor 
$$
\LIE_{B/A}^{op}\to \CDGA_{B}^{\egr}\,,
$$
which lands in the full sub-$\infty$-category consisting of graded mixed cdgas $C$ such that $C_0\simeq A$. 
It even factors through the $\infty$-category consisting of graded mixed cdgas $C$ together with a the data of an equivalence $C_0\simeq A$, which has an initial object, denoted $\bold{DR}(A/B)$ (an explicit description of which is provided in \cite{CPTVV}). 
Actually, the functor $\bold{DR}(-/B):\CDGA_B\to \CDGA_B^\egr$ is left adjoint to the functor sending a graded mixed $B$-cdga $C$ to its weight zero component $C_0$. 
\begin{Lem}
The morphism $\bold{DR}(A/B)\to CE^\epsilon_{B/A}(\mathbb{T}_{A/B})$ is an equivalence whenever $\mathbb{L}_{A/B}$ is coherent and (cohomologically) bounded above. 
\end{Lem}
\begin{proof}
Equivalences between graded mixed objects are checked componentwise. 
On weight components the map is given by 
$$
\bold{DR}(A/B)_n\simeq S^n_A(\mathbb{L}_{A/B}[-1])\to S^n_A(\mathbb{T}_{A/B}[1])^\vee\simeq CE^\epsilon_{B/A}(\mathbb{T}_{A/B})_n\,, 
$$
which is an equivalence if $\mathbb{L}_{A/B}$ is coherent and (cohomologically) bounded above. 
\end{proof}

If $B$ is a non-positively graded $\kk$-cdga ($\kk$ being a field, for simplicity) that is \textit{almost finitely presented}, then we know that $|\mathbf{DR}(B_{red}/B)|\simeq B$ after \cite[Lemma 2.2.4]{CPTVV}. 
Hence we get that 
$$
|CE_{\kk/B_{red}}^\epsilon(\mathbb{T}_{B_{red}/B})|\simeq|CE_{B/B_{red}}^\epsilon(\mathbb{T}_{B_{red}/B})|\simeq B\,,
$$
which we already knew for $B$ almost of finite type.

\bibliographystyle{plain}
\bibliography{ringard}

\let\thefootnote\relax\footnote{~\\
D.C.: IMAG, Univ Montpellier, CNRS, Montpellier, France \& Institut Universitaire de France \\
Email: \url{damien.calaque@umontpellier.fr}\\
~\\
J.G.: CNRS, I2M (Marseille) \& IH\'ES \\
Email: \url{jgrivaux@math.cnrs.fr}
}

\end{document}